\documentclass[11pt]{amsart}
\usepackage{amsthm}
\usepackage{enumitem}
\usepackage{mathtools}
\usepackage{moreenum}
\usepackage{amssymb,verbatim}
\usepackage{upgreek}
\usepackage{xspace,cmap}
\usepackage{csquotes}
\usepackage{stmaryrd} 
\usepackage[colorlinks,citecolor=blue,urlcolor=black,linkcolor=black]{hyperref}

\usepackage{mathtools}

\usepackage{pst-node}
\usepackage{tikz-cd}

\renewcommand{\restriction}{\mathbin\upharpoonright}    
\newtheorem*{theorem*}{Theorem}
\newtheorem*{maintheorem*}{Main Theorem}
\newtheorem*{corollary*}{Corollary}
\newtheorem*{definition*}{Definition}
\newtheorem{mainthm}{Main Theorem}

\newtheorem{theorem}{Theorem}[section]

\newtheorem{claim}{Claim}[theorem]
\newtheorem{subclaim}{Subclaim}[claim]
\newtheorem{lemma}[theorem]{Lemma}
\newtheorem{cor}[theorem]{Corollary}

\newtheorem{fact}[theorem]{Fact}

\theoremstyle{definition}
\newtheorem{example}[theorem]{Example}
\newtheorem{definition}[theorem]{Definition}
\newtheorem{notation}[theorem]{Notation}
\newtheorem{conv}[theorem]{Convention}

\newtheorem{setup}{Setup}

\theoremstyle{remark}
\newtheorem{remark}[theorem]{Remark}

\newcommand\cat[1]{{}^\curvearrowright #1}

\newcommand{\UEP}{\vec{\mathcal{U}}\textsf{-EP}}
\newcommand{\UBP}{\vec{\mathcal{U}}\textsf{-BP}}

\newcommand*\axiomfont[1]{\textsf{\textup{#1}}}
\newcommand\zfc{\axiomfont{ZFC}}
\newcommand\zf{\axiomfont{ZF}}
\newcommand\ac{\axiomfont{AC}}
\newcommand\ad{\axiomfont{AD}}
\newcommand\dc{\axiomfont{DC}}

\newcommand\sch{\axiomfont{SCH}}
\newcommand\psp{\axiomfont{PSP}}
\newcommand\bp{\axiomfont{BP}}

\DeclareMathOperator{\sky}{Sky}
\DeclareMathOperator{\con}{Con}

\DeclareMathOperator{\supp}{supp}
\DeclareMathOperator{\crit}{crit}

\DeclareMathOperator{\ob}{OB}

\DeclareMathOperator{\id}{id}

    \def\sle{\preceq}

    \def\sq{\sqsubseteq}
    
    \newcommand{\one}{\mathop{1\hskip-3pt {\rm l}}}

\makeatletter
\newcommand{\tpitchfork}{
  \vbox{
    \baselineskip\z@skip
    \lineskip-.52ex
    \lineskiplimit\maxdimen
    \m@th
    \ialign{##\crcr\hidewidth\smash{$-$}\hidewidth\crcr$\pitchfork$\crcr}
  }
}
\makeatother

\def\s{\subseteq}
\def\forces{\Vdash}

\DeclareMathOperator{\col}{Col}

\DeclareMathOperator{\ro}{\mathcal{B}}

\DeclareMathOperator{\cf}{cf}

\DeclareMathOperator{\ord}{Ord}

\renewcommand{\mid}{\mathrel{|}\allowbreak}
\newcommand{\lh}{\ell}
\newcommand{\mc}{\mathop{\mathrm{mc}}\nolimits}
\newcommand{\dom}{\mathop{\mathrm{dom}}\nolimits}

\newcommand{\Col}{\mathop{\mathrm{Col}}}

\title[Regularity properties at singulars cardinals]{The Baire and perfect set properties at singulars cardinals}
\author[Dimonte]{Vincenzo Dimonte}
\address[Dimonte]{Dipartimento di scienze matematiche, informatiche e fisiche, UniversitÃ  degli studi di Udine, Udine, Italy}
\email{vincenzo.dimonte@uniud.it}
\author[Poveda]{Alejandro Poveda}
\address[Poveda]{ Department of Mathematics and Center of Mathematical Sciences and Applications, Harvard University, Cambridge MA, 02138, USA}
\email{alejandro@cmsa.fas.harvard.edu}
\urladdr{www.alejandropovedaruzafa.com}
\author[Thei]{Sebastiano Thei}
\address[Thei]{Dipartimento di scienze matematiche, informatiche e fisiche, UniversitÃ  degli studi di Udine, Udine, Italy}
\email{thei91.seba@gmail.com}

\thanks{This  work is part of Thei's Ph.D. dissertation under   Dimonte at Udine University. The  abstract framework and the  main results of the paper were obtained by  Poveda and Thei  during a visit of the latter to Harvard University in the Spring of 2024.}

\begin{document}

\maketitle
\begin{abstract}
   We construct a model of $\zfc$ with a singular cardinal $\kappa$ such that every subset of $\kappa$ in $L(V_{\kappa+1})$ has both the $\kappa$-Perfect Set Property and the $\mathcal{\vec{U}}$-Baire Property. This is a higher analogue of Solovay's  result  for $L(\mathbb{R})$ \cite{solovay1970model}. We obtain this configuration starting with large-cardinal assumptions in the realm of supercompactness, thus improving 
  former theorems by  Cramer \cite{Cramer}, Shi \cite{Shi} and Woodin \cite{WoodinPartII}. 
\end{abstract}

\section{Introduction}
The \emph{Perfect Set Property} ($\psp$) and the \emph{Baire Property} ($\bp$) are cornerstone concepts in the study of the real line. Recall that a set $A\s \mathbb{R}$  has the $\psp$ if and only if $A$ is either countable or contains a non-empty \emph{perfect set}; namely, a closed set without isolated points in the topology of $\mathbb{R}$. Likewise, $A\s \mathbb{R}$ has the $\bp$ if there is an open set $U\s\mathbb{R}$ such that $A\triangle U$ is meager. 

\smallskip

The $\psp$ and the $\bp$ are paradigms of what are commonly referred to as \emph{regularity properties}; namely, properties indicatives of \emph{well-behaved} sets of reals. Classical theorems in Descriptive Set Theory, due to Luzin and Suslin, show that all \emph{Borel sets} (in fact, all \emph{analytic sets}) have both the $\psp$ and the $\bp$  as well as other regularity properties such as Lebesgue measurability. For more complex subsets of the real line 
the issue becomes sensitive to the consistency of strong axioms of infinity called \emph{large cardinals} \cite{Koel}. 

\smallskip

The \emph{Axiom of Choice} $(\ac)$, one of the backbone components of modern mathematics, yields  sets lacking the above-mentioned regularity properties. A natural competitor of $\ac$ is  Mycielski-Steinhaus \emph{Axiom of Determinacy} $(\ad)$ which asserts that every set of reals is determined. $\ad$ wipes out most of the pathologies inoculated by the non-constructive nature of $\ac$. For instance, under $\ad$ every set of reals has the $\psp$ (Davis, 1964), the $\bp$ (Banach-Mazur, 1957) and is Lebesgue measurable (Mycielski-Swierczkowski, 1964). Besides from these pleasant consequences, $\ad$ yields a coherent rich theory of the real line (cf. \cite{Kan, KoelWoo, Koel, Larsson}).

\smallskip

Since $\ac$ contradicts the Axiom of Determinacy, who is then a natural model for $\ad$?  $L(\mathbb{R})$, the smallest transitive model of \emph{Zermelo-Fraenkel set theory} ($\zf$)  containing the real numbers. But for $\ad$ to hold in $L(\mathbb{R})$ very large cardinals are required to exist. Indeed,  a deep theorem of Woodin says that $\ad^{L(\mathbb{R})}$ (i.e., the assertion $``L(\mathbb{R})\models \ad$'' holds) is equiconsistent with infinitely-many \emph{Woodin cardinals}. However mild large cardinal hypotheses suffice to grant the regularity properties inferred from $\ad^{L(\mathbb{R})}$, as proved by Solovay in his celebrated \emph{Annals of Mathe\-matics}' paper \cite{solovay1970model}. Solovay  shows that starting (merely) with an \emph{inaccessible cardinal} one can construct a model of $\zfc$ ($\zf$ together with the Axiom of Choice) where every set of reals $A\s\mathbb{R}$ in $L(\mathbb{R})$ has both the $\psp$ and $\bp$. The inaccessible cardinal is needed for every set of reals in $L(\mathbb{R})$  to have the $\psp$ \cite{solovay1970model} yet superfluous when it comes to the $\bp$, as showed by Shelah in \cite{ShelahBook}.

\smallskip

Curiously, the first proof of the consistency of  $\ad^{L(\mathbb{R})}$   
relied on 
a model beyond the confines of the continuum. The said model is $L(V_{\kappa+1})$ and it  was inspected under the lenses of axiom $I_0(\kappa)$, one of the strongest-known large-cardinal axioms (Woodin, 1984).\footnote{$L(V_{\kappa+1})$ stands for  the least transitive model of $\zf$ containing $V_{\kappa+1}$, the power set of the $\kappa$th-stage of the \emph{Von Neumann} hierarchy of sets (i.e., $V_\kappa$). Axiom $I_0(\kappa)$ postulates the existence of certain \emph{elementary embeddings} between $L(V_{\kappa+1})$ and itself. By results of Shelah and Kunen this requires $\kappa$ to be a singular cardinal of \emph{counatble cofinality}.}  
The investigation of regularity properties of sets $A\s\kappa$ in $ L(V_{\kappa+1})$ is regarded as a natural generalization of the classical study of regularity properties of sets $A\s\mathbb{R}$ in $L(\mathbb{R})$.  
Indeed, Woodin realized that $L(V_{\kappa+1})$ behaves under $I_0(\kappa)$ very much like $L(\mathbb{R})$ does under $\ad^{L(\mathbb{R})}$ \cite{Koel, Cramer}.  
Thus, $I_0(\kappa)$ provides the appropriate axiomatic framework to develop \emph{Generalized Descriptive Set theory} ($\textsf{GDST}$) 
in generalised Baire/Cantor spaces (e.g., ${}^\omega \kappa$ and ${}^\kappa 2$) at singular cardinals $\kappa$.

  Our knowledge of $\textsf{GDST}$ at regular uncountable cardinals has been substantially expanded in recent years (cf.  \cite{Friedman,  LuckeSchlicht, LuckeSchlichtMotto, Agostini}).  
On the contrary, similar investigations at  singular cardinals are barely existent. This is perhaps due to  the seemingly disconnection between (1) the descriptive-set-theoretic methods and (2) the techniques employed in the study of singular cardinals. This paper aims to bridge this conceptual gap.

 \smallskip

 Work by Woodin \cite[\S7]{WoodinPartII}  and his students,  Cramer \cite[\S5]{Cramer} and Shi \cite[\S4]{Shi}, have revealed that under axiom $I_0(\kappa)$ every set $A\s {}^\omega \kappa$ in $L(V_{\kappa+1})$ has the $\kappa$-$\psp$ (see \S\ref{sec: Generalized Descriptive Set Theory}). This is in clear concordance with the effects of $\ad^{L(\mathbb{R})}$ upon sets of reals in $L(\mathbb{R})$. Our first main result reads:
 \begin{mainthm}
    Suppose that $\kappa<\lambda$ are cardinals with $\kappa$ ${<}\lambda$-supercompact and $\lambda$  inaccessible. Then, there is a model of $\zfc$ where $\kappa$ is strong limit singular with $\cf(\kappa)=\omega$ and every set in $\mathcal{P}( {}^\omega \kappa)\cap L(V_{\kappa+1})$ has the $\kappa$-$\psp$. 
 \end{mainthm}

  The hypotheses employed in our theorem are noticeably weaker than those used by Cramer, Shi and Woodin. 
 Our result also improves former theorems by Dimonte \cite{Dimonte} and Dimonte, Iannella and Lücke \cite{DimonteIannella}. 

 \smallskip

 The second regularity property analyzed in this paper is the $\UBP$ (see \S\ref{sec: Generalized Descriptive Set Theory}), a generalization of the classical Baire Property introduced by Dimonte, Motto Ros and Shi in \cite{DimonteMottoShi} and recently studied in \cite{DimonteIannella}. 

\begin{mainthm}
     Suppose that $\Sigma=\langle \kappa_n\mid n<\omega\rangle$ is a strictly increasing sequence of ${<}\lambda$-supercompact cardinals and that $\lambda$ is an inaccessible cardinal above $\kappa:=\sup(\Sigma)$. Then, there is a model of $\zfc$ where $\kappa$ is a strong limit singular cardinal of countable cofinality and the following hold:
     \begin{enumerate}
         \item   Every set in $\mathcal{P}({}^\omega \kappa)\cap L(V_{\kappa+1})$ has the $\kappa$-$\psp$. 
         \item Every set in $\mathcal{P}(\prod_{n<\omega}\kappa_n)\cap L(V_{\kappa+1})$ has the $\vec{\mathcal{U}}$-$\bp$ (for certain  $\vec{\mathcal{U}}$).
     \end{enumerate}
\end{mainthm}
The above provides the singular-cardinal analogue of Solovay's theorem \cite{solovay1970model}. 
Also, it
enhances the extent of $\kappa$-$\psp$ and $\vec{\mathcal{U}}$-$\bp$ obtained in \cite{DimonteIannella}.

\smallskip

Our main theorems follow from a general 
framework developed by the last two  authors to get  \(\kappa\)-\(\psp\) and \(\UBP\) configurations at singular cardinals. This paradigm is underpinned by the  \emph{$\Sigma$-Prikry technology}, developed in \cite{PartI,PartII,PartIII, PartIV}. As we will show, our analysis generalize earlier work of Kafkoulis in \cite{Kafkoulis1994-KAFTCS}. Also,  
they lay the groundwork 
for  investigating regularity properties at small singular cardinals;  such as \(\aleph_\omega\).

\smallskip

The organization of the paper is as follows. In \S\ref{sec: prelimminaries} we set notations and provide relevant prelimminaries. In \S\ref{sec: Sigma Prikry Tool Box} the abstract framework is presented.  In \S\ref{sec: applications} our main consistency results are proved. Here we also develop a new forcing poset (the \emph{Diagonal Supercompact Extender Based Prikry forcing}) which has set-theoretic interests beyond the scope of the manuscipt.

\section{Preliminaries and notations}\label{sec: prelimminaries}
In this section we collect a few definitions and notations which should help the reader navigate the manuscript. When it comes to  \emph{set-theoretic forcing} our main reference is Kunen's \cite{Kunen}, whilst Kanamori's text  \cite{Kan} will be our preferred source when it comes to \emph{large-cardinal axioms}.
\subsection{Review on projections and complete embeddings}\label{sec: projections}
  Following the set-theoretic tradition our notation for (forcing) posets will be $\mathbb{P},  \mathbb{Q}$, etc. As customary members $p$ of a poset $\mathbb{P}$ will be referred as \emph{conditions}. Given two conditions $p,q\in \mathbb{P}$ we will write $``q\leq p$'' as a shorthand for \emph{$``q$ is stronger than $p$''}. We shall denote by $\mathbb{P}_{\downarrow p}$ the subposet of $\mathbb{P}$ whose universe is  $\{q\in \mathbb{P}\mid q\leq p\}$. Finally, \emph{$p$ and $q$ are compatible} if there is $r\in \mathbb{P}$ such that $r\leq p,q$. 

    \begin{definition}\label{def: projections}
        Let $\mathbb{P}$ and $\mathbb{Q}$ be forcing posets. \begin{itemize}
\item \cite{ForWoo} A \emph{weak projection} is a map $\pi\colon \mathbb{P}\rightarrow\mathbb{Q}$ such that: \begin{enumerate}
            \item For all $p,p'\in\mathbb{P}$, if $p\leq p'$ then $\pi(p)\leq \pi(p')$;
            \item For all $p\in\mathbb{P}$ there is $p^*\leq p$  such that for all $q\in \mathbb{Q}$ with $q\leq \pi(p^*)$ there is $p'\leq p$ such that $\pi(p')\leq q$.
        \end{enumerate} 
            \item A  weak projection $\pi\colon \mathbb{P}\rightarrow\mathbb{Q}$ is called a \emph{projection} if $\pi(\one_{\mathbb{P}})=\one_{\mathbb{Q}}$ and in Clause~(3) above the condition $p^*$ can be taken to be $p$.\footnote{Recall that the weakest condition of a forcing poset $\mathbb{P}$ is customarily denoted by $\one_{\mathbb{P}}$ -- or simply by $\one$ if there is no confusion. }
           \item A \emph{complete embedding} is a map $\sigma\colon\mathbb{Q}\rightarrow\mathbb{P}$ such that: 
            \begin{enumerate}
             \item For all $q,q'\in\mathbb{Q}$, if $q\leq q'$ then $\sigma(q)\leq \sigma(q')$;
             \item For all $q,q'\in\mathbb{Q}$, if $q$ and $q'$ are incompatible then so are $\sigma(q)$ and $\sigma(q')$; 
             \item For all $p\in\mathbb{P}$ there is $q\in\mathbb{Q}$ such that for all $q'\leq q$, $\sigma(q')$ and $p$ are compatible.
        \end{enumerate}
        \end{itemize} 
    \end{definition}
    \begin{fact}\label{fact: basics of weak projections}
        Let $\mathbb{P}$ and $\mathbb{Q}$ be forcing notions and $G$ be $\mathbb{P}$-generic. \begin{enumerate}
            \item If $\pi\colon \mathbb{P}\rightarrow\mathbb{Q}$ is weak a projection, then the upwards closure of the set $\pi``G$
            (namely, $\{q\in \mathbb{Q}\mid \exists p\in G\;\pi(p)\leq q\}$)
            is a $\mathbb{Q}$-generic filter (see {\cite[Proposition~2.8]{ForWoo}}).
            \item If $\sigma\colon\mathbb{Q}\rightarrow\mathbb{P}$ is a complete embedding then $\{q\in\mathbb{Q}\mid\sigma(q)\in G\}$ is $\mathbb{Q}$-generic (cf {\cite{Kunen}}).
        \end{enumerate}
    \end{fact}
    \begin{conv}
          We will identify $\pi``G$ with  its upwards closure.
    \end{conv}
 
    \begin{definition}[Quotient forcing]
    Given a projection $\pi\colon \mathbb{P}\rightarrow\mathbb{Q}$ between posets and a $\mathbb{Q}$-generic filter $H$ one defines the \emph{quotient forcing $\mathbb{P}/H$} as the subposet of $\mathbb{P}$ with universe $\{p\in \mathbb{P}\mid \pi(p)\in H\}.$
    \end{definition}
    \begin{fact}
        Every $\mathbb{P}/H$-generic $G$ (over $V[H]$) is $\mathbb{P}$-generic (over $V$).
    \end{fact}
        We are going to identify $\mathbb{P}$ with its isomorphic copy in $\ro(\mathbb{P})$, the \emph{regular open Boolean algebra} of $\mathbb{P}$ (minus its least element $\mathbf{0}$) \cite{Kunen}. According to this (usual) identification, $\mathbb{P}$ is dense in $\ro(\mathbb{P})$ and $\mathbb{P}$-names are also $\ro(\mathbb{P})$-names. In particular, if $\mathbb{Q}$ is a forcing and $\rho\colon\ro(\mathbb{P})\rightarrow\ro(\mathbb{Q})$ is a projection, then so is its restriction to $\mathbb{P}$. Moreover, if $G$ is $\ro(\mathbb{P})$-generic then $G\cap\mathbb{P}$ is $\mathbb{P}$-generic. Conversely, if $G$ is $\mathbb{P}$-generic, then its upwards closure in $\ro(\mathbb{P})$ $$\{b\in\ro(\mathbb{P})\mid \exists p\in G\ p\leq
        b\}$$ is $\ro(\mathbb{P})$-generic. Forcing with complete Boolean algebras has the   advantage that   every formula $\varphi$ admits weakest a condition $\llbracket\varphi\rrbracket\in \ro(\mathbb{P})$
        (\emph{the Boolean value of $\varphi$})  forcing $\varphi$; this is defined as $$\llbracket\varphi\rrbracket:=\bigvee\{p\in\mathbb{P}\mid p\Vdash_{\mathbb{P}}\varphi\}.$$
        
    If $\mathbb{P}$ and $\mathbb{Q}$ are forcings and $\dot g$ is a $\mathbb{P}$-name such that $\one \Vdash_{\mathbb{P}}``\dot g \text{ is } \check{\mathbb{Q}}\text{-generic}$'', there is a way to define a complete embedding $\sigma\colon \ro(\mathbb{Q})\rightarrow\ro(\mathbb{P})$, as well as a projection from $\mathbb{P}$ into (a cone of) $\ro(\mathbb{Q})$, using the $\mathbb{P}$-name $\dot g$. 
    \begin{definition}\label{def: projection/complete embedding induced by a name}
        Assume $\mathbb{P}$ and $\mathbb{Q}$ are forcings and $\dot g$ is a $\mathbb{P}$-name such that $\one \Vdash_{\mathbb{P}}``\dot g \text{ is } \check{\mathbb{Q}}\text{-generic}$''. The \emph{embedding induced by} $\dot g$ is the complete embedding $\sigma:\ro(\mathbb{Q})\rightarrow\ro(\mathbb{P})$ given by $\sigma(q)=\llbracket\check{q}\in\dot g\rrbracket,$ for all $q\in\mathbb{Q}$.
    \end{definition}
    \begin{fact}[Folklore]\label{fact: connection between projections and complete embeddings}
         Let  $\mathbb{P}$ and $\mathbb{Q}$ be forcings. The following are equivalent: \begin{enumerate}
             \item There is a projection $\pi\colon\mathbb{P}\rightarrow\ro(\mathbb{Q})$.
             \item There is a complete embedding $e\colon\ro(\mathbb{Q})\rightarrow\ro(\mathbb{P})$.
         \end{enumerate}
    \end{fact}
    \begin{fact}[{\cite[Lemma~2.1]{FrieHonz}}]\label{fact: projection to the rescue}
        Let $\mathbb{P}$ and $\mathbb{Q}$ be forcing notions. Suppose that $p$ is a condition in $\mathbb{P}$ and $\dot{g}$ is a $\mathbb{P}$-name such that $p\forces_{\mathbb{P}}\text{$``\dot{g}$ is $\mathbb{Q}$-generic''}$. 
        
       Let $\rho\colon \ro(\mathbb{P})_{\downarrow p}\rightarrow\ro(\mathbb{Q})$ be the extension to $\ro(\mathbb{P})_{\downarrow p}$ of the map defined as
        $$p'\leq_{\mathbb{P}} p\mapsto\rho(p'):=\bigwedge\{q\in\mathbb{Q}\mid p'\forces_{\mathbb{P}}\check{q}\in\dot{g}\}.$$ 
        Then $\rho\colon \ro(\mathbb{P})_{\downarrow p}\rightarrow\ro(\mathbb{Q})_{\downarrow\rho(p)}$ as well as its restriction  $\rho\restriction\mathbb{P}_{\downarrow p}$  
        are projections.

       \smallskip
        
        Moreover, for each $\mathbb{P}$-generic filter $G$ with $p\in G$, the upwards closure in $\ro(\mathbb{Q})$ of $\rho``G$ is
        $\{b\in\ro(\mathbb{Q})\mid \exists q\in \dot{g}_G\, (q\leq b)\}.\footnote{ {Here we are identifying the $\mathbb{P}$-generic $G$ with its induced $\ro(\mathbb{P})$-generic.}}$
    \end{fact}
    \begin{conv}\label{conv: convention about projections of boolean algebras}
         In the situation described in Fact \ref{fact: projection to the rescue}, we will refer to the map $\rho\colon \ro(\mathbb{P})_{\downarrow p}\rightarrow\ro(\mathbb{Q})_{\downarrow\rho(p)}$ as the \emph{projection induced by} $\dot{g}$ (and $p$). If $\sigma\colon \ro(\mathbb{P})\rightarrow\ro(\mathbb{Q})_{\downarrow\sigma(\one_{\mathbb{P}})}$ is a projection we will economize language and say instead that $``\sigma\colon \ro(\mathbb{P})\rightarrow\ro(\mathbb{Q})$ is a projection''. Note that $\sigma(\one_{\mathbb{P}})$ is possibly stronger than the least element in $\mathcal{B}(\mathbb{Q})$, thus our remark.
    \end{conv}

A technical caveat with certain Prikry-type forcings (the principal  posets of this paper) is the non-existence of projections, but only weak projections. However, given a weak projection $\pi\colon \mathbb{P}\rightarrow \mathbb{Q}$ one can combine Fact~\ref{fact: basics of weak projections}(1) with Fact~\ref{fact: projection to the rescue} to form the projection induced by $\pi``\dot{G}$, $\rho\colon \ro(\mathbb{P})\rightarrow\ro(\mathbb{Q})$. The next technical lemma shows that any commutative system of weak projections has a  {mirror (commutative) system of projections} between their regular open algebras. This observation will be important in \S\ref{sec: Sigma Prikry Tool Box}.  
    \begin{lemma}\label{from a weak system to a system}
       Every commutative diagram of weak projections induces a commutative diagram of projections between their regular open algebras:
    \begin{displaymath}
\begin{tikzcd}
\mathbb{P} \arrow{r}{\pi_1} \arrow{d}{\pi} & \mathbb{Q} \arrow{ld}{\pi_2}   \\
\mathbb{R} &
\end{tikzcd}
\Longrightarrow
\begin{tikzcd}
\ro(\mathbb{P}) \arrow{r}{\rho_1} \arrow{d}{\rho} & \ro(\mathbb{Q}) \arrow{ld}{\rho_2}   \\
\ro(\mathbb{R}) &
\end{tikzcd}.
\end{displaymath}  

    \end{lemma}
    \begin{proof}
        Let $\dot{G}_{\mathbb{P}}$ (resp. $\dot{G}_{\mathbb{Q}}$) be the canonical  name for the $\mathbb{P}$-generic (resp. $\mathbb{Q}$-generic) filter.  
        Define three generics as $\dot{g}_1:=\{\langle\check{q},p\rangle\mid p\in\mathbb{P}\ \wedge\ \pi_1(p)\leq q \}$, $\dot{g}:=\{\langle\check{r},p\rangle\mid p\in\mathbb{P}\ \wedge\ \pi(p)\leq r \}$ and $\dot{g}_2:=\{\langle\check{r},q\rangle\mid q\in\mathbb{Q}\ \wedge\ \pi_2(q)\leq r \}$.

       One can check that $$\one\forces_{\mathbb{P}}``\forall i\in\{1,2\}\;\dot{g}_i \text{ is the upwards closure of }\pi_i``\dot{G}_{\mathbb{P}}\text{"}$$ and $$\one\Vdash_{\mathbb{Q}}``\dot{g}_2 \text{ is the upwards closure of }\pi_2``\dot{G}_{\mathbb{Q}}\text{"}.$$ Appealing to Fact \ref{fact: projection to the rescue}, we let $\rho_1,\rho$ and $\rho_2$ be the projections respectively induced by the pairs $\langle \pi_1,g_1\rangle$, $\langle \pi,g_2\rangle$ and $\langle \pi_2, g_{1,2}\rangle$.

       \smallskip

       Let us  verify that the projections $\rho_i$ commute. For this it suffices to check that the they commute when restricted to $\mathbb{P}$. For each $p\in\mathbb{P}$ we have $$\rho_2(\rho_1(p))=\rho_2(\bigwedge\{q\in\mathbb{Q}\mid p\Vdash_{\mathbb{P}}\check{q}\in\dot{g}_1\})=\bigwedge\{\rho_2(q)\mid q\in\mathbb{Q}\ \wedge\  p\Vdash_{\mathbb{P}}\check{q}\in\dot{g}_1\}.$$ 
       We prove $\rho_2(\rho_1(p))\leq \rho(p)$ and leave the dedicated reader the verification of the other inequality.
    
            Note that everything amounts to verify that $$\text{$\rho_2(\rho_1(p))\leq r$ for all $r\in \mathbb{R}$ such that $p\Vdash_{\mathbb{P}}\check{r}\in\dot{g}$.}$$ 
            So fix a condition $r$ as above. First observe that $\pi_1(p)\Vdash_{\mathbb{Q}}\check{r}\in\dot{G}_{1,2}$. 
            
            To see this let be a $\mathbb{Q}$-generic filter $H$ with $\pi_1(p)\in H$, and let $G$ be $\mathbb{P}/H$-generic with $p\in G$. 
            As $p\Vdash_{\mathbb{P}}\check{r}\in\dot{g}$ and $p\in G$, we deduce that $$r\in(\dot{g})_G=\pi``G=\pi_2``\pi_1``G=\pi_2``G_{\mathbb{Q}}=(\dot{g}_2)_{G_{\mathbb{Q}}},$$ leading to $\pi_1(p)\Vdash_{\mathbb{Q}}\check{r}\in\dot{g}_2$. In particular, $\rho_2(\pi_1(p))\leq r$.
            
            On the other hand, $p$ forces that $\pi_1(p)\in\dot{g}_1$. 
            Recalling that $$\rho_2(\rho_1(p))=\bigwedge\{\rho_2(q)\mid q\in\mathbb{Q}\ \wedge\  p\Vdash_{\mathbb{P}}\check{q}\in\dot{g}_1\},$$ we finally conclude that $\rho_2(\rho_1(p))\leq\rho_2(\pi_1(p))\leq r$.
        \end{proof}

\subsection{A review of the $\Sigma$-Prikry framework.}\label{sec: Sigma Prikry framework}
Let us provide a succinct account of the $\Sigma$-Prikry framework developed in \cite{PartI, PartII,PartIII, PartIV}. 
Basic acquitance with this matter  will be required for \S\ref{sec: Sigma Prikry Tool Box}.

\begin{definition}\label{gradedposet} We say that a pair $(\mathbb P,\lh)$ is a \emph{graded poset}
whenever $\mathbb P$ is a poset and $\lh:\mathbb{P}\rightarrow\omega$ is a surjection such that for all $p\in \mathbb{P}$:
\begin{itemize}
\item  For every $q\le p$, $\lh(q)\geq\lh(p)$;
\item  There exists $q\le p$ with $\lh(q)=\lh(p)+1$.
\end{itemize}
\end{definition}
    If $(\mathbb{P}, \ell)$ is a graded poset $p\in\mathbb{P}$, and $n<\omega$, we denote $\mathbb{P}_n:=\{p\in\mathbb{P}\mid \ell(p)=n\}$, $\mathbb{P}_n^p:=\{q\in\mathbb{P}\mid q\leq p\ \wedge \  \ell(q)=\ell(p)+n\}$ and $\mathbb{P}_{\geq n}^p:=\{q\in \mathbb{P}\mid q\leq p \ \wedge \   \ell(q)\geq \ell(p)+n\}$. We will also write $q\leq^n p$ whenever $q\in \mathbb{P}^p_n$. Note that $\mathbb{P}_{\downarrow p}=\mathbb{P}^p_{\geq 0}$. If $q\in\mathbb{P}_0^p$, we write $q\leq^\ast p$ and call $q$ a \emph{direct extension} of $p$.

\begin{definition}[\cite{PartII}]\label{SigmaPrikry}
Suppose that $\mathbb P$ is a notion of forcing with a greatest element $\one$,
and that $\Sigma=\langle \kappa_n\mid n<\omega\rangle$ is a non-decreasing sequence of regular uncountable cardinals,
converging to some cardinal $\kappa$.
Suppose that $\lambda$ is a cardinal such that $\one\forces_{\mathbb P}\check\lambda=\kappa^+$.
For a function $\lh:\mathbb{P}\rightarrow\omega$
we say that $(\mathbb{P},\lh)$ is \emph{$\Sigma$-Prikry} iff all of the following hold:
\begin{enumerate}
\item\label{c4} $(\mathbb P,\lh)$ is a graded poset;
\item\label{c2} For $n<\omega$, $\mathbb P_n$ contains a dense subposet 
which is $\kappa_n$-directed-closed;
\item\label{c5} For all $p\in \mathbb{P}$, $n,m<\omega$ and $q\le^{n+m}p$, the set $\{r\le^n p\mid  q\le^m r\}$ contains a greatest element which we denote by $m(p,q)$. 
In the special case $m=0$, we shall write $w(p,q)$ rather than $0(p,q)$;\footnote{
Note that $w(p,q)$ is the weakest $n$-step extension of $p$  above $q$.}
\item\label{csize} For all $p\in \mathbb{P}$,
the set $W(p):=\{w(p,q)\mid q\le p\}$ has size $<\lambda$;
\item\label{itsaprojection} For all $p'\le p$ in $\mathbb{P}$, the map $w\colon W(p')\rightarrow W(p)$ defined by $q\mapsto w(p,q)$ is order-preserving;
\item\label{c6}(\emph{Complete Prikry Property})  Suppose that $U\s \mathbb{P}$ is a $0$-open set.\footnote{A set $U$ is \emph{$0$-open} whenever $r\in U$ iff $\mathbb{P}^r_0\s U$.}
Then, for all $p\in \mathbb{P}$ and $n<\omega$, there is $q\le^0 p$, such that, either $\mathbb{P}^{q}_n\cap U=\emptyset$ or $\mathbb{P}^{q}_n\s U$.
\end{enumerate}
\end{definition}

\begin{definition}[\cite{PartI}]\label{def: open coloring}
    Assume $(\mathbb{P},\ell)$ is a graded poset, $d\colon \mathbb{P} \rightarrow\theta$ is some coloring, with $\theta$ a nonzero
cardinal, and $H\subseteq P$. \begin{itemize}
    \item $d$ is $0$\emph{-open} if either $d(p)=0$ or $d(p)=d(q)$ for all $q\leq^*p$:  
    \item $H$ is a \emph{set of indiscernibles for} $d$ if for all $p,q\in H$, $$\text{$\ell(p)=\ell(q)\,\rightarrow\,d(p)=d(q)$.}$$
\end{itemize}
\end{definition}
\begin{lemma}[\cite{PartI}]\label{lem: indiscernibles}
    Assume $(\mathbb{P}, \ell)$ is  $\Sigma$-Prikry and let $p\in\mathbb{P}$.
    
    \begin{enumerate}
        \item\label{cones of indiscernibles} For  
every $0$-open coloring $d\colon \mathbb{P}\rightarrow n$, $n<\omega$,
there exists $q\leq^\ast p$ such that $\mathbb{P}_{\downarrow q}$ is a set
of indiscernibles for $d$;
\item\label{Prikry Property}\emph{(Prikry property)} If $\varphi$ is a formula in the forcing language of $\mathbb{P}$, there is a condition $q\leq^\ast p$ deciding $\varphi$;
\item\label{Strong Prikry Property}\emph{(Strong Prikry property)} If $D\subseteq\mathbb{P}$ is a dense open set there are $q\leq^\ast p$ and $n<\omega$ such that $\mathbb{P}^q_{\geq n}\subseteq D$.
    \end{enumerate}
\end{lemma}

\subsection{Generalized Descriptive Set Theory}\label{sec: Generalized Descriptive Set Theory}
Classical Descriptive Set \linebreak Theory \cite{Kec} studies the properties of definable subsets of Polish spaces, with the real line \(\mathbb{R}\) and the Baire space ${}^\omega\omega$ serving as paradigmatic examples. A number of results in the area show that simply definable sets  (e.g., \emph{Borel} or \emph{analytic} sets) do posses a rich canonical structure theory. 
Instead, the emerging field of 
\emph{Generalize Descriptive Set theory} stems from the study of definable objects beyond the continuum. Specifically, this line of research employs descriptive set-theoretic tools to study definable sets in higher function spaces like $^\kappa 2$ and $^\kappa \kappa$ (see \cite{Friedman}). It turns out that when $\kappa$ is an infinite cardinal with $\cf(\kappa)=\omega$, there are a number of results that parallel the findings of classical descriptive set theory. 

\smallskip 

\begin{setup} \label{setup: GDST}
 For the sake of a uniform presentation  we will assume that $\kappa$ is a  cardinal with $\cf(\kappa)=\omega$ and $\beth_\kappa=\kappa$. Our base axiomatic theory is $\textsf{ZF}+\textsf{DC}_\kappa$.

\end{setup}

Recall that $\dc_\kappa$ stands for Levi's \emph{Axiom of $\kappa$-Dependent Choice}; namely, this is the assertion that for every non-empty set $X$ and every total relation $R\s{}^{<\kappa} X\times X$ (i.e., for every $s\in{}^{<\kappa} X$ there is $x\in X$ such that $(s,x)\in R$) there is $f\colon \kappa\rightarrow X$ such that $(f\restriction \alpha, f(\alpha))\in R$ for all $\alpha<\kappa.$

\smallskip

The first demand of Setup~\ref{setup: GDST} ensures that our forthcoming topological spaces will admit a metric compatible with their topology;   the second demand ensures that $V_{\kappa+1}$ and $\mathcal{P}(\kappa)$ are interchangeable; finally, the theory $\textsf{ZF}+\textsf{DC}_\kappa$  will grant that the classical  proofs go smoothly even in this generalized case.

\begin{definition}[{\cite{DimonteMotto}}]
    A topological space $\mathcal{X}$ is called \emph{$\kappa$-Polish} in case it is homeomorphic to a completely metrizable space with weight $\kappa$.
\end{definition}

\begin{conv}
    Given a non-empty set ${X}$ and an infinite cardinal $\lambda$, the set ${}^\lambda X$ of all functions from $\lambda$ to ${X}$ will be equipped with the so-called \emph{bounded topology}; namely, the topology whose basic open neighborhoods are $N_{\xi, s}:=\{x\in{}^\lambda X\mid x\restriction\xi=s\}$, where $\xi<\lambda$ and $s\in{}^\xi X.$ In case $\lambda=\omega$ this coincides with the product  of the discrete topologies in $X.$
\end{conv}

The following accounts for what can be considered a canonical \(\kappa\)-Polish space, playing a role analogous to the Baire space \({}^\omega\omega\) in CDST.

\begin{example}[Examples of canonical $\kappa$-Polish spaces]\label{example: kappa polish spaces}  \hfill
    \begin{enumerate}
        \item The \emph{Generalized Baire Space} ${}^\omega\kappa$. 
        \item The \emph{Generalized Cantor Space} ${}^\kappa 2$ 
        (this is homeomorphic to $^\omega \kappa$).
        \item Let $\Sigma=\langle \kappa_n\mid n<\omega\rangle$ be an increasing sequence of regular cardinals with $\kappa=\sup(\Sigma)$. The space  $\textstyle \prod_{n<\omega}\kappa_n:= \{\text{$x\in{}^\omega\kappa\mid$  $\forall n<\omega\;x(n)\in \kappa_n$}\}$ (also denoted $C(\Sigma)$) is closed in ${}^\omega\kappa$ and thus  $\kappa$-Polish. It is homeomorphic to $^\omega \kappa$.   
        \item  $\mathcal{P}(\kappa)$ is $\kappa$-Polish  when endowed with the topology whose basic open sets are  $N_{\eta,a}:=\{b\in \mathcal{P}(\kappa)\mid b\cap \eta=a\}$ for $\eta<\kappa$ and $a\s\eta.$
    \end{enumerate}
    $\mathcal{P}(\kappa)$ and ${}^\kappa 2$ are homeomorphic -- this can be seen by sending each $a\in \mathcal{P}(\kappa)$ to its characteristic function. Similarly, there is a homeomorphism between ${}^\omega \kappa$ and a closed subspace of $\mathcal{P}(\kappa)$ given  by  $x\mapsto \{\prec n, x(n)\succ\mid n<\omega\}$.\footnote{Here $\prec\cdot,\cdot\succ$ denotes the G\"odel pairing function.}
\end{example}
The classical descriptive hierarchy on a Polish space extends naturally to the new setting of $\kappa$-Polish spaces. In particular, one has the following:
\begin{definition}[{\cite{DimonteMotto}}]
    Let $\mathcal{X}$ be a $\kappa$-Polish space. A set $A\s \mathcal{X}$ is:
    \begin{itemize}
        \item \emph{$\kappa$-Borel}: If $A$ is in the $\kappa$-algebra generated by the open sets of $\mathcal{X}$.
        \item \emph{$\kappa$-analytic}: If $A$ is a continuous image of another $\kappa$-Polish space $\mathcal{Y}$ (equivalently, $A$ is a continuous image of ${}^\omega\kappa$).
         \item \emph{$\kappa$-coanalytic}: If $A=\mathcal{X}\setminus B$ for a  $\kappa$-analytic set $B\s \mathcal{X}$.
         \item \emph{$\kappa$-projective}: If $A$ is in the family that contains the $\kappa$-Borel sets and it is closed under complements and continuous images.
    \end{itemize}
\end{definition}
\begin{remark}
	If $\mathcal{X}$ is a definable subset of $V_{\kappa+1}$ (e.g., all the spaces in Example~\ref{example: kappa polish spaces}), it is routine to prove that this hierarchy is equivalent to a definability hierarchy. In particular, the $\kappa$-projective subsets of $\mathcal{X}$ are exactly those that are definable in $V_{\kappa+1}$ with parameters in $\mathcal{X}$.
\end{remark}

One of the main tenets of classical descriptive set theory are regularity properties, a way to show that low-level sets in the descriptive hierarchy are well-behaved. Two paradigmatic examples are the \emph{Perfect Set Property} and the \emph{Baire Property} which we will discuss  in the context of $\kappa$-Polish spaces.

\medskip

\textbf{The Perfect Set Property:} Recall that a set $A\s \mathbb{R}$ has the \emph{Perfect Set Property} ($\psp$) if either $A$ is countable or there is a \emph{perfect set} $P\s A$; that is, $P$ is closed and does not have isolated points.  

A map $\iota\colon \mathcal{X}\rightarrow\mathcal{Y}$ between topological spaces is called an \emph{embedding} if it is an homeomorphism between $\mathcal{X}$ and $\mathrm{ran}(\iota)$. A well-known theorem in Classical Descriptive Set Theory says that a set $A$ contains a perfect set if and only if it contains a Cantor set, i.e., a closed set homeomorphic to ${}^\omega 2$ (see \cite[Theorem 6.2]{Kec}). So $A\s \mathbb{R}$ has the $\psp$ iff either $A$ is countable or there is an embedding $\iota\colon{}^{\omega}2\rightarrow A$ closed-in-$\mathbb{R}$. This second equivalent definition is the one that is usually generalized to higher cardinalities:

\begin{definition}[$\kappa$-Perfect Set Property]
    Let $\mathcal{X}$ be $\kappa$-Polish. A set $A\s \mathcal{X}$ has the $\kappa$\emph{-Perfect Set Property} (briefly, $\kappa$\emph{-\textsf{PSP}}) if either $|A|\leq\kappa$, or there exists an embedding from $^\kappa 2$ to $A$ closed-in-$X$.
\end{definition}

The following has been proved in \cite{DimonteMotto}:
\begin{fact} \label{fact PSP}
    Let $\mathcal{X}$ be $\kappa$-Polish and a set $A\s \mathcal{X}$. The following are equivalent: 
    \begin{enumerate}
        \item \label{fact PSP-1} $A$ has the $\kappa$-$\psp$;
        \item \label{fact PSP-2} $|A|\leq \kappa$ or there is an embedding $\iota\colon \mathcal{C}\rightarrow \mathcal{X}$ with $\mathrm{ran}(\iota)\s A$, for $\mathcal{C}$ some (all) of the canonical $\kappa$-Polish spaces of Example~\ref{example: kappa polish spaces}. 
        \item \label{fact PSP-3} $|A|\leq \kappa$ or there is a continuous injection $\iota\colon \mathcal{C}\rightarrow \mathcal{X}$ with $\mathrm{ran}(\iota)\s A$, for $\mathcal{C}$ some (all) of the canonical $\kappa$-Polish spaces of Example~\ref{example: kappa polish spaces}. 
        \item \label{fact PSP-4} $A$ contains a $\kappa$-perfect set $P$, i.e., a set that is closed and such that if $x\in P$, for all open neighborhoods $U$ of $x$ we have $|P\cap U|\geq\kappa$.
    \end{enumerate}
\end{fact}

Note that in the second item we are not asking for $\mathrm{ran}(\iota)$ to be closed.

\begin{fact}[{\cite{DimonteMotto}}]
  $\kappa$-analytic sets in a $\kappa$-Polish space $\mathcal{X}$ have the $\kappa$-$\psp$.
\end{fact}

If $\mathsf{AC}$ holds, there are always sets without the $\kappa$-$\psp$: Since the set of embeddings between ${}^{\omega}\kappa$ and $\mathcal{P}(\kappa)$ has cardinality $2^\kappa$, one can build a subset of $\mathcal{P}(\kappa)$ without the $\kappa$-$\psp$ via a standard recursive construction (see \cite[Proposition~11.4]{Kan}). By results of Shelah \cite{ShelahChoice}) if $\kappa$ is a strong limit singular of uncountable cofinality then $L(V_{\kappa+1})\models \ac$ and as a result there are  sets $X\in \mathcal{P}(\kappa)\cap L(V_{\kappa+1})$ without the $\kappa$-$\psp$.
In contrast, a theorem of Cramer \cite[Theorem 5.1]{Cramer} says that under axiom $I_0(\kappa)$ every set $X\in \mathcal{P}(\kappa)$ in $L(V_{\kappa+1})$ has the $\kappa$-$\psp.$ Similar results were also obtained by Shi \cite{Shi} and Woodin \cite{WoodinPartII}. In \S\ref{sec: applications} we will prove the consistency of this very configuration starting from much weaker large-cardinal hypothesis. Under axiom $I_0(\kappa)$, however, $L(V_{\kappa+1})$ manifests other desirable properties -- e.g., $(\kappa^+)^V$ is measurable in $L(V_{\kappa+1})$.

\smallskip

Many proofs in this paper take place in $L(V_{\kappa+1})$, the smallest model of $\zf$ that contains $V_{\kappa+1}$. Since $\beth_\kappa=\kappa$ is equivalent to $|V_\kappa|=\kappa$, 
the bijection that witnesses it is in $V_{\kappa+1}$ and therefore $L(V_{\kappa+1})\vDash\beth_\kappa=\kappa$. Moreover, $L(\mathcal{P}(\kappa))\vDash\mathsf{DC}_\kappa$ in the same way that $L(\mathbb{R})\vDash\mathsf{DC}$ (\cite[Lemma 4.10]{DimonteRankIntoRank}). Therefore,  $L(V_{\kappa+1})$ is a model for the configuration described in Setup~\ref{setup: GDST}.

It is standard to prove that if $\iota:{}^\omega\kappa\to{}^\omega\kappa$ is continuous, then it is uniquely defined by its behavior on ${}^{<\omega}\kappa$, just like every continuous function in $\mathbb{R}$ is defined by its behavior on $\mathbb{Q}$. Therefore $\iota$ is codeable by an element in $V_{\kappa+1}$. The same reasoning works if we swap one or both instances of ${}^\omega\kappa$ with any space in Example~\ref{example: kappa polish spaces}. As a result  the $\kappa$-$\psp$ is absolute between models that have the same $\mathcal{P}(\kappa)$, thus it is absolute between $L(V_{\kappa+1})$ and $V$.

\medskip

\textbf{The Baire Property:} Recall that a topological space $\mathcal{X}$ is a  \emph{Baire space} if every non-empty open subset of $\mathcal{X}$ is not meager. A set  $A\s \mathcal{X}$ has the \emph{Baire Property} ($\bp$, for short) in case there is an open set $O\s \mathcal{X}$ such that $A\triangle O$ is meager. In a similar fashion \cite{DimonteMottoShi}, one defines the notions of a $\emph{$\kappa$-Baire space}$ and $\kappa$-\emph{Baire Property} ($\kappa$-$\bp$) by replacing ``meager'' by $``\kappa$-meager'' (i.e., a union of $\kappa$-many nowhere dense sets).

\smallskip

While the notion of a $\kappa$-Baire space is meaningful for general topological spaces it does not fit well with the canonical $\kappa$-Polish topologies. Specifically, $C(\Sigma)$ (the natural extension of the Baire space ${}^\omega\omega$) is \textbf{not} a $\kappa$-Baire in its $\kappa$-Polish topology for it is the union of $\omega_1$-many nowhere dense sets \cite{DimonteMottoShi}.\footnote{For each $\alpha<\omega_1$, set $U_\alpha=\{x\in C(\Sigma)\mid\exists n<\omega\, x(n)=\alpha\}$ and note that $C(\Sigma)=\bigcup_{n<\omega_1}(C(\Sigma)\setminus U_\alpha)$. Since the $U_\alpha$'s are dense open, the conclusion follows.} To bypass this issue another route is outlined in \cite{DimonteMottoShi}. Namely, the authors consider $C(\Sigma)$ with a topology accessory to its natural product topology --  the $\vec{\mathcal{U}}$-\emph{Ellentuck-Prikry} topology ($\UEP$). While we will insist on $C(\Sigma)$ retaining its $\kappa$-Polish topology (in fact, $C(\Sigma)$ is not $\kappa$-Polish in the $\UEP$ topology) it is in the $\UEP$ topology where the $\kappa$-$\bp$ for subsets  $A\s C(\Sigma)$ will be formulated. This is because $C(\Sigma)$ is a $\kappa$-Baire space with respect to the accessory $\UEP$ topology \cite{DimonteMottoShi}.

\smallskip

The choice of the $\UEP$ topology is inspired by the fact that in Classical Descriptive Set Theory the product topology of the Baire space ${}^\omega\omega$ is homeomorphic to the topology of maximal filters on Cohen forcing.\footnote{This is the topology generated by the open sets  $N_p=\{F\in \beta\omega\mid \supp(p)\in F\}$ where $\supp(p)$ is the support of a condition $p\colon \omega\rightarrow 2$ in Cohen forcing $\mathrm{Add}(\omega,1).$}
The analogy in the singular case is provided by Magidor's \emph{Diagonal Prikry forcing}.

\smallskip

Hereafter we assume that $\vec{\mathcal{U}}=\langle \mathcal{U}_n\mid n<\omega\rangle$ are normal (non-principal) ultrafilters over each member $\kappa_n$ of an increasing sequence $\Sigma=\langle \kappa_n\mid n<\omega\rangle$. We shall set $\kappa:=\sup(\Sigma)$. Recall $C(\Sigma)$ denotes the $\kappa$-Polish space $\prod_{n<\omega}\kappa_n.$

\begin{definition}[Diagonal Prikry forcing (Magidor)]\label{def: diagonal prikry}
A condition in the \emph{Diagonal Prikry forcing with $\vec{\mathcal{U}}$} (in symbols, $\mathbb{P}(\vec{\mathcal{U}})$) is a sequence $$p=\langle \alpha^p_0,\dots, \alpha^p_{\ell(p)-1}, A^p_{\ell(p)}, A^p_{\ell(p)+1},\dots \rangle$$ where
$s^p=\langle \alpha^p_0,\dots, \alpha^p_{\ell(p)-1}\rangle\in \prod_{n<\ell(p)}\kappa_n$ is strictly increasing,  $A^p_n$ belongs to $\mathcal{U}_n$ for $n\geq \ell(p)$ and every $\beta\in A_{n+1}^p$ is bigger than $\kappa_n$. 

Given conditions $p,q\in \mathbb{P}(\mathcal{U})$ write $p\leq q$ if $s^p$ end-extends $s^q$ (in symbols,  $s^p\sq s^q$), $s^p(n)\in A^q_n$ for $n\in [\ell(q),\ell(p))$ and $A^p_n\s A^q_n$ for all $n\geq \ell(p).$
\end{definition}
\begin{definition}[The $\UEP$ topology]\label{def: ellentuck prikry}\hfill
   \begin{enumerate}
    \item For each $x\in \prod_{n<\omega}\kappa_n$ we consider the  filter $\mathcal{F}_x\s \mathbb{P}({\Vec{\mathcal{U}}})$ defined as
    $$\mathcal{F}_x:=\{p\in \mathbb{P}({\Vec{\mathcal{U}}})\mid s^p\sqsubseteq x\,\wedge\,\forall n\geq \ell(p)\, (x(n)\in A^p_n)\}$$
and we say that \emph{$x$ is $\mathbb{P}(\vec{\mathcal{U}})$-generic}  if $\mathcal{F}_x$ is a $\mathbb{P}(\vec{\mathcal{U}})$-generic filter.  
       \item For each condition $p\in \mathbb{P}({\Vec{\mathcal{U}}})$, define $$\textstyle N_p:=\{x\in\prod_{n<\omega}\kappa_n\mid p\in\mathcal{F}_x\}.$$ 
       \item  The \emph{$\Vec{\mathcal{U}}$-Ellentuck-Prikry}  ($\UEP$) topology $\mathcal{T}_{\UEP}$ is the topology which has  $\{N_p\mid p\in\mathbb{P}({\Vec{\mathcal{U}}})\}$ as  basic open sets. 
   \end{enumerate}
\end{definition}

\begin{definition}[$\Vec{\mathcal{U}}$-Baire Property]\label{def: kappa baire}
   A set $A\subseteq C(\Sigma)$ has the \emph{$\Vec{\mathcal{U}}$-Baire Property} (in short, $\UBP$) if it has the $\kappa$-$\bp$ as a subset of the topological space $(C(\Sigma), \mathcal{T}_{\UEP})$. More verbosely, if  there is a $\UEP$-open set $O\subseteq C(\Sigma)$ such that $A\Delta O$ is $\kappa$-meager in the $\UEP$ topology. 
\end{definition}

This definition generalizes many classical results, like the Mycielski and the Kuratowski-Ulam theorems (\cite{Kec}), and it is a  regularity property:

\begin{fact}\hfill
\begin{enumerate}
      \item $($\cite{DimonteMottoShi}$)$  $(C(\Sigma), \mathcal{T}_{\UEP})$ is a $\kappa$-Baire space.  
    \item $($\cite{DimonteMottoShi}$)$  All the $\kappa$-analytic subsets of $\prod_{n<\omega}\kappa_n$ have the $\Vec{\mathcal{U}}$-$\bp$.
    \item $(${\cite{DimonteIannella}}$)$  There exists a subset of $\prod_{n<\omega}\kappa_n$ without the $\vec{\mathcal{U}}$-$\bp$.
\end{enumerate}
\end{fact}

There is a natural way to construct a $\kappa$-comeager set in the  $\UEP$ topology. This is again inspired by the classical example of the Baire space ${}^\omega \omega$ and Cohen forcing $\mathrm{Add}(\omega,1)$. Specifically, if $M$ is an inner model of $\zfc$ and the cardinality of its continuum $(2^{\aleph_0})^M$ is countable (in $V$) then the set of Cohen reals over $M$ is comeager in ${}^\omega\omega$ (see e.g., \cite[Lemma~8.10]{Schindler}).

In the singular case the construction of a $\kappa$-comeager subset of the product space $C(\Sigma)$ appeals to the \emph{Mathias criterion of genericty} of $\mathbb{P}(\vec{\mathcal{U}})$. 

\begin{fact}[Mathias criterion of genericity]
    Let $M\s N$ be inner models of $\zfc$ with $\mathbb{P}(\vec{\mathcal{U}})\in M$. A sequence $x\in (\prod_{n<\omega}\kappa_n)^N$ is $\mathbb{P}(\vec{\mathcal{U}})$-generic over $M$ if and only if for each sequence $\langle A_n\mid n<\omega\rangle\in (\prod_{n<\omega}\mathcal{U}_n)^M$ there is $n^*<\omega$ such that $x(n)\in A_n$ for all $n\geq n^*$.
\end{fact}

The next lemma will be crucial in our verification of the $\vec{\mathcal{U}}$-$\bp$ in \S\ref{sec: applications}:
\begin{lemma}\label{lemma: C is comeager}
    Let $\mathbb{P}$ be a forcing such that $V[G]\models``|\mathcal{P}(\kappa)^{{V}}|<(\kappa^+)$" for certain $\mathbb{P}$-generic filter $G$. Working in $V[G]$ define $$\textstyle C:=\{x\in C(\Sigma)^{V[G]}\mid x \text{ is }\mathbb{P}({\Vec{\mathcal{U}}})\text{-generic over } V\}.$$ Then, $C$ is $\kappa$-comeager in the $\UEP$-topology.
\end{lemma}
    \begin{proof}
         
         For each $\Vec{A}=\langle A_n\mid n<\omega\rangle\in(\prod_{n<\omega}\mathcal{U}_n)^V$, set $$\textstyle O({\Vec{A}}):=\{x\in C(\Sigma)^{V}\mid \exists n<\omega\,\forall m\geq n\ (x_n\in A_n)\}.$$ 
         The set $O(\vec{A})$ is dense and open in the $\UEP$ topology (as computed in $V$). 
         
         Moreover, appealing to the Mathias criterion for genericity, we have that  $$\textstyle \bigcap_{\Vec{A}\in(\prod_{n<\omega}\mathcal{U}_n)^V} O({\Vec{A}})\subseteq C.$$ 
        Our assumption, $``V[G]\models |\mathcal{P}(\kappa)^V|<(\kappa^+)$'' implies that  
         $(\prod_{n<\omega}\mathcal{U}_n)^V$ has cardinality $\kappa$ in $V[G]$ and from this we infer that  
         $C$ contains the intersection of $\kappa$-many $\UEP$-open dense subsets of $C(\Sigma)^{V[G]}$, which amounts to saying that $C$ is $\kappa$-comeager.
    \end{proof}

\section{The $\Sigma$-Prikry tool box}\label{sec: Sigma Prikry Tool Box}
This section presents our abstract $\Sigma$-Prikry-based framework to get the consistency of the  $\kappa$-$\psp$ and $\UBP$  for a singular strong limit cardinal $\kappa$ of countable cofinality.\footnote{Results of Shelah \cite{ShelahChoice} reveal that the configurations obtained here cannot be generalized to singular cardinals of uncountable cofinality.  } 
The consistency of these configurations will be established in \S\ref{sec: merimovich} and \S\ref{sec: alternative AIM} as an application of this technology.

\subsection{Two general lemmas}\label{s: general lemmas} Let us begin proving two general lemmas:
\begin{lemma}\label{functions in the intermediate model}
Let $\kappa$ be an infinite (possibly singular) cardinal,  $\mathbb{P}$ a poset not adding bounded subsets of  $\kappa$, $\tau$  a $\mathbb{P}$-name for a function in $^\omega\kappa$ and $p$ a condition in $\mathbb{P}$ such that $p\Vdash_{\mathbb{P}}\tau\colon \Check{\omega}\rightarrow\Check{\kappa}$. 

Consider the set of possible decisions about $\tau$ made by conditions below $p$:
    $$\Delta(p):=\{\xi<\kappa\mid\exists q\in\mathbb{P}\big(\exists m<\omega(q\leq p\,\wedge\,q\Vdash_{\mathbb{P}}\tau(\Check{m})=\Check{\xi})\big)\}.$$
 If $\Delta(p)$ has cardinality ${<}\kappa$ then $p\forces_{\mathbb{P}}\tau\in V$. 
\end{lemma}

\begin{proof}
Let $B_p=\{\langle m,\xi\rangle\in\omega\times\kappa\mid\exists q\in\mathbb{P}(q\leq p\wedge q\Vdash_{\mathbb{P}}\tau(\Check{m})=\Check{\xi})\}$. As $|\Delta(p)|<\kappa$, it follows that $|B_p|<\kappa$. Pick $\eta<\kappa$ such that $|B_p|=\eta$ and let $f\colon B_p\to \eta$ be a bijection in $V$ witnessing it. Let $h$ be $\mathbb{P}$-generic  with $p\in h$. In $V[h]$, $\tau_h\subseteq B_p$ and $f``\tau_h\subseteq\eta$. So $f``\tau_h$ is a bounded subset of $\kappa$ in $V[h]$. But $\mathbb{P}$ does not add bounded subsets to $\kappa$ 
so $f``\tau_h\in V$. Finally note that $\tau_h$ is definable from $f$ and $f``\tau_h$ as $\tau_h=\{\langle m,\xi\rangle\in\omega\times\kappa\mid f(\langle m,\xi\rangle)\in f``\tau_h\}$. 
\end{proof}

\begin{lemma}\label{lemma: transitive closure}
Suppose that $\mathbb{P}$ is a forcing poset and that $\kappa$ and $\lambda$  are cardinals such that $p\forces_{\mathbb{P}}`` (\kappa^+)^{V[\dot{G}]}=\lambda$'' for a condition $p\in\mathbb{P}.$ 
Then, $$p\forces_{\mathbb{P}}(H_\lambda)^{V[\dot{G}]}\subseteq L(V[\dot{G}]_{\kappa+1}).$$ In particular, if there is a weak projection $\pi\colon \mathbb{P}\rightarrow \mathbb{Q}$,  $G$ is a $\mathbb{P}$-generic containing $p$ and $\lambda$ is inaccessible in $V[\pi``G]$,
$(^\omega\kappa)^{V[\pi``G]}\in L(V[G]_{\kappa+1})$. 
\end{lemma}
\begin{proof}
Fix $G$ a $\mathbb{P}$-generic filter with  $p\in G.$ Let $x\in V[G]$ be such that $|\mathrm{tcl}(\{x\})|^{V[G]}<\lambda$. As 
 $(\kappa^+)^{V[G]}=\lambda$ the transitive closure of $\{x\}$  has cardinality $\eta$ in $V[G]$ 
for some $\eta\leq\kappa$. Let $f\colon \eta\rightarrow\mathrm{tcl}(\{x\})$ be a bijection in $V[G]$ and define the relation $\triangleleft\subseteq\eta\times\eta$ as $\text{$\xi\triangleleft\delta:\Longleftrightarrow f(\xi)\in f(\delta)$.}$ Note that $\triangleleft$ can be coded as a subset of $\eta\leq\kappa$ (e.g., using G\"odel's pairing function) and thus $(\eta,\triangleleft)\in L(V[G]_{\kappa+1})$. If we compute  the transitive collapse of $(\eta,\triangleleft)$ (in $ L(V[G]_{\kappa+1})$) it has to take the form $(\mathrm{tcl}(\{x\}),\in)$ and thus,  by transitivity, $x\in L(V[G]_{\kappa+1})$. All of these show $(H_\lambda)^{V[G]}\s L(V[G]_{\kappa+1}).$

For the second assertion note that $V[G]\models``(|\mathrm{tcl}((^\omega\kappa)^{V[\pi``G]})|<\lambda)$'' by inaccessibility of $\lambda$  in the intermediate model $V[\pi``G]$. Combining this with our previous conclusion, $(^\omega\kappa)^{V[\pi``G]}\in (H_\lambda)^{V[\pi``G]}\subseteq L(V[G]_{\kappa+1}).$

\end{proof}

The forthcoming sections (i.e., \S\ref{s: sigma prikry and goodness}, \S\ref{sec: direct systems of posets}, \S\ref{sec: sigma prikry and kappaperfect}, \S\ref{sec: Sigma Prikry and BP}) will instead regard $\Sigma$-Prikry forcings. Our blanket assumptions therein will be as follows:

\begin{setup}\label{setup kappapsp}

  $\Sigma=\langle \kappa_n\mid n<\omega\rangle$ is a fixed non-decreasing (possibly constant) sequence of regular uncountable cardinals with  $\kappa:=\sup(\Sigma)$ and $\lambda>\kappa$ is an inaccessible. Unless otherwise stated, posets will be assumed to be $\Sigma$-Prikry.

\end{setup}

\subsection{The interpolation lemma}\label{s: sigma prikry and goodness}
The goal of this section is to prove the \emph{Interpolation Lemma} (Lemma~\ref{lemma: interpolation}) which is one of the main technical devices in the proof of the $\kappa$-$\psp$ and the $\UBP$. Before that we prove another  lemma saying that if $\kappa$ is singular in a generic extension $V[G]$ by a  $\Sigma$-Prikry forcing then every function $f\colon \kappa\rightarrow E$ in $V[G]$ has its \emph{traces} in $V$.

\begin{lemma}\label{lemma: goodness}
     Let $G$ a $\mathbb{P}$-generic filter and suppose that in $V[G]$,  $\cf(\kappa)=\omega$ and $f\colon \kappa\rightarrow E$ is a function with $E\in V$. Then there is $\langle B_n\mid n<\omega\rangle\in V[G]$ consisting of bounded subsets of $\kappa$ such that $\bigcup_{n<\omega}B_n=\kappa$ and $f\restriction B_n\in V$. 
\end{lemma}
\begin{proof}
Fix $\langle \nu_n\mid n<\omega\rangle\in \prod_{n<\omega}\kappa_n$ an increasing cofinal sequence in $\kappa$. In the proof we define auxiliary sequences $\langle r_{n,i}\mid i,n<\omega\rangle$ and $\langle C_{n,i}\mid i,n<\omega\rangle$ such that the following properties hold for each $n<\omega:$
\begin{enumerate}
    \item[$(\alpha)$] $\langle r_{n,i}\mid i<\omega\rangle\s G$ is a $\leq$-decreasing sequence with $\ell(r_{n,i+1})>\ell(r_{n,i}).$
    \item[$(\beta)$] $\langle C_{n,i}\mid i<\omega\rangle$ is a $\subsetneq$-increasing sequence of subsets of $\nu_n\setminus \nu_{n-1}$ in $V[G]$ such that $\bigcup_{i<\omega}C_{n,i}=\nu_n\setminus \nu_{n-1}.$ Here, by convention, $\nu_{-1}:=0.$
    \item[$(\gamma)$] $r_{n,i}\forces_{\mathbb{P}}``\dot{f}\restriction C_{n,i}\in\check{V}$'', where $\dot{f}$ is a $\mathbb{P}$-name such that $\dot{f}_G=f.$ 
\end{enumerate}
For the rest of the proof we fix $n<\omega$ such that $f\restriction (\nu_{n}\setminus \nu_{n-1})\notin V$ -- otherwise one simply sets $B_n:=\nu_n\setminus \nu_{n-1}$. Since $n$ will be fixed, in what follows we opt to write $\langle r_i\mid i<\omega\rangle$ and $\langle C_i\mid i<\omega\rangle$ in place of the more cumbersome notation involving the extra subindex $n$.

\smallskip

Let $p\in G$ and $\tau$ be a $\mathbb{P}$-name such that $p\forces_{\mathbb{P}}``\tau=\dot{f}\restriction (\nu_n\setminus \nu_{n-1})$``. The construction of the above sequences is accomplished by induction on $i<\omega.$

\smallskip
\underline{Case $i=0$:} For each $\xi\in  [\nu_{n-1},\nu_n)$ consider the coloring $d_\xi:\mathbb{P}\rightarrow 2$,
    \[
d_\xi(q):=\begin{cases*}
  1, & if $q\parallel\tau(\Check{\xi})$ (i.e., $q\Vdash_\mathbb{P}\tau(\check{\xi})=\Check{e}_\xi$, for some $e_\xi\in E$);\\
  0, & otherwise.
\end{cases*}
\]   
$d_\xi$ is a $0$-open coloring (see Definition~\ref{def: open coloring}) and so Lemma~\ref{lem: indiscernibles} applies.

\begin{claim}
  Let $D_0$ be the set of all conditions $r\leq p$ such that: \begin{enumerate}
        \item\label{def of D_0, 1} For all $\xi\in\nu_n\backslash\nu_{n-1}$, $\mathbb{P}_{\downarrow r}$ is a set of indiscernibles for $d_\xi$;
        \item\label{def of D_0, 2} $C_r$ is non-empty and $r\Vdash_{\mathbb{P}}\tau\restriction\check{C}_{r}\in \check{V}$, being
        $$C_r:=\{\xi\in [\nu_{n-1},\nu_n)\mid r\parallel \tau(\check{\xi})\}.$$
        
    \end{enumerate}
    Then $D_0$ is dense below $p\in G.$
\end{claim}
\begin{proof}[Proof of claim]
    Fix $q\leq p$. We find $r\leq q$ in $D_0$. Without loss of generality $\ell(q)>n$, hence the $\leq^*$-order of $\mathbb{P}$ below $q$ is $\kappa_n$-closed and  $\nu_n<\kappa_n$. 
    
    First, we show how to meet Clause~(1) of the claim. For this we construct a $\leq^*$-increasing sequence of auxiliary conditions $\langle q_\xi\mid \xi\in [\nu_{n-1},\nu_{n})\rangle$ below $q$ such that each $\mathbb{P}_{\downarrow q_{\xi}}$ is a set of indiscernibles for $d_\xi.$ To do so put $q_{\nu_{n-1}}:=q$, invoke Lemma~\ref{lem: indiscernibles}(1) at successor stages and combine the  closure of $\langle \mathbb{P}_{\downarrow q},\leq^*\rangle$ with the said lemma. This allows us to construct each $q_\xi$. In fact, more is true -- this allows us to define $r_0\leq^* q$ that serves as a $\leq^*$-lower bound for the $q_\xi$'s. Clearly, $r_0$ witnesses Clause~(1).  By $\leq$-extending $r_0$ to $r$ if necessary, we may assume that $r\forces_{\mathbb{P}} \tau(\check{\nu}_{n-1})=\check{e}_{\nu_{n-1}}$ for some $e_{\nu_{n-1}}\in E$, which yields $C_r\neq \emptyset$. By setting $h:=\{\langle \xi, e_\xi\rangle \mid \xi\in C_r\}$ we conclude that $r\forces_{\mathbb{P}} \tau\restriction \check{C}_r=\check{h}\in \check{V}.$ 
\end{proof}
 Let $r_0\in G\cap D_0$ and stipulate $C_0:=C_{r_0}$. This witnesses $(\gamma)$.

\smallskip

\underline{Induction step:} Suppose that $\langle r_i\mid i\leq j\rangle$ and $\langle C_i\mid i\leq j\rangle$ have been constructed complying with Clauses~$(\alpha)-(\gamma)$ above. Note that $C_j$ is not $[\nu_{n-1}, \nu_{n})$ for otherwise $r_j\in G$ would have decided  the value of $f\restriction [\nu_{n-1},\nu_n)$, contrary to our departing assumption. As a result the set
$$D_{j+1}:=\{r\leq r_j\mid \ell(r)>\max(\ell(r_j), j)\;\wedge\; r\forces_{\mathbb{P}}\text{$``\tau\restriction \check{C}_r\in \check{V}$''}\;\wedge\; C_j\subsetneq C_r\}$$
is dense below $r_j\in G$. 

So again we can pick $r_{j+1}\in G\cap D_{j+1}$ and set $C_{j+1}:=C_{r_{j+1}}.$

\medskip

This finishes the construction of $\langle r_i, C_i\mid i<\omega\rangle$ witnessing $(\alpha)$--$(\gamma)$.

\begin{claim}
    $\bigcup_{i<\omega}C_i=[\nu_{n-1},\nu_n)$ and $f\restriction C_i\in V.$
\end{claim}
\begin{proof}[Proof of claim]
    By construction,  $r_i\forces_{\mathbb{P}} \dot{f}\restriction \check{C}_i=\tau\restriction \check{C}_i\in \check{V}$ hence (as $r_i\in G$) $f\restriction C_i\in V.$ For the other claim we use the indiscernibility of $r_0$: Let $\xi\in [\nu_{n-1},\nu_n)$ and $r\leq r_0$ such that $r\parallel \tau(\xi)$. Since the length of the $r_i$'s is cofinal in $\omega$ we can find $r_i\leq r_0$ such that $\ell(r_i)>\ell(r)$. By extending $r$ further if necessary we may assume that both $r_i$ and $r$ have the same length. Since $r$ decides the value of $\tau(\xi)$ that means that $d_\xi(r)=1$. However, every member of the cone $\mathbb{P}_{\downarrow r_0}$ is an indiscernible for $d_\xi$ so that $d_\xi(r_i)=1$, as well. This tantamounts to saying $\xi\in C_{i}$, and we are done.
\end{proof}
The above argument was carried out for a fixed $n<\omega$ so 
we obtain a sequence 
$\langle C_{n,i}\mid i,n<\omega\rangle$ such that $\bigcup_{i<\omega}C_{n,i}:=[\nu_{n-1},\nu_n)$. Pick a  bijection $F\colon \omega\rightarrow\omega\times\omega$ (e.g., the inverse of G\"odel pairing
function) and stipulate $B_m:=C_{F(m)}$, for all $m<\omega$. Clearly $\langle B_m\mid m<\omega\rangle$ is as desired.
\end{proof}

The time is now ripe to formulate and prove the \emph{Interpolation Lemma}:

\begin{lemma}[Interpolation]\label{lemma: interpolation}
    Let $\mathbb{P}$ and $\mathbb{Q}$ be $\Sigma$-Prikry forcings such that $$\text{$\mathbb{Q}\in H_\lambda$ and $\one\forces_{\mathbb{P}}``\cf(\kappa)^{V[\dot{G}]}=\omega\ \wedge \ (\kappa^+)^{V[\dot{G}]}=\lambda$'',}$$ and let $\mathbb{R}$ be any forcing (not necessarily $\Sigma$-Prikry and possibly trivial). 
    
    Suppose also that we are given a commutative system of weak projections 
    \begin{displaymath}
       \begin{tikzcd}
  \mathbb{P} \arrow{r}{\pi_1} \arrow[bend right]{rr}{\pi} & \mathbb{Q} \arrow{r}{\pi_{2}}  & \mathbb{R}
\end{tikzcd}
    \end{displaymath}
    
    and that its mirror $\mathcal{B}$-system of projections (see Lemma~\ref{from a weak system to a system}) is
       \begin{displaymath}
       \begin{tikzcd}
  \ro(\mathbb{P}) \arrow{r}{\rho_1} \arrow[bend right]{rr}{\rho} & \ro(\mathbb{Q}) \arrow{r}{\rho_{2}}  & \ro(\mathbb{R})
\end{tikzcd}.\footnote{Recall that by Convention~\ref{conv: convention about projections of boolean algebras} we are implicitly agreeing that the new weakest conditions of $\ro(\mathbb{Q})$ and $\ro(\mathbb{R})$ are, respectively, $\rho_1(\one_{\mathbb{P}})$ and $\rho(\one_{\mathbb{Q}}).$ Formally speaking, these may not be the original trivial conditions in $\ro(\mathbb{Q})$ and $\ro(\mathbb{R})$  but in that case we simply force below the cone given by the new  trivial conditions.  }
    \end{displaymath}

   \emph{\textbf{(Interpolation)}} Let $G$ be a $\ro(\mathbb{P})$-generic, $g$ be a $\ro(\mathbb{R})$-generic with $g\in V[G]$, and $p\in\ro(\mathbb{P})/g$. Then there is a $\ro(\mathbb{Q})/g$-generic filter $h$ such that $\rho_1(p)\in h$. 
    
    Moreover, $h$ is obtained as the upwards closure of a $\leq$-decreasing sequence $\langle p_n\mid n<\omega\rangle\in V[G]$  of conditions in  $\ro(\mathbb{Q})/g$ below $\rho_1(p)$. Explicitly,
    $$h=\{b\in\ro(\mathbb{Q})\mid \exists n<\omega\,(p_n\leq b)\}$$
    whose corresponding sequence of lengths $\langle \ell(p_n)\mid n<\omega\rangle$ is cofinal in $\omega$. 

    \smallskip
    
    \emph{\textbf{(Capturing)}} In addition, if in $V[g]$ there is a projection $$\sigma\colon(\ro(\mathbb{Q})/g)_{\downarrow t}\rightarrow\ro(\mathbb{R})_{\downarrow\Bar{p}},$$ for some $t\in(\ro(\mathbb{Q})/g)_{\downarrow\rho_1(p)}$ and $\Bar{p}\in \rho``G$, we can choose $h$ above in such a way that $t\in h$ and $\rho``G\in V[h]$.
\end{lemma}

\begin{proof}
First we prove \emph{Interpolation} and then we will dispose with \emph{Capturing}.

\smallskip

    \underline{\textbf{Interpolation}:} As $|\mathbb{Q}|<\lambda$ and $\lambda$ is inaccessible, ${|\mathcal{P}(\mathbb{Q})|}<\lambda$. By our assumption $V[G]\models``\lambda=(\kappa^+)$'', hence there is an enumeration $\mathcal{D}:=\langle D_\alpha\mid\alpha<\kappa\rangle$ in $V[G]$ of the set  $E:=\{D\in \mathcal{P}(\mathbb{Q})^V\mid D\;\text{is dense open in}\; \mathbb{Q}\}.$ Applying Lemma~\ref{lemma: goodness} to $\mathbb{P}$ and $\mathcal{D}$, we find a sequence $\langle B_n\mid n<\omega\rangle\in V[G]$ of bounded subsets of $\kappa$ such that $\bigcup_{n<\omega}B_n=\kappa$ and $\mathcal{D}\restriction B_n\in V$ for $n<\omega$. 
    
    We construct a sequence $\langle p_n\mid n<\omega\rangle$ such that $p_{n+1}\leq p_n\leq\rho_1(p)$ and $\ell(p_{n+1})>\ell(p_n)$, for all $n<\omega$. We proceed by induction.  For $n=0$ let  $$E_0:=\{q\in\mathbb{Q}_{\downarrow\rho_1(p)}\mid \forall\alpha\in B_0\,\exists k_\alpha<\omega \ \mathbb{Q}^q_{\geq k_\alpha}\subseteq D_\alpha\}.$$ 
    Here $\mathbb{Q}_{\downarrow\rho_1(p)}$ is a shorthand for the poset $\mathbb{Q}\cap \ro(\mathbb{Q})_{\downarrow\rho_1(p)}$ (note that formally speaking $\rho_1(p)$ may not be a condition in $\mathbb{Q}$ but only a condition in $\ro(\mathbb{Q})$).
    \begin{claim}\label{claim: E_0 is dense in Q}
        $E_0$ is a dense subset of $\ro(\mathbb{Q})_{\downarrow\rho_1(p)}$.
    \end{claim}
    \begin{proof}[Proof of claim]
        Note that $E_0$ is in $V$ because $\mathcal{D}\restriction B_0\in V$ -- this is crucial to run the forthcoming argument.
    Let $b\in\ro(\mathbb{Q})_{\downarrow\rho_1(p)}$ and let $q\in\mathbb{Q}_{\downarrow b}$. Since $B_0$ is bounded in $\kappa$ there is $n<\omega$ such that $B_0\subseteq \kappa_m$. Set $\langle \alpha_\beta\mid \beta<\gamma\rangle$ a one-to-one enumeration of $B_0$ and (by extending $p$ if necessary) suppose that $\ell(p)>n$. We construct a decreasing sequence $\langle u_\beta\mid \beta<\gamma\rangle$ $\leq^*$-below $q$ such that   $\mathbb{Q}^{u_\beta}_{\geq k_\beta}$ is included in $D_{\alpha_\beta}$ for some $k_\beta<\omega$. The construction of these $u_\beta$'s   combines the \emph{Strong Prikry Lemma} (Lemma~\ref{lem: indiscernibles}(3)) with  $\kappa_m$-closure of $\langle \mathbb{Q}_{\downarrow q},\leq^*\rangle.$ 
    Finally, let $q^*\leq^* q$ a lower bound for the $u_\beta$'s. 
    \end{proof} Since $\rho_2$ is a projection $\rho_2``E_0$ is a dense subset of $\ro(\mathbb{R})_{\downarrow \rho(p)}$ (recall that $\rho_2(\rho_1(p))=\rho(p)$). Moreover, $\rho(p)\in g$ and so $\rho_2``E_0\cap g\neq\emptyset$. Pick a condition $q\in E_0$ with $\rho(q)\in\rho_2``E_0\cap g$ and stipulate $p_0:=q$. 
    
    \smallskip

    Now suppose $\langle p_m\mid m\leq n\rangle$ has already been defined, for some $n<\omega$, and define $$E_{n+1}=\{q\in\mathbb{Q}_{\downarrow p_n}\mid \ell(q)>\ell(p_n)\ \wedge\ \forall\alpha\in B_{n+1}\,\exists k_\alpha<\omega \ \mathbb{Q}^q_{\geq k_\alpha}\subseteq D_\alpha\}.$$ Carrying out exactly the same argument in the proof of Claim \ref{claim: E_0 is dense in Q}, we deduce that $E_{n+1}$ is dense below $p_n$. Therefore, there is  $q\in E_{n+1}$ such that $\rho_2(q)\in g$. Stipulate $p_{n+1}:=q$. Thus we have constructed a $\leq$-decreasing sequence $\langle p_n\mid n<\omega\rangle$ such that $\ell(p_{n+1})>\ell(p_n)$, $p_n\in E_n$ and $\rho_2(p_n)\in g$, for all $n<\omega$. Let $h$ be the $\ro(\mathbb{Q})$-upwards closure of $\langle p_n\mid n<\omega\rangle$, i.e. $$h=\{b\in\ro(\mathbb{Q})\mid \exists n<\omega\ (p_n\leq b)\}.$$
\begin{claim}\label{claim: h is ro(Q)/g-generic}
    $h$ is $\ro(\mathbb{Q})/g$-generic over $V[g]$.
\end{claim}
\begin{proof}[Proof of claim]
    By construction $p_n\in \ro(\mathbb{Q})/g$ for all $n<\omega$ so is enough to verify that $h$ is $\ro(\mathbb{Q})$-generic. Let $D$ be dense open in $\ro(\mathbb{Q})$. 
    Then the set $$D^\ast:=\{q'\in\mathbb{Q}\mid q'\in D\cap\mathbb{Q}_{\downarrow \rho_1(p)}\,\vee\, q'\text{ and }\rho_1(p)\text{ are incompatible}\}$$ is a dense open subset of $\mathbb{Q}$. 
   
   Therefore there are $\alpha<\kappa$ and $n<\omega$ such that $D^\ast=D_\alpha$ and $\alpha\in B_n$. By definition of $E_n$, there is some $k_\alpha<\omega$ such that $\mathbb{Q}^{p_n}_{\geq k_\alpha}\subseteq D^\ast$. But by construction the sequence $\langle \ell(p_n)\mid n<\omega\rangle$ is cofinal in $\omega$ and so there is some $m<\omega$ with $\ell(p_m)\geq\ell(p_n)+k_\alpha$, yielding $p_m\in D^\ast$. Thus, $p_m$ witnesses  $h\cap D\neq\emptyset$, as needed.
\end{proof}

\smallskip

    \underline{\textbf{Capturing:}} Work in $V[G]$ and assume $t\in(\ro(\mathbb{Q})/g)_{\downarrow\rho_1(p)}$, $\Bar{p}\in\rho``G$ and $$\sigma\colon(\ro(\mathbb{Q})/g)_{\downarrow t}\rightarrow\ro(\mathbb{R})_{\downarrow\Bar{p}}$$ are as in the assumptions of the lemma. Mimicking the previous construction we  define a $\ro(\mathbb{Q})/g$-generic $h$ which \emph{captures} the generic filter $\rho``G$.

\smallskip

To simplify notations let us agree that $(\mathbb{Q}/g)_{\downarrow t}:= \mathbb{Q}\cap (\ro(\mathbb{Q})/g)_t$. 
Let $$F_0=\{u\in(\mathbb{Q}/g)_{\downarrow t}\mid\forall\alpha\in B_0\exists n_\alpha<\omega\ \mathbb{Q}^u_{\geq n_\alpha}\subseteq D_\alpha\}.$$
        \begin{claim}\label{claim: F_0 is a dense subset of ro(Q)/g}
            $F_0$ is a dense subset of $(\ro(\mathbb{Q})/g)_{\downarrow t}$.
        \end{claim}
        \begin{proof}
            Let $s\in(\ro(\mathbb{Q})/g)_{\downarrow t}$. Since $\rho_2:\ro(\mathbb{Q})\rightarrow\ro(\mathbb{R})$ is a projection and$$E_0=\{v\in \mathbb{Q}_{\downarrow s}\mid \forall\alpha\in B_0\,\exists n_\alpha<\omega\ \mathbb{Q}^v_{\geq n_\alpha}\subseteq D_\alpha\}$$ is a dense subset of $\ro(\mathbb{Q})_{\downarrow s}$ (see Claim \ref{claim: E_0 is dense in Q}), it follows that $\rho_2``E_0$ is a dense subset of $\ro(\mathbb{R})$ below $\rho_2(s)$. Since $\rho_2(s)\in g$,  $g\cap \rho_2``E_0\neq\emptyset$. Let $s'\in E_0$ be such that $\rho_2(s')\in g$. Clearly, $s'\leq s$ and moreover $s'\in F_0$.
        \end{proof} Since $\sigma\colon(\ro(\mathbb{Q})/g)_{\downarrow t}\rightarrow\ro(\mathbb{R})_{\downarrow\Bar{p}}$ is a projection, $\sigma``F_0$ is a dense subset of $\ro(\mathbb{R})_{\downarrow\Bar{p}}$ and since  $\Bar{p}\in\rho``G$ it follows that $\rho``G\cap\sigma``F_0\neq\emptyset.$ Let $t_0$ be a condition in $F_0$ such that $\sigma(t_0)\in \rho``G$. It is clear by construction that $t_0\leq t$, $\rho_2(t_0)\in g$ and $\sigma(t_0)\leq \Bar{p}$. In general, we define a $\leq$-decreasing sequence $\langle t_n\mid n<\omega\rangle$ with $\sigma(t_n)\in \rho``G$ as in the previous argument taking as a dense set 
        
        $$F_n=\{u\in(\mathbb{Q}/g)_{{\downarrow t_{n-1}}}\mid\ell(u)>\ell(t_{n-1})\ \wedge\ \forall\alpha\in B_n\exists n_\alpha<\omega\ \mathbb{Q}^u_{\geq n_\alpha}\subseteq D_\alpha\}.$$ 
        
       As in the ``interpolation argument'', we let $$h=\{b\in\ro(\mathbb{Q})\mid \exists n<\omega\ (t_n\leq b)\}.$$ The argument in Claim \ref{claim: h is ro(Q)/g-generic} shows that $h$ is $\ro(\mathbb{Q})/g$-generic over $V[g]$ and by construction $t\in h$ and $\sigma(t_n)\in\rho``G$ for all $n<\omega$. Since $\sigma\colon(\ro(\mathbb{Q})/g)_{\downarrow t}\rightarrow\ro(\mathbb{R})_{\downarrow\Bar{p}}$ is a projection and $t\in h$,  $\sigma``h_{\downarrow t}$ is $\ro(\mathbb{R})_{\downarrow\Bar{p}}$-generic. Also, $\sigma``h_{\downarrow t}\subseteq(\rho``G)_{\downarrow\Bar{p}}$ so, by maximality, $\sigma``h_{\downarrow t}=(\rho``G)_{\downarrow\Bar{p}}$.

        This implies that $\rho``G\in V[h]$,  
        thus showing that $h$ has the desired capturing feature.
        
\end{proof}
\begin{remark}
    If the original $\pi_i$'s were projections one would be able to prove interpolation lemma (including its capturing feature) replacing the Boolean algebras $\ro(\mathbb{P})$, $\ro(\mathbb{Q})$ and $\ro(\mathbb{R})$ for the posets $\mathbb{P},\mathbb{Q}$ and $\mathbb{R}$.
\end{remark}

\subsection{The constellation lemma}\label{sec: direct systems of posets}

In order to build a model of $\zfc$ in which all subsets of $^{\omega}{\kappa}$ in $L(V_{\kappa+1})$ have the $\kappa$-$\psp$ we mimic Solovay's original proof \cite{solovay1970model} that (consistently) every set in $\mathcal{P}({}^{\omega}{\omega})\cap L(\mathbb{R})$ has the $\psp\,$ and $\bp$. 

Solovay's configuration is established in a generic extension where an inaccessible cardinal $\lambda$ is collapsed to  $\omega_1$ via the standard LÃ©vy collapse $\Col(\omega,{<}\lambda)$ (\cite[p.~127]{Kan}). Solovay  factors  $\col(\omega,{<}\lambda)$ as $$\col(\omega,{<}\alpha)\times\col(\alpha,{<}\beta)\times\col(\beta,{<}\lambda)$$ and shows that if $A$ is an ordinal-definable  uncountable set of reals in a generic extension by ${\col(\omega,{<}\lambda)}$ then there is a perfect set of $\col(\alpha,{<}\beta)$-generics over $V[G\restriction\alpha]$ (here homogeneity of $\col(\alpha,{<}\beta)$ becomes crucial).  Since the other factor $\col(\beta,{<}\lambda)$ is also homogeneous, each element $x$ in the perfect set satisfies the formula  defining $A$ in $V[G\restriction\alpha][x][H]$ for all $\col(\beta,{<}\lambda)$-generic $H$. But Solovay proved that there is a $\col(\beta,{<}\lambda)$-generic $H_x$ such that $V[G]=V[G\restriction\alpha][x][H_x]$, so the perfect set is actually inside $A$. 

\smallskip

In our case, we start with an inaccessible $\lambda$ above some (possibly singular) cardinal $\kappa$ that will eventually become (or remain) of countable cofinality after a $\Sigma$-Prikry forcing $\mathbb{P}$. Like Solovay, we collapse $\lambda$ to $\kappa^+$ via $\mathbb{P}$. 
In our case, we will pick a subset $A$ of the $\kappa$-Baire space ${}^\omega\kappa$ of cardinality ${>}\kappa$ living in $L(V_{\kappa+1})$ (this latter as computed in the
generic extension by $\mathbb{P}$). Then we will factor $\mathbb{P}$ as a three-stages iteration, where the first factor of the iteration will be responsible for capturing the parameters defining $A$ (in $V^{\mathbb{P}}$). The Interpolation Lemma (Lemma \ref{lemma: interpolation}) will help us construct a $\kappa$-perfect set of small generics -- the construction will take place in the second stage of the iteration.
Eventually, this set of generics will induce a $\kappa$-perfect set in $L(V_{\kappa+1})$ as it will be demonstrated in Theorem~\ref{the main construction}. It would remain to provide an analogue of the generic $H_x$ in Solovay's proof but it turns out that this is unnecessary: to prove that the $\kappa$-perfect set is inside $A$ we will argue that it suffices to exhibit a $\mathbb{P}$-generic $H^*$ such that $\mathcal{P}(\kappa)^{V[H]}=\mathcal{P}(\kappa)^{V[H^*]}$. 

\smallskip

This section is devoted to define a poset (the \emph{constellating forcing})   which yields this generic $H^*$. The poset is a generalization of Kafkoulis forcing $\mathcal{Q}$ in \cite[Definition~3.1.1]{Kafkoulis1994-KAFTCS}. The section's main result is the \emph{Constellation Lemma} (Lemma~\ref{lemma: computing correctly the power set}) which displays the main properties of the  forcing.

\smallskip

Recall that we are still working under the precepts of Setup~\ref{setup kappapsp}.

\begin{definition}\label{def: directed system}
Let $(\mathcal{D},\sle)$ be a directed set with a maximal element $\infty$.\footnote{I.e., for all $d,e\in\mathcal{D}\setminus \{\infty\}$ there is $f\in\mathcal{D}\setminus \{\infty\}$ such that $d,e\sle f$.}

A sequence $$\mathcal{P}=\langle\mathbb{P}_d, \pi_{e,d}\colon \mathbb{P}_e\rightarrow\mathbb{P}_d\mid d,e\in \mathcal{D}\ \wedge\ d\sle e\rangle$$
is a \emph{(weak) directed system of forcings} (shortly, a \emph{(weak) system}) whenever:
\begin{enumerate}
	\item $\mathbb{P}_d$ is a  forcing poset for all $d\in \mathcal{D}$;
	\item $\pi_{e,d}$ is a (weak) projection 
 and $\pi_{d,d}=\id$ for all $d\sle e$ in  $\mathcal{D}$;
 \item $\pi_{f,d}=\pi_{e,d}\circ \pi_{f,e}$ for all $d\sle e\sle f$ in $\mathcal{D}$.
	
\end{enumerate}
For an index $e\in \mathcal{D}$ the \emph{$e$-truncation} of $\mathcal{P}$ is the (weak) system 
$$\mathcal{P}\restriction e:=\langle\mathbb{P}_d, \pi_{d,f}\colon \mathbb{P}_d\rightarrow\mathbb{P}_f\mid d,f\in \mathcal{D}\ \wedge\ d\sle f\sle e\rangle$$

If $\langle \mathbb{P}_d\mid d\in\mathcal{D}\rangle$  happen to be  $\Sigma$-Prikry one says that $\mathcal{P}$ is a \emph{(weak) system of $\Sigma$-Prikry forcings} (or, shortly, a \emph{(weak) $\Sigma$-system}).

     For the sake of readability we will denote  $$\text{$\mathcal{D}^*:=\mathcal{D}\setminus \{\infty\}$,  $\mathbb{P}:=\mathbb{P}_{\infty}$ and $\pi_d:=\pi_{\infty,d}$.}$$
\end{definition}

Given a weak system $\mathcal{P}$ there is a natural way to construct a system $\mathcal{B}$ via the duality between weak projections and projections between the corresponding regular open Boolean algebras (see Fact~\ref{fact: projection to the rescue} and Lemma~\ref{from a weak system to a system}):
\begin{lemma}\label{a system of boolean algebras}
    Given a weak system $$\mathcal{P}=\langle\mathbb{P}_d, \pi_{e,d}\colon \mathbb{P}_e\rightarrow\mathbb{P}_d\mid d,e\in \mathcal{D}\ \wedge\ d\sle e\rangle$$ there is an associated system between their regular open algebras $$\mathcal{B}:=\langle{\ro(\mathbb{P}_d)}_{\downarrow\rho_d(\one_{\mathbb{P}})}, \rho_{e,d}\colon \ro(\mathbb{P}_e)_{\downarrow\rho_e(\one_{\mathbb{P}})}\rightarrow{\ro(\mathbb{P}_d)}_{\downarrow\rho_d(\one_{\mathbb{P}})}\mid d,e\in \mathcal{D}\ \wedge\ d\sle e\rangle.\qed$$ 
\end{lemma}
 
\begin{definition}\label{def: boolean completion of a system}
    The \emph{Boolean completion} of a weak system $\mathcal{P}$ is the system  $\mathcal{B}$ defined above. We will denote $\mathbb{B}:=\ro(\mathbb{P})$ and $\mathbb{B}_e:=\ro(\mathbb{P}_e)_{\downarrow\rho_e(\one_{\mathbb{P}})}$ for $e\in\mathcal{D}$.
\end{definition}
Formally speaking the Boolean completion $\mathcal{B}$ of a weak $\Sigma$-system $\mathcal{P}$ may not be a $\Sigma$-system. Nevertheless, $\mathcal{B}$ will inherit all the pleasant $\Sigma$-Prikry features of the individuals of $\mathcal{P}$ in that $\mathbb{P}_e$ is a dense subforcing of $\mathbb{B}_e$.

\begin{remark}
Let us explain our reasons for considering weak projections instead of projections and for passing to the corresponding Boolean completions. The first reason is that there are \(\Sigma\)-Prikry forcings which do not admit projections to their corresponding ``small" forcings -- only weak projections are available. Examples of such \(\Sigma\)-Prikry forcings will be given in \S\ref{sec: applications}, but these pathologies were already observed by Foreman and Woodin \cite{ForWoo}. The reason for looking at the Boolean completion of weak systems is because projections provide a safe environment to make claims of the following nature: ``If \(D\) is dense in \(\mathbb{P}\), then \(\pi_e``D\) is dense in \(\mathbb{P}_e\)" or ``If \(g\) is \(\mathbb{P}_e\)-generic and \(G\) is \(\mathbb{P}/g\)-generic, then \(G\) is \(\mathbb{P}\)-generic and \(V[G] = V[g][G]\)." If the \(\pi_e\)'s are mere weak projections, these claims  may  fail.
\end{remark}

\smallskip

The main technical devise developed in this section is the \emph{constellating forcing $\mathbb{C}$}. Before introducing it and explain its functionality we  first present the notion of a $\mathcal{P}$-sky and $\mathcal{P}$-constellation of a $\mathbb{P}$-generic filter $G$:
\begin{definition}
    Let $\mathcal{P}$ be a system and $G$ a $\mathbb{P}$-generic filter.
    \begin{enumerate}
        \item The \emph{$\mathcal{P}$-sky of $G$} is $\sky_{\mathcal{P}}(G):=\{g\in V[G]\mid \exists e\in \mathcal{D}^*\, \text{($g$ is $\mathbb{P}_e$-generic)}\}$.
        \item The \emph{$\mathcal{P}$-constellation of $G$} is $\con_{\mathcal{P}}(G):=\{\pi_e``G\mid e\in \mathcal{D}^*\}$.
    \end{enumerate}
\end{definition}
\begin{remark}
    Clearly, $\con_{\mathcal{P}}(G)\s \sky_{\mathcal{P}}(G)$. Also note since $\mathcal{P}$ is a system, $$\text{$g\in \con_{\mathcal{P}}(G)$ if and only if 
 $G$ is $\mathbb{P}/g$-generic.}$$
\end{remark}
Given a system $\mathcal{P}$ and a $\mathbb{P}$-generic filter $G$,  the \emph{constellating poset} $$\mathbb{C}:=\mathbb{C}(\mathcal{P},G)$$ will be  defined in $V[G]$ and will have the next two properties:
\begin{enumerate}
    \item \textbf{($\kappa$-captures)} Forcing with $\mathbb{C}$  induces a $\mathbb{P}$-generic $G^*$ such that $$\mathcal{P}(\kappa)^{V[G]}=\mathcal{P}(\kappa)^{V[G^*]}.$$
    \item \textbf{(Constellates)} Given $g\in \sky_{\mathcal{P}}(G)$ there is $c\in\mathbb{C}$ such that
    $$V[G]\models\text{$``c\forces_{\mathbb{C}}\check{g}\in \con_{\mathcal{P}}(\dot{G}^*)$''.}$$
 \end{enumerate}

\begin{definition}[Nice systems]\label{def: nice system}
A  (weak) system $\mathcal{P}$ is  \emph{$(\kappa,\lambda)$-nice} if:
\begin{enumerate}[label=(\greek*)]
    \item $\mathcal{P}$ is \emph{$\kappa$-capturing}: \begin{equation*}\label{equation: capturing subsets}
       \one\forces_{\mathbb{P}} \forall x\in\mathcal{P}(\kappa)^{V[\dot{G}]}\,\forall d\in \mathcal{D}^*\,\exists e\in \mathcal{D}^*(d\sle e\,\wedge\, x\in V[\pi_{e}``\dot{G}]);
    \end{equation*}
 \item\label{directed system: Clause 2} $\mathcal{P}$ is \emph{$\lambda$-bounded}:   $$\text{$\langle \mathbb{P}_d\mid d\in\mathcal{D}^*\rangle \s H_\lambda$};$$
 
 \item\label{directed system: Clause 3} $\mathcal{P}$ is \emph{amenable to interpolations}: $$\text{$\one \Vdash_{\mathbb{P}}``\lambda=(\kappa^+)^{V[\dot{G}]}\ \wedge\ \cf(\kappa)^{V[\dot{G}]}=\omega$''.}$$
\end{enumerate} 
\end{definition}
If $\mathcal{P}$ is a weak $(\kappa,\lambda)$-nice system its Boolean completion $\mathcal{B}$ is a $(\kappa,\lambda)$-nice system. This is because $\mathbb{B}$ and $\mathbb{P}$ are forcing equivalent (hence $(\alpha)$ and $(\gamma)$ hold for $\mathbb{B}$) and members of $\mathcal{B}$ belong to $H_\lambda$ because $\lambda$ is  inaccessible.

\begin{lemma}\label{lemma: capturing generic filters with a small generic}
     
Suppose that $\mathcal{P}$ is a $(\kappa,\lambda)$-nice system and $G\s \mathbb{P}$ is $V$-generic. For each $g\in \sky_{\mathcal{P}}(G)$ there is $e\in \mathcal{D}^*$ such that $g\in \sky_{\mathcal{P}\restriction e}(\pi_e``G).$   
\end{lemma}
\begin{proof}
Let $d\in\mathcal{D}^*$ such that $g$ is $\mathbb{P}_d$-generic.
    By Clauses~\ref{directed system: Clause 2} and $(\gamma)$ in  Definition~\ref{def: nice system},  
    $|\mathrm{tcl}(\mathbb{P}_d)|<\lambda=(\kappa^+)^{V[G]}.$ Working in $V[G]$, pick a bijection $f\colon \nu\rightarrow \mathrm{tcl}(\mathbb{P}_d)$, for some $\nu\leq\kappa$, and define 
    $\text{$\xi\triangleleft\delta:\Longleftrightarrow f(\xi)\in f(\delta)$.}$ Clearly $\triangleleft$ can be coded as a subset of $\eta\leq\kappa$. 
    By the $\kappa$-capturing property of $\mathcal{P}$  there is $d\sle e_0\in \mathcal{D}^*$ such that $(\nu,\triangleleft)\in V[\pi_{e_0}``G]$. The transitive collapse of $(\nu,\triangleleft)$ as computed in $V[\pi_{e_0}``G]$ has to be $(\mathrm{tcl}(\mathbb{P}_d),\in)$, and consequently $\mathrm{tcl}(\mathbb{P}_d)\in V[\pi_{e_0}``G]$ as well. In fact, the isomorphism in $V[\pi_{e_0}``G]$ witnessing $(\nu,\triangleleft)\cong(\mathrm{tcl}(\mathbb{P}_d),\in)$ is unique and so it has to be $f$, yielding $f\in V[\pi_{e_0}``G]$. On the other hand, $g\subseteq\mathbb{P}_d\subseteq\mathrm{tcl}(\mathbb{P}_d)$ so as before we can code $g$ as a subset of $\kappa$. Specifically,  let $a:=\{\alpha\in\nu\mid f(\alpha)\in g\}$, which is clearly a subset of $\kappa$ in $V[G]$. Use again  \ref{equation: capturing subsets} to pick $e_0\sle e_1\in \mathcal{D}^*$ such that $a\in V[\pi_{e_1}``G]$. 
    Since $g=f``a$, $g$ belongs to $V[\pi_e``G]$; that is, $g\in \sky_{\mathcal{P}\restriction e}(\pi_e``G).$  
\end{proof}

\textbf{Disclaimer:}\label{Disclaimer} All the results in this section as well as those in the subsequent ones (i.e., \S\ref{sec: sigma prikry and kappaperfect} and $\S\ref{sec: Sigma Prikry and BP}$) work for arbitrary $(\kappa,\lambda)$-nice $\Sigma$-systems. Thus if we are given a $(\kappa,\lambda)$-nice \textbf{$\Sigma$-system} $\mathcal{P}$ passing to its Boolean completion $\mathcal{B}$ becomes  unnecessary. The reason for why we formulate the next results for the Boolean completions of a \textbf{weak}  $(\kappa,\lambda)$-nice $\Sigma$-system $\mathcal{P}$ is that $\Sigma$-Priky members of $\mathcal{P}$ densely (and canonically) embed into members of $\mathcal{B}$. Nevertheless, notice that the Boolean algebras in $\mathcal{B}$ are not $\Sigma$-Prikry. \qed

\medskip

The time is now ripe to present the constellating forcing, $\mathbb{C}$. For this we  begin with a weak $(\kappa,\lambda)$-nice $\Sigma$-system $\mathcal{P}$ and pass to its Boolean completion $\mathcal{B}$, which is a $(\kappa,\lambda)$-nice system.  $\Sigma$-Prikry forcings are necessary to fulfill the hypothesis of the \emph{Interpolation Lemma} proved in \pageref{lemma: interpolation}. This, as we will show,  is the key ingredient to ensure that $\mathbb{C}$ indeed ``constellates''.
\begin{definition}\label{def: fancy poset}
     Let $\mathcal{B}$ be the Boolean completion of a $(\kappa,\lambda)$-nice $\Sigma$-system $\mathcal{P}$  
    and let $G$ be $\mathbb{B}$-generic. Working in $V[G]$, we define the \emph{constellating poset} $\mathbb{C}:=\mathbb{C}(\mathcal{B},G)$ as the collection of all triples $\langle p,d,g\rangle$ such that $d\in\mathcal{D}^*$ witnesses $g\in \sky_{\mathcal{B}}(G)$ and $p\in \mathbb{B}/g$.

     The order between conditions is $$\langle q,e,h\rangle\leq \langle p,d,g\rangle$$ if and only one of the following requirements is met: 
     \begin{itemize}
         \item  $q\leq_{\mathbb{B}}p$ and $\langle d, g\rangle=\langle e,h\rangle$;
         \item  $q\leq_{\mathbb{B}} p$, $d\prec e$  and $g\in\con_{\mathcal{B}\restriction e}(h)$ (i.e., $h$ is $\mathbb{B}_e/g$-generic).
     \end{itemize} 
\end{definition}
\begin{lemma}\label{lemma: properties Pbar}
    Let $\bar{G}$ be a $\mathbb{C}$-generic over $V[G]$. Then, 
    \begin{enumerate}
        \item\label{G star} $G^*=\{p\in\mathbb{B}\mid \langle q,d,g\rangle\in \bar{G}\,\text{for some $q\leq p$, $d\in\mathcal{D}$ and $g$}\}$ is a $\mathbb{B}$-generic filter.
        \item\label{density of E_d}  $E_d=\{\langle p,e,h\rangle\in\mathbb{C}\mid d\sle e \}$ is dense for each $d\in\mathcal{D}^*$.
        \item\label{projection of G star} $\rho_e``G^*=g$ for each $\langle p, e,g\rangle\in \bar{G}$.
    \end{enumerate}
\end{lemma}

\begin{proof}
    (1) It is not hard to show that $G^*$ is a filter. To check genericity  we argue that $D\cap G^\ast\neq\emptyset$ for all dense open sets $D\in\mathcal{P}(\mathbb{B})^V$. First note that $\Bar{D}=\{\langle q,e,h\rangle\in\mathbb{C}\mid q\in D\}$ is dense in $\mathbb{C}$. To see this pick $\langle q,e,h\rangle\in\mathbb{C}$. In particular $q\in\mathbb{B}$ and so, by density of $D$ and since $\rho_e(q)\in h$, there is $q'\leq q$ with $q'\in D\cap \mathbb{B}/h$. It follows that $\langle q',e,h\rangle\leq\langle q,e,h\rangle$ and $\langle q',e,h\rangle\in\Bar{D}$, showing that $\Bar{D}$ is dense in $\mathbb{C}$. Thus, by genericity of $\Bar{G}$, there is $\langle q,e,h\rangle\in\Bar{D}\cap\Bar{G}$. By definition of $\Bar{D}$ and $\Bar{G}$, we deduce that $q\in D\cap G^\ast$.
    
    \smallskip
    
    (2) Suppose that $\langle q_0,e_0,h_0\rangle\in\mathbb{C}$. By directedness of $\mathcal{D}$ there is $e\in\mathcal{D}^*$ such that $e_0,d\sle e$. We are going to find $h$ such that $\langle q_0,e,h\rangle\leq\langle q_0,e_0,h_0\rangle$.  The main obstacle  is to find a $\mathbb{B}_e$-generic $h$ (over $V$) which is also $\mathbb{B}_e/h_0$-generic (over $V[h_0]$). Here the \emph{Interpolation Lemma} (Lemma~\ref{lemma: interpolation}) comes into our rescue. Note that since $\mathcal{P}$ is a weak $(\kappa,\lambda)$-nice $\Sigma$-system the assumptions of this lemma are fulfilled when regarded with respect to 
    \begin{displaymath}
             \begin{tikzcd}
  \mathbb{P} \arrow{r}{\pi_e} \arrow[bend right]{rr}{\pi_{e,e_0}} & \mathbb{P}_e \arrow{r}{\pi_{e,e_0}}  & \mathbb{P}_{e_0}
\end{tikzcd} 
             \begin{tikzcd}
  \mathbb{B} \arrow{r}{\rho_e} \arrow[bend right]{rr}{\rho_{e,e_0}} & \mathbb{B}_e \arrow{r}{\rho_{e,e_0}}  & \mathbb{B}_{e_0}
\end{tikzcd}
    \end{displaymath}
    and the condition $q_0\in \mathbb{B}/h_0$. 
    The interpolation lemma thus gives us a $\mathbb{B}_e/h_0$-generic $h\in\sky_{\mathcal{B}}(G)$  such that $\rho_{e}(q_0)\in h$.  As a result, $\langle q_0,e,h\rangle\in\mathbb{C}$.

    \smallskip

    (3)  Since $G^\ast$ is $\mathbb{B}$-generic, $\rho_e``G^\ast$ is $\mathbb{B}_e$-generic so that it suffices to check $$\rho_e``G^*\s g.$$ Fix $q\in G^*$. By definition, there are $r\leq q$, $d\in\mathcal{D}^*$ and $h$ such that  $\langle r, d, h\rangle\in \bar{G}$. Since $\langle r, d, h\rangle, \langle p, e, g\rangle\in \bar{G}$ we can let $\langle r', d', h'\rangle\in \bar{G}$ witnessing compatibility. If $\langle d', h'\rangle=\langle e,g\rangle$ then $\rho_e(r')\in g$ (because $r'\in\mathbb{B}/h'$) and so $\rho_e(r')\leq  \rho_e(q)\in g$. Alternatively, $e\sle d'$ and $h'$ is $\mathbb{B}_{d'}/g$-generic. In that case we have that $\rho_{d'}(r')\leq \rho_{d'}(q)\in h'$ and so $\rho_{d',e}(\rho_{d'}(q))=\rho_e(q)\in g$. 
   
\end{proof}
\begin{cor}\label{remark after lemma: properties Pbar}\hfill
\begin{enumerate}
    \item For each $d\in \mathcal{D}$, $\rho_d`` G^*$ is $\mathbb{B}_d$-generic. 
    \item For each  $g\in \sky_{\mathcal{B}}(G)$ the condition $\langle \one, d, g\rangle\in \mathbb{C}$ constellates $g$: i.e.,
    $$V[G]\models \langle \one,d,g\rangle\forces_{\mathbb{C}}\check{g}\in \con_{\mathcal{B}}(\dot{G}^*).$$
\end{enumerate}
\end{cor}

Next we show that $\mathbb{C}$ captures all sets in $\mathcal{P}(\kappa)^{V[G]}$. The main technical observation towards this end is the following lemma:

\begin{lemma}\label{D_x is dense in barP}
         $D_x=\{\langle q,e,h\rangle\in\mathbb{C}\mid x\in V[h]\}$ is dense for all $x\in\mathcal{P}(\kappa)^{V[G]}$.
\end{lemma}
\begin{proof}
Let $\langle q_0, e_0,h_0\rangle\in\mathbb{C}$. By the $\kappa$-capturing property of $\mathcal{B}$ there is an index $e_1\in\mathcal{D}^*$ such that $e_0\sle e_1$ and $x\in V[\rho_{e_1}``G]$. Also,  
as shown in Lemma \ref{lemma: properties Pbar} (\ref{density of E_d}), the set $E_{e_1}=\{\langle p,e,g\rangle\in\mathbb{C}\mid e_1\sle e\}$ is dense, so we may let $\langle q_2,e_2,h_2\rangle\in E_{e_1}$ with $\langle q_2,e_2,h_2\rangle\leq \langle q_0,e_0,h_0\rangle$. We have two cases: 

\medskip

        \textbf{\underline{Case $h_2=\rho_{e_2}``G$:}} In that case we are done because $$x\in V[\rho_{e_1}``G]\subseteq V[\rho_{e_2}``G]=V[h_2].$$ 

       \smallskip

        \textbf{\underline{Case $h_2\neq\rho_{e_2}``G$:}} Let $\Bar{p}\in \rho_{e_2}``G\setminus h_2$. This condition must exist as otherwise $h_2\s \rho_{e_2}``G$ and thus, by maximality of generic filters,  $h_2= \rho_{e_2}``G$. 

        \smallskip

        The proof strategy is divided in the following three steps.
        \begin{enumerate}
            \item For some $e_3\in\mathcal{D}^*$ with $e_2\sle e_3$, we find a $\mathbb{B}_{e_3}/h_2$-name $\dot g_2$, and a condition $t\leq\rho_{e_3}(q_2)$ such that $\rho_{e_3,e_2}(t)\in h_2$ and $$t\Vdash_{\mathbb{B}_{e_3}/h_2}``\dot g_2\text{ is }(\mathbb{B}_{e_2})_{\downarrow\Bar{p}}\text{-generic"}.$$
            \item\label{setpthree} Using $\dot g_2$ we construct the complete embedding induced by $\dot{g}_2$, $$\iota\colon  ({\mathbb{B}_{e_2}})_{\downarrow\Bar{p}}\rightarrow (\mathbb{B}_{e_3}/h_2)_{\downarrow t}$$ and appealing to Fact \ref{fact: connection between projections and complete embeddings} we get a projection $$\sigma\colon(\mathbb{B}_{e_3}/h_2)_{\downarrow t}\rightarrow(\mathbb{B}_{e_2})_{\downarrow\Bar{p}}.$$
            \item\label{step3} Apply the \emph{Capturing} part of Lemma \ref{lemma: interpolation} we find a $\mathbb{B}_{e_3}/h_2$-generic filter $h$ such that $t\in h$ and $\rho_{e_2}``G\in V[h]$.
        \end{enumerate} 
        
        Once those three steps are verified we will be done with density of $D_x$:

        \begin{claim}
 $\langle q_2,e_3,h\rangle\leq \langle q_0,e_0,h_0\rangle$ and $\langle q_2,e_3,h\rangle\in D_x$.
        \end{claim}
        \begin{proof}[Proof of claim]
                   First $x\in V[\rho_{e_1}``G]\subseteq V[\rho_{e_2}``G]\s V[h]$ (by the choice of $h$). 
                   Second, since $t\leq\rho_{e_3}(q_2)$ and $t\in h$, it follows that $\langle q_2,e_3,h\rangle$ is a condition in $\mathbb{C}$. To conclude, note that $\langle q_2,e_3,h\rangle\leq \langle q_2,e_2,h_2\rangle\leq \langle q_0,e_0,h_0\rangle.$
        \end{proof}

        \smallskip

        \underline{\textbf{{Step one:}}} Let $\Bar{H}$ be a $\mathbb{C}$-generic filter over $V[G]$ with $\langle q_2,e_2,h_2\rangle\in\Bar{H}$. By Lemma \ref{lemma: properties Pbar}, this induces a $\mathbb{B}$-generic $H^*$ such that $h_2\in\con_{\mathcal{B}}(H^*)$. 
        
      Consider the system of weak projections and its Boolean completion
  \begin{displaymath}
             \begin{tikzcd}
  \mathbb{P} \arrow{r}{\pi_{e_2}}  & \mathbb{P}_{e_2} \arrow{r}{\tau}  & \{\bar p\},
\end{tikzcd} 
             \begin{tikzcd}
  \mathbb{B} \arrow{r}{\rho_{e_2}}  & \mathbb{B}_{e_2} \arrow{r}{\tau^*}  & \{\bar p\}
\end{tikzcd}
    \end{displaymath}
    being $\tau$ the trivial projection sending every condition of $\mathbb{P}_e$ to $\bar{p}$. Let us invoke the Interpolation Lemma (Lemma~\ref{lemma: interpolation}) with respect to this maps, taking also $\{\bar p\}$ as $\{\bar p\}$-generic (in $V[H^*]$) and as condition $p\in \mathbb{B}/\{\bar p\}$ the $\rho_{e_2}$-preimage of $\bar{p}$ (recall that $\bar{p}\in \rho_{e_2}``G$). The lemma thus give us a $\mathbb{B}_{e_2}$-generic $g_2\in \sky_{\mathcal{B}}(H^*)$ such that $\rho_{e_2}(p)=\bar{p}\in g_2$.
       
    Now apply Lemma \ref{lemma: capturing generic filters with a small generic} to $\sky_{\mathcal{B}}(H^*)$ and $e_2$ thus finding $e_3\in \mathcal{D}^*$ such that $e_2\sle e_3$ and $g_2\in \sky_{\mathcal{B}}(\rho_{e_3}``H^*).$
   As the condition $\langle q_2,e_2,h_2\rangle$ is in $\Bar{H}$, Lemma \ref{lemma: properties Pbar} (\ref{projection of G star}) yields $$h_2\in\con_{\mathcal{B}}(\rho_{e_3}``H^*).$$ Thus, $\rho_{e_3}``H^*$ is $\mathbb{B}_{e_3}/h_2$-generic (over $V[h_2]$) and  $$V[h_2][\rho_{e_3}``H^\ast]\models``g_2\text{ is }(\mathbb{B}_{e_2})_{\downarrow \Bar{p}}\text{-generic}\text{"}.$$ 
       
       Let $t\in\rho_{e_3}``H^\ast\subseteq\mathbb{B}_{e_3}/h_2$ and an \emph{open} $\mathbb{B}_{e_3}/h_2$-name $\dot g_2$ such that $$\text{$V[h_2]\models t\forces_{\mathbb{B}_{e_3}/h_2}``\dot g_2\text{ is }(\mathbb{B}_{e_2})_{\downarrow \Bar{p}}\text{-generic}$".}$$

       Recall that a  name $\tau$ is called \emph{open} whenever $\langle \sigma, q\rangle\in \tau$ yields $\langle\sigma, r\rangle\in \tau$ for all $r\leq q$. For any name $\sigma$ there is always an open name $\tau$ with $\one\forces\tau=\sigma.$  

       \smallskip

        \underline{\textbf{{Step two:}}} Since $\dot g_2$ is open we may assume that it is actually a $(\mathbb{B}_{e_3}/h_2)_{\downarrow t}$-name.\footnote{If not define $\dot g'_2=\{\langle\check{u},v\rangle\mid v\leq t\ \wedge\ \langle\check{u},v\rangle\in\dot g_2\}$. By openness $t\forces_{\mathbb{B}_{e_3}/h_2}\dot g_2=\dot g_2'$.} Moreover, as $t\in\rho_{e_3}``H^\ast$ and $q_2\in H^\ast$, we also assume that $t\leq\rho_{e_3}(q_2)$.\footnote{If not, say $\rho_{e_3}(s)\leq t$ for some $s\in H^\ast$. Then there is $s'\leq s,q_2$ with $s'\in H^\ast$. Replace $t$ with $\rho_{e_3}(s')\leq\pi_{e_3}(q_2)$.} Note that $\Bar{p}$ is trivially the maximal element of $(\mathbb{B}_{e_2})_{\downarrow\Bar{p}}$ and thus $$V[h_2]\models t\forces_{\mathbb{B}_{e_3}/h_2}\check{\Bar{p}}\in \dot g_2.$$

        Working for a moment inside $V[h_2]$, let $\iota:(\mathbb{B}_{e_2})_{\downarrow\Bar{p}}\rightarrow(\mathbb{B}_{e_3}/h_2)_{\downarrow t}$ be the complete embedding induced by $\dot g_2$. Namely, $$\iota(r):=\llbracket\check{r}\in\dot g_2\rrbracket =\bigvee\{s\in(\mathbb{B}_{e_3}/h_2)_{\downarrow t}\mid s\Vdash_{(\mathbb{B}_{e_3}/h_2)_{\downarrow t}}\check{r}\in\dot g_2\}\big.$$
       Appealing to Fact~\ref{fact: connection between projections and complete embeddings}  we deduce (inside $V[h_2]$) the existence of a projection $$\sigma\colon(\mathbb{B}_{e_3}/h_2)_{\downarrow t}\rightarrow(\mathbb{B}_{e_2})_{\downarrow\Bar{p}}.$$

        \underline{\textbf{{Step three:}}} Once again the \emph{Interpolation Lemma} (Lemma~\ref{lemma: interpolation}) comes to our rescue -- this time we will be able to use its ``capturing feature'' too thanks to our finding of the projection $\sigma$. Thus, call the lemma with  
         \begin{displaymath}
             \begin{tikzcd}
  \mathbb{P} \arrow{r}{\pi_{e_3}}  \arrow[bend right]{rr}{\pi_{e_2}}  & \mathbb{P}_{e_3} \arrow{r}{\pi_{e_3,e_2}}  & \mathbb{P}_{e_2},
\end{tikzcd} 
             \begin{tikzcd}
  \mathbb{B} \arrow{r}{\rho_{e_3}}\arrow[bend right]{rr}{\rho_{e_2}}  & \mathbb{B}_{e_3} \arrow{r}{\rho_{e_3,e_2}}  & \mathbb{B}_{e_2}
\end{tikzcd}
    \end{displaymath}
   as maps; as a $\mathbb{B}_{e_2}$-generic $h_2\in V[G]$;  as conditions $q_2\in\mathbb{B}/h_2$ (for the interpolation part) and $t\leq\rho_{e_3}(q_2)$ and $\Bar{p}\in\rho_{e_2}``G$ (for the capturing part); as a projection in $V[h_2]$ we let $\sigma\colon(\mathbb{B}_{e_3}/h_2)_{\downarrow t}\rightarrow(\mathbb{B}_{e_2})_{\downarrow\Bar{p}}$.

\smallskip

The lemma gives us a $\mathbb{B}_{e_3}/h_2$-generic filter $h$  such that $t\in h$ and $\rho_{e_2}``G\in V[h]$. This completes the goal of \textbf{Step three} (see (3) in page~\pageref{step3}).
\end{proof}

\begin{lemma}[The constellation lemma]\label{lemma: computing correctly the power set} 
Let $G$ be a $\mathbb{B}$-generic filter and  $\langle p,d,g\rangle\in\mathbb{C}$. 
   There is a $\mathbb{B}$-generic $G^\ast$ with $p\in G^*$, $g\in\con_{\mathcal{B}}(G^*)$ 
    and $$\mathcal{P}(\kappa)^{V[G^\ast]}=\mathcal{P}(\kappa)^{V[G]}.$$
    In particular, for each $g\in \sky_{\mathcal{B}}(G)$ there is a $\mathbb{B}$-generic filter $G^*$ such that $g\in \con_{\mathcal{B}}(G^*)$ and $\mathcal{P}(\kappa)^{V[G^*]}= \mathcal{P}(\kappa)^{V[G]}.$
\end{lemma}

\begin{proof}
Let $\bar{G}$ a $\mathbb{C}$-generic filter (over $V[G]$) with $\langle p,d,g\rangle\in \bar{G}$ and let $G^*$ be its induced $\mathbb{B}$-generic. Applying Corollary~\ref{remark after lemma: properties Pbar} we have that $g\in \con_{\mathcal{B}}(G^*)$ and clearly $p\in G^*$. Let us now show that $G^*$ \emph{captures} $\mathcal{P}(\kappa)^{V[G]}$.
\begin{claim}
    $\mathcal{P}(\kappa)^{V[G^\ast]}\subseteq\mathcal{P}(\kappa)^{V[G]}$.
\end{claim}
\begin{proof}[Proof of claim]
    Let $x\in \mathcal{P}(\kappa)^{V[G^*]}$. Since $\mathcal{B}$ is $\kappa$-capturing (Definition~\ref{def: nice system}) there is $d\sle e$ in $\mathcal{D}^*$ such that $x\in V[\rho_e``G^*]$. Let $\langle q,e,h\rangle\in \bar{G}\cap E_e$ (see Lemma~\ref{lemma: properties Pbar}(\ref{projection of G star})). Then, $\rho_e``G^*=h\in V[G]$, which yields $x\in\mathcal{P}(\kappa)^{V[G]}.$ 
\end{proof}

\begin{claim}
    $\mathcal{P}(\kappa)^{V[G]}\subseteq\mathcal{P}(\kappa)^{V[G^\ast]}$.
\end{claim}
\begin{proof}[Proof of claim]
    Let $x\in \mathcal{P}(\kappa)^{V[G]}$. Since $D_x$ is dense (Lemma~\ref{D_x is dense in barP}) there is $\langle q,e,h\rangle\in \bar{G}\cap D_x$. This  yields $x\in V[h]=V[\rho_e``G^*]$.
\end{proof}
This completes the main claim of the lemma. As for the ``in particular'' one, this follows by applying the first part to the condition $\langle \one, d, g\rangle$.\qedhere

\end{proof}

\subsection{$\Sigma$-Prikry forcings and the $\kappa$-$\psp$}\label{sec: sigma prikry and kappaperfect}
Mimicking Solovay's argument \cite{solovay1970model} (see page~\pageref{sec: direct systems of posets} for an account on it), in this section we show how to construct sets with the $\kappa$-$\psp$ 
in generic extensions by $\Sigma$-Prikry forcings.

\smallskip

Recall that we are still working under Setup~\ref{setup kappapsp} in page~\pageref{setup kappapsp}.
\begin{theorem}\label{the main construction}
 Let $\mathbb{P},\mathbb{Q}$ and $\mathbb{R}$ be forcings satisfying the assumptions of the Interpolation Lemma (Lemma~\ref{lemma: interpolation}). Suppose also that
 
   \begin{displaymath}
       \begin{tikzcd}
  \mathbb{P} \arrow{r}{\pi_1} \arrow[bend right]{rr}{\pi} & \mathbb{Q} \arrow{r}{\pi_{2}}  & \mathbb{R}
\end{tikzcd}
       \begin{tikzcd}
  \ro(\mathbb{P}) \arrow{r}{\rho_1} \arrow[bend right]{rr}{\rho} & \ro(\mathbb{Q}) \arrow{r}{\rho_{2}}  & \ro(\mathbb{R})
\end{tikzcd}
\end{displaymath}
are, respectively, weak projections and their $\mathcal{B}$-mirror projections.

Let $G\s\mathbb{B}$ be a $V$-generic filter, $\tau$ a $\ro(\mathbb{Q})/\rho``G$-name and $p_0\in \ro(\mathbb{Q})/\rho``G$ such that 
$$V[\rho``G]\models \text{$``\one\forces_{{\ro(\mathbb{Q})/\rho``G}}\tau\notin \check{V} \wedge p_0\Vdash_{\ro(\mathbb{Q})/\rho``G} \tau\colon \check{\omega}\rightarrow\check{\kappa}$''}.$$

 Then, the set defined as
    \begin{center}
    $\mathcal{T}:=\{\tau_h \mid h$ is $\ro(\mathbb{Q})/\rho``G$-generic over $V[\rho``G]$ and $p_0\in h\}$, 
    \end{center}
    contains a copy of a $\kappa$-perfect set. More specifically, if $\langle \nu_n\mid n<\omega\rangle$ is a cofinal increasing sequence in $\kappa$ living in $V[G]$ then there is a map $$\textstyle\text{$\iota\colon \prod_{n<\omega}\nu_n\rightarrow {}^\omega\kappa$}$$ that defines a topological embedding
     with $\mathrm{ran}(\iota)\s \mathcal{T}$, so $\mathcal{T}$ satisfies Fact \ref{fact PSP}(\ref{fact PSP-2}).

\end{theorem}

\begin{proof}
    For the sake of readability we denote $\Vdash_{\ro(\mathbb{Q})/\rho``G}$ just by $\Vdash$ as this is the only forcing relation involved in the arguments. As $\cf(\kappa)^{V[G]}=\omega$, we may pick a strictly increasing sequence $\langle\nu_n\mid n<\omega\rangle\in V[G]$ cofinal in $\kappa$. 

    \smallskip

    In $L(V[G]_{\kappa+1})$ we are going to define  a tree  of triples $\langle p_s,q_s,r_s\rangle$ consisting of conditions in  $\ro(\mathbb{Q})/\rho``G$ with indices  $s\in \bigcup_{1\leq n<\omega}\prod_{i<n}\nu_i$ and such that $$\text{$q_{s^\smallfrown\langle\xi\rangle}\leq p_s\leq r_s\leq q_s$ provided $\xi<\nu_{|s|}$.}$$ 
    The idea behind the construction is summarized by the following points:
 \begin{enumerate}
  \item For each index $s$ and distinct  $\xi_1,\xi_2< \nu_{|s|}$, $q_{s^\smallfrown\langle\xi_1\rangle}$ and $q_{s^\smallfrown\langle\xi_2\rangle}$ are going to force a different behavior for the $\ro(\mathbb{Q})/\rho``G$-name $\tau$. For the construction of these conditions  we will use Lemma~\ref{functions in the intermediate model}.
  \item If $x\in\prod_{i<\omega}\nu_i$ is a branch then $\langle p_{x\restriction n}\mid n<\omega\rangle$ generates a $\ro(\mathbb{Q})/\rho``G$-generic filter, $h_x$. For this we use the  $\Sigma$-Prikryness of $\mathbb{Q}$, which is a dense subforcing of $\ro(\mathbb{Q}).$ 
  \item  $r_s$ decides the behavior of $\tau\colon \check{\omega}\rightarrow \check{\kappa}$ up to a certain natural number.
 \end{enumerate}    
At the end of this construction we shall show that the map 
$$\textstyle \iota\colon x\in \prod_{i<\omega}\nu_i \mapsto  \tau_{h_x}\in  \mathcal{T}$$
is a topological embedding; namely, it yields a homeomorphism with $\mathrm{ran}(\iota)$.

 Working in $V[G]$, pick an enumeration $\mathcal{D}:=\langle D_\alpha\mid\alpha<\kappa\rangle$ of the family $$\{D\subseteq\mathbb{Q}\mid D\;\text{is dense open in}\; \mathbb{Q}\}.$$ An enumeration of that length exists because $\lambda$ is inaccessible in $V$, $\mathbb{Q}\in H_\lambda$ and $\lambda$ becomes the successor of $\kappa$ in $V[G]$. 
 Also, since $\mathbb{P}$ is  $\Sigma$-Prikry, 
 apply Lemma~\ref{lemma: goodness} to $\mathbb{P}$ and $\mathcal{D}$, and deduce the existence of a sequence $\langle B_n\mid n<\omega\rangle$ in $V[G]$ consisting of bounded subsets of $\kappa$ such that $\bigcup_{n<\omega}B_n=\kappa$ and $$\text{$\mathcal{D}\restriction B_n\in V$ for all $n<\omega$.}$$

 We start the construction by setting $p_{\emptyset}=p_0$ (recall that $p_0$ was chosen so as to force $``\tau\colon \check{\omega}\rightarrow\check{\kappa}$"). Since $\one_{\ro(\mathbb{Q})}\forces``\tau\notin V[\rho``G]$" Lemma~\ref{functions in the intermediate model} implies that the set of possible decisions about $\tau$ made by conditions $q\leq_{\ro(\mathbb{Q})/\rho``G}p_{\emptyset}$ is of cardinality $\kappa$ in $V[\rho``G]$ (note Lemma~\ref{functions in the intermediate model} is applied to $\ro(\mathbb{Q})/\rho``G$, $\tau$, $p_{\emptyset}$ and  $V[\rho``G]$, where $V[\rho``G]$ is playing the role of the ground model).

 For each $m<\omega$, consider the set of possible decisions for $\tau(\check{m})$
$$\Delta_m(p_\emptyset):=\{\gamma<\kappa \mid  \exists q\leq p_\emptyset\ (q\in\ro(\mathbb{Q})/\rho``G\ \wedge\ q\Vdash\tau(\check{m})=\check{\gamma})\}.$$
There must exist an index $m_{\emptyset}<\omega$ for which there are at least $\nu_0$-many possible decisions about $\tau(m_\emptyset)$; namely, there is $m_\emptyset$ with  $|\Delta_{m_\emptyset}(p_\emptyset)|\geq\nu_0$.

Let $\langle\gamma_{\emptyset,\xi} \mid \xi<\nu_0\rangle$ be an injective enumeration of a set in $[\Delta_{m_\emptyset}(p_\emptyset)]^{\nu_0}$. Then for each $\xi<\nu_0$ choose $q_{\langle\xi\rangle}\leq p_\emptyset$ such that $q_{\langle\xi\rangle}\in\ro(\mathbb{Q})/\rho``G$ and $$q_{\langle\xi\rangle}\Vdash\tau(\check{m}_\emptyset)=\check{\gamma}_{\emptyset,\xi}.$$ 
\begin{claim}
    Without loss of generality we may assume that $q_{\langle\xi\rangle}\in\mathbb{Q}$.
\end{claim}
\begin{proof}[Proof of claim]
     If this is not the case, note that $\mathbb{Q}$ is dense in $\ro(\mathbb{Q})$ 
     and so $\rho_2``\mathbb{Q}_{\downarrow q_{\langle\xi\rangle}}$ (i.e., $\rho_2``(\mathbb{Q}\cap \ro(\mathbb{Q})_{\downarrow q_{\langle\xi\rangle}}))$ is dense below $\rho_2(q_{\langle\xi\rangle})\in\rho``G\subseteq\ro(\mathbb{R})$. In particular, $\rho_2``\mathbb{Q}_{\downarrow q_{\langle\xi\rangle}}\cap\rho``G\neq\emptyset$. Now replace $q_{\langle\xi\rangle}$ with some condition $q'_{\langle\xi\rangle}\leq q_{\langle\xi\rangle}$ satisfying $\rho_2(q'_{\langle\xi\rangle})\in\rho``G$. Clearly, $q'_{\langle\xi\rangle}\in\mathbb{Q}/\rho``G$. 
\end{proof}

The sequence $\langle q_{\langle\xi\rangle}\mid\xi<\nu_0\rangle$ has been constructed in $V[\rho``G]$ but actually: 

\begin{claim}\label{claim: coding the poset in Lpkappa}
     $\langle q_{\langle\xi\rangle}\mid\xi<\nu_0\rangle\in L(V[G]_{\kappa+1})$.
\end{claim}
\begin{proof}[Proof of claim]
This is a consequence of one of our general preliminary lemmas; namely, Lemma \ref{lemma: transitive closure}.
    The sequence $\langle q_{\langle\xi\rangle}\mid\xi<\nu_0\rangle$ is a subset of $\nu_0\times\mathbb{Q}$. But $\lambda$ is inaccessible so that $(\nu_0\times\mathbb{Q})\in V_\lambda=H_\lambda$. In particular, $\mathrm{tcl}(\langle q_{\langle\xi\rangle}\mid\xi<\nu_0\rangle)\subseteq\mathrm{tcl}(\nu_0\times\mathbb{Q})\in (H_\lambda)^{V[G]}$. Finally, use Lemma~\ref{lemma: transitive closure}  to infer that $\mathrm{tcl}(\{\langle q_{\langle\xi\rangle}\mid\xi<\nu_0\rangle\})\in (H_\lambda)^{V[G]}\subseteq L(V[G]_{\kappa+1})$.
\end{proof}

\smallskip

Next extend $q_{\langle\xi\rangle}$ to a condition $r_{\langle\xi\rangle}$ in $\ro(\mathbb{Q})/\rho``G$ so that for each $m\leq m_\emptyset$, $r_{\langle\xi\rangle}$ decides the value of $\tau(m)$. As argued in the previous paragraph, we may assume that $r_{\langle\xi\rangle}\in\mathbb{Q}$. The choice of $p_{\langle\xi\rangle}\in\ro(\mathbb{Q})/\rho``G$ is provided by the following \emph{Interpolation-like} argument. Recall that $B_0$ was a bounded subset of $\kappa$ such that $\mathcal{D}\restriction B_0\in V$ and $\mathcal{D}$  an enumeration of all dense open subsets $D_\alpha\s \mathbb{Q}$ in $V$. In  Claim~\ref{claim: E_0 is dense in Q} we already proved the following:
\begin{claim}
    Define 
    $$E_{0,\xi}=\{p\in\mathbb{Q}_{\downarrow r_{\langle\xi\rangle}}\mid\forall \alpha\in B_0\,\exists n_\alpha<\omega \ \mathbb{Q}^p_{\geq n_\alpha}\subseteq D_\alpha\}.$$ Then $E_{0,\xi}$ is a dense subset of $\ro(\mathbb{Q})_{\downarrow r_{\langle\xi\rangle}}$.\qed
\end{claim}

Since $\rho_2:\ro(\mathbb{Q})\rightarrow\ro(\mathbb{R})$ is a projection, $\rho_2``E_{0,\xi}$ is dense below $\rho_2(r_{\langle\xi\rangle})$ in $\ro(\mathbb{R})$. On the other hand, $\rho_2(r_{\langle\xi\rangle})\in\rho``G$ (recall that $r_{\langle\xi\rangle}\in \ro(\mathbb{Q})/\rho``G$) and so we may pick a condition $$r\in\rho_2``E_{0,\xi}\cap\rho``G.$$ Thus there is $p\in E_{0,\xi}$ such that $\rho_2(p)=r$. Stipulate $p_{\langle\xi\rangle}:=p$. It is clear that $p_{\langle\xi\rangle}\in \ro(\mathbb{Q})/\rho``G$ and by definition of $E_{0,\xi}$,  $p_{\langle\xi\rangle}\in \mathbb{Q}_{\downarrow r_{\langle\xi\rangle}}$.

\medskip

The above completes the construction of the first level of the tree $$\langle \langle p_{\langle \xi\rangle}, q_{\langle \xi\rangle}, r_{\langle \xi\rangle}\rangle\mid \xi<\nu_0\rangle\in{}^{\nu_0}(\mathbb{Q}\times\mathbb{Q}\times\mathbb{Q}).$$

\smallskip

For an arbitrary node $s\in \bigcup_{1\leq n<\omega}\prod_{i<n}\nu_i$ we  construct the next level 
$$\langle \langle p_{s^\smallfrown\langle \xi\rangle}, q_{s^\smallfrown\langle \xi\rangle}, r_{s^\smallfrown\langle \xi\rangle}\rangle\mid \xi<\nu_{|s|}\rangle\in{}^{\nu_{|s|}}(\mathbb{Q}\times\mathbb{Q}\times\mathbb{Q})$$
in the very same fashion. More precisely, we start the process with $p_s\in\ro(\mathbb{Q})/\rho``G$ then find an index $m_s<\omega$ such that $|\Delta_{m_s}(p_s)|\geq \nu_{|s|}$ and obtain the various $r_{s^\smallfrown\langle \xi\rangle}\leq q_{s^\smallfrown\langle \xi\rangle}\leq p_s$ with $r_{s^\smallfrown\langle \xi\rangle}, q_{s^\smallfrown\langle \xi\rangle}\in\mathbb{Q}$ and $r_{s^\smallfrown\langle \xi\rangle}$ forcing a value for $\tau(i)$ for every $i<\max\{m_s,|s|+1\}$. Then for each $\xi<\nu_{|s|}$
$$E_{|s|,\xi}=\{p\in\mathbb{Q}_{\downarrow r_{s^\smallfrown\langle\xi\rangle}}\mid\,\ell(p)>\ell(p_s)\,\wedge\, \forall \alpha\in B_{|s|}\,\exists n_\alpha<\omega \ \mathbb{Q}^p_{\geq n_\alpha}\subseteq D_\alpha\},$$  is a dense subset of $\ro(\mathbb{Q})_{\downarrow r_{s^\smallfrown\langle\xi\rangle}}$. Finally, we get $p_{s^\smallfrown\langle\xi\rangle}$ in $(\ro(\mathbb{Q})/\rho``G)\cap E_{|s|,\xi}$.

\smallskip

For each branch $x\in\prod_{i<\omega}\nu_i$ in the tree, let $h_x$ be the upwards closure of the sequence $\{ p_{x\restriction n}\mid n<\omega\}$ in $\ro(\mathbb{Q})$. That is, we let 
$$h_x:=\{b\in\ro(\mathbb{Q})\mid\exists n<\omega \ p_{x\restriction n}\leq b\}.$$
By Claim \ref{claim: h is ro(Q)/g-generic}, $h_x$ is $\ro(\mathbb{Q})/\rho``G$-generic over $V[\rho``G]$.

\smallskip

To conclude the proof it remains to show that 
$$\textstyle \iota\colon x\in \prod_{i<\omega}\nu_i \mapsto  \tau_{h_x}\in  \mathcal{T}$$
is a (topological) embedding. 
\begin{claim}
    $\iota$ is an embedding.
\end{claim}
\begin{proof}[Proof of claim]
    The injectivity is granted by the the $q_s$'s. Let $x,y\in\prod_{n<\omega}\nu_n$ distinct and let $n<\omega$ be maximal such that $x\restriction n=y\restriction n$. By construction, 
    $q_{x\restriction n^\smallfrown\langle x_n\rangle}$ and $q_{x\restriction n^\smallfrown\langle y_n\rangle}$ are conditions below $p_{x\restriction n}$ satisfying $$\text{$q_{x\restriction n^\smallfrown\langle x_n\rangle}\Vdash\tau(\check{m}_{x\restriction n})=\check{\gamma}_{x\restriction n, x_n}$ and $q_{x\restriction n^\smallfrown\langle y_n\rangle}\Vdash\tau(\check{m}_{x\restriction n})=\check{\gamma}_{y\restriction n, y_n}$.}$$ The construction was performed so that  $x_n\neq y_n$ implies $\gamma_{x\restriction n, x_n}\neq\gamma_{x\restriction n, y_n}$. On the other hand, $q_{x\restriction (n+1)}\in h_x$ and $q_{y\restriction (n+1)}\in h_y$ because $$\text{$p_{x\restriction (n+1)}\leq q_{x\restriction (n+1)}$ and $p_{y\restriction (n+1)}\leq q_{y\restriction (n+1)}$}.$$ So $\tau_{h_x}(m_{x\restriction n})=\gamma_{x\restriction n, x_n}\neq\gamma_{x\restriction n, y_n}=\tau_{h_y}(m_{x\restriction n})$. 

    \smallskip

    Now we deal with the continuity of $\iota^{-1}\colon \tau_{h_x}\mapsto x$ (the proof that $\iota$ is continuous is also similar): for each $x,y\in\prod_{i<\omega}\nu_i$, if $n$ is maximal such that $x\restriction n=y\restriction n$, consider $r_{x\restriction n}$. Since $p_{x\restriction(n+1)}, p_{y\restriction(n+1)}\leq r_{x\restriction n}$, $r_{x\restriction n}$ is both in $h_x$ and in $h_y$. What is more, $r_{x\restriction n}$ forces a value for $\tau(i)$ for every $$i<\max\{m_{x\restriction (n-1)},|x\restriction (n-1)|+1\}=\max\{m_{x\restriction (n-1)},n\}.$$ Therefore, $\tau_{h_x}(i)=\tau_{h_y}(i)$, whenever $i<\max\{m_{x\restriction (n-1)},n\}$.
\end{proof}
The above completes the proof of the theorem.
\end{proof}
\begin{remark}
    As a bonus result, note that the $\kappa$-perfect set $\mathrm{ran}(\iota)=\{\tau_{h_x}\mid x\in \prod_{n<\omega}\nu_n\}$ in fact belongs to $L(V[G]_{\kappa+1})$. The reason is that the generics $h_x$ belong to $L(V[G]_{\kappa+1})$ (by an argument analogous to Claim~\ref{claim: coding the poset in Lpkappa}) and the map assigning  $x\mapsto\tau_{h_x}$ can be coded in $L(V[G]_{\kappa+1}).$
\end{remark}
We are now in conditions to prove our first abstract theorem on the $\kappa$-$\psp$ in $\Sigma$-Prikry generic extensions:
\begin{theorem}\label{main theorem 1 in terms of nice systems}
    If $\mathcal{P}=\langle \mathbb{P}_e, \pi_{e,d}\colon\mathbb{P}_e\rightarrow\mathbb{P}_d\mid e,d\in\mathcal{D}\,\wedge\, d\sle e\rangle$ is a weak $(\kappa,\lambda)$-nice $\Sigma$-system 
    then
   $$\text{$\one\forces_{\mathbb{P}}``\forall X\in \mathcal{P}(^\omega\kappa)\cap L(V_{\kappa+1})$\,\text{($X$ has the $\kappa$-$\psp$)''}}.$$ 
\end{theorem}
\begin{proof}
Let $\mathcal{B}=\langle \mathbb{B}_e, \rho_{e,d}\colon\mathbb{B}_e\rightarrow\mathbb{B}_d\mid e,d\in\mathcal{D}\,\wedge\, d\sle e\rangle$ be the Boolean completion of the system $\mathcal{P}$ (see Definition~\ref{def: boolean completion of a system}) and recall that $\mathcal{B}$ is a $(\kappa,\lambda)$-nice system (i.e., the $\rho_{e,d}$ now are projections but $\mathbb{B}_e$ are not $\Sigma$-Prikry).

Let $G$ be a $\mathbb{B}$-generic. 
Unless otherwise stated,  we shall be working in $V[G]$. 
Fix $\langle \nu_n\mid n<\omega\rangle\in V[G]$ an increasing cofinal sequence in $\kappa$. This is possible in that the  definition of $(\kappa,\lambda)$-niceness requires $\one\Vdash_{\mathbb{P}}``\cf(\kappa)=\omega$". 

\smallskip

Fix $A\in \mathcal{P}({}^\omega\kappa)\cap L(V[G]_{\kappa+1})$  with $|A|^{V[G]}>\kappa.$  Let $\varphi(x,y_0,\dots, y_n)$ be a formula in the language of set theory and parameters $$\text{$a_0,\dots, a_n\in \mathcal{P}(\kappa)^{V[G]}\cup \ord$}$$ which together with $\varphi(x,y_0,\dots, y_n)$ define $A$; to wit,
    $$A=\{x\in{}^\omega\kappa\mid V[G]\models \varphi(x,a_0,\dots,a_n)\}.$$

   Our formula $\varphi(x,y_0,\dots, y_n)$ has the pleasant property of being  absolute between inner models sharing the same $\mathcal{P}(\kappa)$ (equivalently, $V_{\kappa+1}$). Indeed, there is a formula $\psi(x,y_0,\dots, y_n)$ in the language of set theory such that $$V[G]\models \varphi(x,a_0,\dots, a_n)\;\Leftrightarrow\;V[G]\models``\text{$L(V_{\kappa+1})\models \psi(x,a_0,\dots,  a_n)$''}.$$

 In particular, if $G^*$ is another $V$-generic filter for the complete Boolean algebra $\mathbb{B}$ such that $\mathcal{P}(\kappa)^{V[G]}=\mathcal{P}(\kappa)^{V[G^*]}$ then 
   \begin{equation}\label{eq: absoluteness}
      \tag{$\mathcal{A}$}V[G]\models \varphi(x,a_0,\dots, a_n)\;\Leftrightarrow\;V[G^*]\models \varphi(x,a_0,\dots, a_n).  
   \end{equation}
   We will obtain this alternative generic $G^*$ at the end of our argument employing to this effect the \emph{Constellation Lemma} (Lemma~\ref{lemma: computing correctly the power set}).

   \smallskip

   As $\mathcal{B}$ is a $(\kappa,\lambda)$-nice system, we may appeal to its $\kappa$-\emph{capturing property} to infer the existence of an index $d\in \mathcal{D}^*$ such that $$\{a_0,\dots, a_n\}\in V[\rho_d``G].$$ 

    \smallskip

    The proof is divided into two steps: First, we show that there is a sequence $b\in A$ that is not in $V[\rho_d``G]$. Since this sequence can be coded as a subset of $\kappa$ the capturing property of $\mathcal{B}$ allows us to pick an index $e$ with  $d\sle e\in \mathcal{D}^*$ such that $b\in V[\rho_e``G]$. Thus, there is a condition $q\in \mathbb{B}_e/\rho_d``G$ forcing  $b$ to be a function $\dot{b}\colon \check{\omega}\rightarrow \check{\kappa}$. This would enable us to invoke Theorem~\ref{the main construction} with respect to  the posets $\mathbb{P}$, $\mathbb{P}_e$, $\mathbb{P}_d$, the $\mathbb{B}$-generic $G$, the $\mathbb{B}_e/\rho_d``G$-name $\dot{b}$ and the condition $q\in\mathbb{B}_e/\rho_d``G$. Thus, we will obtain a $\kappa$-perfect set\footnote{In fact, this $\kappa$-perfect set will be a member of $L(V[G]_{\kappa+1})$.} 
    $$\textstyle P:=\{b_{h_x}\mid x\in \prod_{n<\omega}\nu_n\},\footnote{For readability, $b_{h_x}$ stands for $\dot{b}_{h_x}$.}\label{the perfect set}$$ where the $h_x$'s are going to be $\mathbb{B}_e/\rho_d``G$-generic filters over $V[\rho_d``G]$ in $L(V[G]_{\kappa+1})$ containing $q$. In the second step of the proof we will show that $P\s A$. For this we will make a crucial use of Lemma~\ref{lemma: computing correctly the power set} (\emph{Constellation}).

    \medskip

   \underline{\textbf{{Step one:}}} By $\lambda$-boundedness of $\mathcal{B}$, 
   Lemma~\ref{lemma: transitive closure} yields $$({}^\omega\kappa)^{V[\rho_d``G]}\in L(V[G]_{\kappa+1})$$ Also, $|({}^\omega\kappa)^{V[\rho_d``G]}|^{V[G]}=\kappa$ as $\lambda$ is collapsed to be the successor of $\kappa$ by $\mathbb{B}$. Let $f\colon \kappa\rightarrow ({}^\omega\kappa)^{V[\rho_d``G]}$ be a bijection in $V[G]$. Since this can be coded as a member of $\mathcal{P}(\kappa)^{V[G]}$ it follows that $f\in L(V[G]_{\kappa+1})$ and thus $$L(V[G]_{\kappa+1})\models ``|({}^\omega\kappa)^{V[\rho_d``G]}|=\kappa\text{"}.$$ 
   Since $A$ was chosen with cardinality ${>}\kappa$ we can certainly pick $b\in A$ not in $({}^\omega\kappa)^{V[\rho_d``G]}.$ Once again, since $b$ can be coded (within $V[G]$) as a subset of $\kappa$ the $\kappa$-capturing property of $\mathcal{B}$ yields $d\sle e$ in $\mathcal{D}^*$ such that $b\in V[\rho_e``G]$. 

   \smallskip
   
   Since $b$ is taken from $A$ we have:
   \begin{equation}\label{eq 1: the function b is in A}
        V[G]\models \varphi(b,a_0,\dots,a_n).
   \end{equation}

Also, since $\rho_e\colon \mathbb{B}\rightarrow \mathbb{B}_e$ and $\rho_{e,d}\colon \mathbb{B}_e\rightarrow \mathbb{B}_d$ are projections,
   \begin{equation}\label{eq 2: intermediate models via projection}
       \text{$V[G]=V[\rho_e``G][G]$ and $V[\rho_e``G]=V[\rho_d``G][\rho_e``G]$}.
   \end{equation}

   \underline{\textbf{Step two:}} Combining the previous equations, we deduce the existence of a condition $p\in G$ such that 
   $$V[\rho_e``G]\models ``p\forces_{\mathbb{B}/\rho_e``G} \text{$\varphi(\check{b},\check{a}_0,\dots,\check{a}_n)$''.}$$
   Denote by $\Phi(p,\mathbb{B}, \rho_e``G, b, a_0,\dots, a_n)$ the  formula $``p\forces_{\mathbb{B}/\rho_e``G} \text{$\varphi(\check{b},\check{a}_0,\dots,\check{a}_n)$}$''.
   
   By the second part of equation~\eqref{eq 2: intermediate models via projection} there is $q\in \rho_e``G$ such that 
   \begin{equation}\label{eq 3: q forces the formula Phi}
       V[\rho_d``G]\models ``\text{$q\forces_{\mathbb{B}_e/\rho_d``G}\Phi(\check{p},\check{\mathbb{B}}, \dot{G}, \dot{b}, \check{a}_0,\dots, \check{a}_n)$''}.
   \end{equation}
  
  Without loss of generality we may assume that $q=\rho_e(p')$ for some condition $p'$ in $G$ below $p$. Moreover, since $b\in V[\rho_e``G]$, by extending $q$ (inside $\rho_e``G$) if necessary we may assume that $$q\forces_{\mathbb{B}_e/\rho_d``G} \dot{b}\colon \check{\omega}\rightarrow\check{\kappa}.$$
 
  Equation~\eqref{eq 3: q forces the formula Phi} tantamouts to 
   $$V[\rho_d``G][h]\models ``\text{$\Phi(\check{p},{\mathbb{B}}, h, \dot{b}_h, {a}_0,\dots, {a}_n)$''};$$
   provided  $h$ is a $\mathbb{B}_e/\rho_d``G$-generic filter (over $V[\rho_d``G]$) containing $q$.
   
   Equivalently,
   \begin{equation}\label{eq 4: forcing over V[h]}
     V[\rho_d``G][h]\models ``p\forces_{\mathbb{B}/h} \text{$\varphi(\check{{b}}_h,\check{a}_0,\dots,\check{a}_n)$''.}
   \end{equation}

\smallskip

   Since $b\notin V[\rho_d``G]$ (this was the conclusion of \textbf{Step one}) we can invoke Theorem~\ref{the main construction}, applied to $\pi_d=\pi_{e,d}\circ\pi_e\colon \mathbb{P}\xrightarrow{\pi_e}\mathbb{P}_e\xrightarrow{\pi_{e,d}}\mathbb{P}_d$, $G$, $\dot{b}$ and $q$, to infer that the set $P:=\{b_{h_x}\mid x\in \prod_{n<\omega}\nu_n\}\in L(V[G]_{\kappa+1})$ is $\kappa$-perfect. 
   
   \smallskip
   
      To finish the proof, we have to verity that $P\subseteq A$. Recall that by construction the $h_x$'s were $\mathbb{B}_e/\rho_d``G$-generics over $V[\rho_d``G]$ with $q\in h_x$. By equation~\eqref{eq 4: forcing over V[h]} above we may infer that for each $x\in \prod_{n<\omega}\nu_n$,
   $$V[h_x]\models``p\forces_{\mathbb{B}/h_x} \text{$\varphi(\check{{b}}_{h_x},\check{a}_0,\dots,\check{a}_n)$''}$$
   (note that here we implicitly used that $V[\rho_d``G][h_x]=V[h_x]$).

   Next apply the \emph{Constellation Lemma} (i.e., Lemma~\ref{lemma: computing correctly the power set})  with respect to the condition $\langle p,e, h_x\rangle$ in the \emph{constellating poset} $\mathbb{C}$ of Definition~\ref{def: fancy poset}.\footnote{Recall that we have intentionally assumed that $q\leq\rho_e(p)$ and so $p\in\mathbb{B}/h_x$.}

   This lemma gives us a $\mathbb{B}$-generic filter $G_x$ such that $p\in G_x$ and that \emph{constellates} $h_x$ -- i.e., $h_x\in\con_{\mathcal{B}}(G_x)$ or, equivalently, $G_x$ is $\mathbb{B}/h_x$-generic. In addition, $\mathcal{P}(\kappa)^{V[G_x]}=\mathcal{P}(\kappa)^{V[G]}$. Therefore, by $p\in G_x$ we deduce that $$V[h_x][G_x]\models \varphi(b_{h_x}, a_0,\dots, a_{n})$$
   and since $V[h_x][G_x]$ (equivalently, $V[G_x]$) knows about the full power set $\mathcal{P}(\kappa)^{V[G]}$ we conclude that $\varphi(b_{h_x}, a_0,\dots, a_n)$ also holds in $V[G]$ (by virtue of equation~\eqref{eq: absoluteness} above and the explanations preceding it). This in turn  is equivalent to saying that $b_{h_x}\in A$. All in all we have shown that
   $P\s A$ where $P$  is the $\kappa$-perfect set defined in page~\pageref{the perfect set}. 
\end{proof}

\smallskip

\begin{cor}
    Let $\mathcal{P}$ be a weak $(\kappa,\lambda)$-nice system. Then $$\one\Vdash_{\mathbb{P}}``L(V_{\kappa+1})\models\neg\ac+\dc_\kappa".$$
\end{cor}
By absoluteness of the $\kappa$-$\psp$ we get:
\begin{cor}\label{main thm for all Polish spaces}
   Assume there is a $(\kappa,\lambda)$-nice system. Then there is a model of $\zf+\dc_\kappa$ satisfying the following property:
    $$\text{``For all $\kappa$-Polish space $\mathcal{X}$ all subsets of $\mathcal{X}$ have the $\kappa$-\psp"}.$$

\end{cor}
\begin{proof}
   Let $G$ be $\mathbb{P}$-generic. We prove that $L(V[G]_{\kappa+1})$ is the desired model.
   
   Work in $V[G]$. We already noticed that $L(V_{\kappa+1})$ is a model of $\zf+\dc_\kappa$.  Let $\mathcal{X}$ be a $\kappa$-Polish space in $L(V_{\kappa+1})$. In \cite[Section 3.3]{Stone} it is proven that there is a closed set $F\subseteq{}^\omega\kappa$ and a continuous bijection $f\colon F\to X$. So let $A\subseteq\mathcal{X}$, $A\in L(V_{\kappa+1})$ and call $B=f^{-1}(A)$. In particular, $B\s {}^\omega\kappa$ and $B\in L(V_{\kappa+1})$, so by Theorem~\ref{main theorem 1 in terms of nice systems} and absoluteness $B$ has the $\kappa$-$\psp$ in $L(V_{\kappa+1})$. 
  Now work in $L(V_{\kappa+1})$. If $|B|\leq\kappa$, since $f$ is a bijection also $|A|\leq\kappa$. If there exists an embedding $\iota$ from ${}^{\lambda} 2$ to $B$, then $f\circ \iota$ is a continuous injection from ${}^\lambda 2$ to $A$, so $A$ satisfies Fact \ref{fact PSP}(\ref{fact PSP-3}). Therefore we proved that $A$ has the $\kappa$-$\psp$ in $L(V_{\kappa+1})$ for all $A\subseteq\mathcal{X}$ in $L(V_{\kappa+1})$.
\end{proof}

\begin{cor}
 Assume there is a $(\kappa,\lambda)$-nice system. Then there is a model of $\zfc$ such that all the $\kappa$-projective subsets of $\mathcal{C}$ have the $\kappa$-$\psp$, where $\mathcal{C}$ is any space in Example~\ref{example: kappa polish spaces}.
\end{cor}
\begin{proof}
  Note that all the spaces in Example~\ref{example: kappa polish spaces} are definable over $V_{\kappa+1}$, therefore their $\kappa$-projective subsets are exactly their subsets in $L_1(V_{\kappa+1})$. But then if $G$ is $\mathbb{P}$-generic, all the $\kappa$-projective subsets of $\mathcal{C}$ in $V[G]$ belong to $L_1(V[G]_{\kappa+1})$, and by Theorem~\ref{main theorem 1 in terms of nice systems} the corollary follows.
\end{proof}

\subsection{$\Sigma$-Prikry forcings and the $\vec{\mathcal{U}}$-Baire property}\label{sec: Sigma Prikry and BP}
In this section we strengthen  the blanket assumptions of Setup~\ref{setup kappapsp} and assume that:

\begin{setup}\label{setup: baire property}
    $\Sigma=\langle\kappa_n\mid n<\omega\rangle$  is a strictly increasing sequence  of measurable cardinals with limit $\kappa$. This fact will be witnessed by a sequence of normal (uniform) ultrafilters that we denote by $\vec{\mathcal{U}}=\langle \mathcal{U}_n\mid n<\omega\rangle$.  As in Setup~\ref{setup kappapsp}, we also require $\lambda$ to be an inaccessible cardinal above $\kappa:=\sup(\Sigma).$ 
\end{setup}

 Recall that we denote by $\mathbb{P}(\vec{\mathcal{U}})$ the \emph{Diagonal Prikry Forcing relative to $\vec{\mathcal{U}}$} (see Definition~\ref{def: diagonal prikry}). The following is a slight tweak of the notion of $(\kappa,\lambda)$-nice $\Sigma$-system (see Definition~\ref{def: nice system}) accommodating $\mathbb{P}(\vec{\mathcal{U}})$:

\begin{definition}
    
    A (weak) $\Sigma$-system $\mathcal{P}$ is said to be \emph{$(\vec{\mathcal{U}}, \lambda)$-nice} if $\mathcal{P}$ is  $(\kappa,\lambda)$-nice and $\mathbb{P}({\Vec{\mathcal{U}}})$  equals $\mathbb{P}_d$ for some index $d\in\mathcal{D}^*$.\footnote{Note that $\vec{\mathcal{U}}$ in fact codes the sequence $\langle \kappa_n\mid n<\omega\rangle$ and therefore  $\kappa$ too. As a result the terminology \emph{$``\mathcal{P}$ is $(\vec{\mathcal{U}}, \lambda)$-nice''} is unambiguous. } 
\end{definition}
And with this at hand we can prove our abstract theorem about the $\UBP$:

\begin{theorem}\label{main theorem 2: BP for all subsets}
    Let $\mathcal{P}=\langle \mathbb{P}_e, \pi_{e,d}\colon\mathbb{P}_e\rightarrow\mathbb{P}_d\mid e,d\in\mathcal{D}\,\wedge\, d\sle e\rangle$ be a weak $(\Vec{\mathcal{U}},\lambda)$-nice $\Sigma$-system (say with $\mathbb{P}_d=\mathbb{P}(\vec{\mathcal{U}})$).   
    Then $$\text{$\textstyle \one\forces_{\mathbb{P}}``\forall A\in \mathcal{P}(\prod_{n<\omega}\kappa_n)\cap L(V[\dot{G}]_{\kappa+1})$\,\text{($A$ has the $\Vec{\mathcal{U}}$-\bp)''}}.\footnote{Note that $\vec{\mathcal{U}}$ remains a sequence of measures in the generic extension by $\mathbb{P}$ for by $\Sigma$-Prikryness the latter does not add bounded subsets to $\kappa$.}$$ 
\end{theorem}
\begin{proof}
    Let $\mathcal{B}=\langle \mathbb{B}_e, \rho_{e,d}\colon\mathbb{B}_e\rightarrow\mathbb{B}_d\mid e,d\in\mathcal{D}\,\wedge\, d\sle e\rangle$ be the Boolean completion of $\mathcal{P}$ and recall that $\mathcal{B}$ is a $(\kappa,\lambda)$-nice system.  Let $G$ be $\mathbb{B}$-generic and unless otherwise stated let us work in $V[G]$. 
    
    Recall that $$\textstyle C(\Sigma):=(\prod_{n<\omega}\kappa_n)^{V[G]}$$ is a $\kappa$-Polish space with respect to the product topology (see Example~\ref{example: kappa polish spaces}) and that the $\UBP$ is formulated in the accessory $\UEP$-topology $\mathcal{T}_{\UEP}$.

    \smallskip

    Since we will be moving from generics for a poset to generics on its corresponding regular open algebra the following convention becomes handy:

\smallskip

\textbf{Convention:} If $H\s \mathbb{Q}$ is generic for a poset, by $\mathcal{B}(H)$ we denote its upwards closure in $\mathcal{B}(\mathbb{Q})$; namely, $\ro(H):=\{b\in\ro(\mathbb{Q})\mid \exists q\in H\, q\leq b\}.$

\medskip

    Let $A\in L(V[G]_{\kappa+1})\cap \mathcal{P}(C(\Sigma))^{V[G]}$. Then there is a first-order formula $\varphi(x,y_0,\dots,y_n)$  and  parameters $a_0,\dots, a_n\in \mathcal{P}(\kappa)^{V[G]}\cup \ord$ such that $$\textstyle A=\{x\in C(\Sigma)\mid V[G]\vDash\varphi(x,a_0,\dots,a_n)\}.$$ 
    By coding $\langle a_0,\dots, a_n\rangle$ if necessary we may assume that there is a unique parameter $a\in \mathcal{P}(\kappa)^{V[G]}$ (this will enhance  readability).  Using the $\kappa$-capturing property of $\mathcal{B}$ there is an index $d\sle e_0\in\mathcal{D}^*$ such that $a\in V[\rho_{e_0}``G]$. 
    
  Recalling that $\mathbb{P}_d$ is  the diagonal Prikry forcing  and we defined $$C:=\textstyle \{x\in C(\Sigma)\mid \mathcal{F}_x \text{ is }\mathbb{P}_d\text{-generic over } V\},$$
where $\mathcal{F}_x$ is the filter \emph{generated by $x$}; more pedantically,  $$\mathcal{F}_x=\{p\in \mathbb{P}(\vec{\mathcal{U}})\mid s^p\sqsubseteq x\,\wedge\, x(n)\in A^p_n\; \text{for all $n\geq |s^p|$}\}.$$
We already proved in  Lemma \ref{lemma: C is comeager} that $C$ 
is $\kappa$-comeager in the Ellentuck-Prikry $\vec{\mathcal{U}}$ topology and 
as result $\textstyle C(\Sigma)\setminus C$ is $\kappa$-meager. To prove that $A$ has the $\vec{\mathcal{U}}$-$\bp$ we have to work with $C(a)$, a seemingly refinement of $C$.

\smallskip

Let $C(a)$ be the collection of $x\in C$ for which there is $e\in \mathcal{D}^*$ with $e_0\sle e$, a condition $p\in \rho_e``G$, a projection $\sigma\colon (\mathbb{B}_e/\rho_{e_0}``G)_{\downarrow p}\rightarrow{\mathbb{B}_d}_{\downarrow\sigma(p)}$  in $V[\rho_{e_0}``G]$ and a $(\mathbb{B}_e/\rho_{e_0}``G)_{\downarrow p}$-generic $h\in V[G]$ such that 
$\sigma``h=\ro(\mathcal{F}_x)$.
\begin{claim}
    $C=C(a)$. In particular, $C(a)$ is $\kappa$-comeager in $C(\Sigma)$.
\end{claim}
\begin{proof}[Proof of claim]
     Let $x\in C$. By definition $\mathcal{F}_x$ is $\mathbb{P}_d$-generic and belongs to $V[G]$ (i.e., $\mathcal{F}_x\in \sky_{\mathcal{B}}(G)$). By Lemma~\ref{lemma: capturing generic filters with a small generic} there is $e\in \mathcal{D}^*$, with $e_0\prec e$, such that $\mathcal{F}_x\in \sky_{\mathcal{B}\upharpoonright e}(\rho_e``G)$. 
     Let a $(\mathbb{B}_e/\rho_{e_0}``G)$-name $\dot{\mathcal{F}}_x$ such that $\mathcal{F}_x=(\dot{\mathcal{F}}_x)_{\rho_e``G}$ and  $p\in \rho_e``G$  a condition such that
    $V[\rho_{e_0}``G]\models ``p\forces_{(\mathbb{B}_e/\rho_{e_0}``G)}\text{``$\dot{\mathcal{F}}_x$ is $\mathbb{P}_d$-generic"}$.  
     Let  $\sigma\colon(\mathbb{B}_e/\rho_{e_0}``G)_{\downarrow p}\rightarrow{\mathbb{B}_d}_{\downarrow \sigma(p)}$ be the projection associated to $\dot{\mathcal{F}}_x$ and $p$: $$q\mapsto \bigwedge \{r\in \mathbb{P}_d\mid q\forces_{(\mathbb{B}_e/\rho_{e_0}``G)}\check{r}\in \dot{\mathcal{F}}_x\}.$$
   By Fact~\ref{fact: projection to the rescue}  this is indeed a projection and $\sigma``(\rho_{e}``G)_{\downarrow p}=\ro(\mathcal{F}_x)$. All in all we have showed that $x\in C(a)$, which yields $C\s C(a).$
\end{proof}
Let us set some additional notation. We define the set of \emph{relevant projections}   $R(a)$ and \emph{relevant conditions} $P(a,\sigma)$ as follows:

\begin{itemize}
    \item[$(\aleph)$] Let $R(a)$ be the set of all projections  $\sigma\colon (\mathbb{B}_e/\rho_{e_0}``G)_{\downarrow p}\rightarrow {\mathbb{B}_d}_{\downarrow\sigma(p)}$ such that, for some $x\in C(\Sigma)$, $\sigma$ witnesses $x\in C(a)$.
    \item[$(\beth)$]\label{item: relevant conditions} For each $\sigma\in R(a)$, by  $P(a,\sigma)$ we denote the set of all conditions $q\in (\mathbb{B}_e/ \rho_{e_0}``G)_{\downarrow p}$ such that, for some $v\in\mathbb{B}$, $\rho_e(v)=q$\footnote{Note that for each $q\in (\mathbb{B}_e/ \rho_{e_0}``G)_{\downarrow p}$ we do not lose generality by assuming $q\in\rho_e``\mathbb{B}$. If not, note that $(\rho_e``\mathbb{B})_{\downarrow q}$ is dense below $q$, and so $\rho_{e,e_0}``(\rho_e``\mathbb{B})_{\downarrow 
 q}$ is dense below $\rho_{e,e_0}(q)$. On the other hand, $\rho_{e,e_0}(q)\in\rho_{e_0}``G$, and so $\rho_{e,e_0}``(\rho_e``\mathbb{B})_{\downarrow 
 q}\cap\rho_{e_0}``G\neq\emptyset$. Let $u\in\rho_{e,e_0}``(\rho_e``\mathbb{B})_{\downarrow 
 q}\cap\rho_{e_0}``G$. Thus $u=\rho_{e,e_0}(v_0)$ for some $v_0\in(\rho_e``\mathbb{B})_{\downarrow 
 q}$, and $v_0=\rho_{e}(v)$ for some $v\in\mathbb{B}$. Now replace $q$ with $v_0\leq q$, and observe that $v_0\in(\mathbb{B}_e/\rho_{e_0}``G)_{\downarrow p}$. } and $$V[\rho_{e_0}``G]\models``\text{$q\forces_{(\mathbb{B}_e/ \rho_{e_0}``G)_{\downarrow p}}\Phi(\check{v},\check{\mathbb{B}}, \dot{g},\check{\sigma}``\dot{g},\check{a})$''}, $$

 where we denoted
 $$\Phi(v,\mathbb{B}, g, \sigma``g, a)\equiv ``v\forces_{\mathbb{B}/g} \text{$\varphi(\check{\tau}_{\sigma``g},\check{a})$''}.$$
 Here $\tau_{\sigma``g}$ denotes the $\mathbb{P}(\vec{\mathcal{U}})$-generic sequence generated by $\sigma``\dot{g}$.
 
\end{itemize}

The definitions of $R(a)$ and $P(a,\sigma)$ have been posed so as to be in conditions to use the \emph{Capturing part} of the \emph{Interpolation Lemma} (Lemma~\ref{lemma: interpolation}). This will become more transparent at a due time.

   Finally, we define an open set in the $\vec{\mathcal{U}}$-Ellentuck-Prikry topology by
   $$O:=\textstyle \bigcup_{\sigma\in R(a)}\bigcup_{q\in P(a,\sigma)}\bigcup_{r\in (\mathbb{P}_d)_{ \downarrow \sigma(q)}}N_r$$ where  $N_r$ is the basic open neighborhood defined as
   $$N_{r}:=\{x\in C(\Sigma)\mid r\in \mathcal{F}_x\}.$$
     For each $\sigma\in R(a)$ and $q\in P(a,\sigma)$, set
$$N_{\sigma(q)}:=\{x\in C(\Sigma)\mid \sigma(q)\in \ro(\mathcal{F}_x)\}.$$

\begin{claim}
 For $\sigma\in R(a)$ and $q\in P(a,\sigma)$, $\bigcup_{r\in (\mathbb{P}_d)_{ \downarrow \sigma(q)}}N_r=N_{\sigma(q)}$. 
\end{claim}
\begin{proof}[Proof of claim]
    Suppose that $x$ belongs to the left-hand-side. Then $r\in \mathcal{F}_x$ and so $\sigma(q)\in \ro(\mathcal{F}_x)$. Conversely, let $x\in N_{\sigma(q)}$. By definition, $\sigma(q)\in \ro(\mathcal{F}_x)$ which means that for some $r'\in \mathcal{F}_x$, $r'\leq \sigma(q)$. But then $x\in N_{r'}$.
\end{proof}
   
We claim that $A\Delta O$ is $\kappa$-meager in  $\mathcal{T}_{\UEP}$. By $\kappa$-comeagerness of $C(a)$, it is enough to prove that $A\Delta O\subseteq C(\Sigma)\setminus C(a)$, which amounts to saying that for each $x\in C(a)$, $x\in A$ if and only if $x\in O$. Bear in mind that for every sequence $x\in C(a)$ the corresponding filter $\mathcal{F}_x$ is $\mathbb{P}_d$-generic over $V$.
\begin{claim}\label{claim: A included in O}
    Let $x\in C(a)$. If $x\in A$ then $x\in O$.
\end{claim}
\begin{proof}[Proof of claim]
    Let $x\in A$. By definition, $V[G]\models \varphi(x,a)$. Since $x\in C(a)$ there is $p\in \rho_e``G$ a projection $\sigma\colon (\mathbb{B}_e/\rho_{e_0}``G)_{\downarrow p}\rightarrow {\mathbb{B}_d}_{\downarrow\sigma(p)}$ in $V[\rho_{e_0}``G]$ with $\sigma\in R(a)$ and a $(\mathbb{B}_e/\rho_{e_0}``G)_{\downarrow p}$-generic filter $h\in V[G]$  with $\sigma``h=\ro(\mathcal{F}_x)$. 

    \smallskip

    Since $h\in \sky_{\mathcal{B}}(G)$ the \emph{Constellation Lemma} in page~\pageref{lemma: computing correctly the power set} yields a $\mathbb{B}$-generic $G^*$ such that $h\in \con_{\mathcal{B}}(G^*)$ and 
 
    $\mathcal{P}(\kappa)^{V[G^*]}=\mathcal{P}(\kappa)^{V[G]}.$
    Thus by absoluteness of $\varphi(\cdot,\cdot)$ we have $V[G^*]\models \varphi(x,a)$.  Let $t\in G^*$ be such that \begin{equation}\label{equation: t forces varphi}
        V[h]\models \text{$``t\forces_{\mathbb{B}/h}\varphi(\check{x},\check{a})$''}
    \end{equation}

    Since $\sigma``h=\ro(\mathcal{F}_x)$ it follows that $\tau_{\sigma``h}$ is the $\mathbb{P}(\Vec{\mathcal{U}})$-generic sequence induced by $\mathcal{F}_x$. 
    Thus, the above is equivalent to saying
     $$V[h]\models \text{$``t\forces_{\mathbb{B}/h}\varphi(\check{\tau}_{\sigma``h},\check{a})$''}.$$

     So $V[h]\models \Phi(t,\mathbb{B}, h, \sigma``h, a)$.

     \smallskip

     Let $r\in h\s (\mathbb{B}_e/\rho_{e_0}``G)_{\downarrow p}$ be such that \begin{equation}\label{equation: r forces Phi}
V[\rho_{e_0}``G]\models\text{$``r\forces_{(\mathbb{B}_e/\rho_{e_0}``G)_{\downarrow p}}\Phi(\check{t},\check{\mathbb{B}}, \dot{g}, \check{\sigma}``\dot{g},\check{a})$''}.
     \end{equation} 
     If we show that there is an extension  $r'\leq r$ such that $\rho_e(v)=r'\in P(a,\sigma)\cap h$ (for some $v\in\mathbb{B}$) we will be done in our task of showing that $x\in O$. Indeed, in that case $\sigma(\rho_e(v))\in \ro(\mathcal{F}_x)$ and then, by definition,  $x\in N_{\sigma(\rho_e(v))}\s O$.

     \begin{subclaim}
         There is $\rho_e(v)\in h\cap P(a,\sigma)$ such that $\rho_e(v)\leq r$.
     \end{subclaim}
     \begin{proof}[Proof of subclaim]
         As $t\in G^*\subseteq\mathbb{B}/h$, $\rho_e(t)\in h$. So we may pick a condition $r'\in h$ witnessing the compatibility between $r$ and $\rho_e(t)$. Since $E:=(\rho_e``(\mathbb{B}_{\downarrow t}))_{\downarrow r'}$ is dense below $r'$ and $r'\in h$ there is $s\in E\cap h$. Let $v\in\mathbb{B}_{\downarrow t}$ be such that $\rho_e(v)=s\leq r'$. Thus $\rho_e(v)\in h$ and so $v\in(\mathbb{B}/h)_{\downarrow t}$. 
         
         Since $v\leq t$ equation~\eqref{equation: t forces varphi} above ensures that $$V[h]\models \text{$``v\forces_{\mathbb{B}/h^{}}\varphi(\check{x},\check{a})$''}.$$ Similarly, since $\rho_e(v)\leq r'\leq r$, equation~\eqref{equation: r forces Phi} ensures that $$V[\rho_{e_0}``G]\models\text{$``\rho_e(v)\forces_{(\mathbb{B}_e/\rho_{e_0}``G)_{\downarrow p}}\Phi(\check{t},\check{\mathbb{B}}, \dot{g}, \check{\sigma}``\dot{g},\check{a})$''}.$$ But now note that $$V[\rho_{e_0}``G]\models\text{$``\rho_e(v)\forces_{(\mathbb{B}_e/\rho_{e_0}``G)_{\downarrow p}}\big(\Phi(\check{t},\check{\mathbb{B}}, \dot{g}, \check{\sigma}``\dot{g},\check{a})\rightarrow\Phi(\check{v},\check{\mathbb{B}}, \dot{g}, \check{\sigma}``\dot{g},\check{a})\big)$''},$$ simply because $v\leq t$. Therefore, $$V[\rho_{e_0}``G]\models\text{$``\rho_e(v)\forces_{(\mathbb{B}_e/\rho_{e_0}``G)_{\downarrow p}}\Phi(\check{v},\check{\mathbb{B}}, \dot{g}, \check{\sigma}``\dot{g},\check{a})$''}.$$ 
     This is equivalent to saying that $\rho_e(v)\in P(a,\sigma)$ (see $(\beth)$ in page~\pageref{item: relevant conditions}).
     \end{proof}
     The above argument disposes with the proof of Claim~\ref{claim: A included in O}.

\end{proof}

Let us now prove the second part of our assertion:

\begin{claim}
    Let $x\in C(a)$. If $x\in O$ then $x\in A$.
\end{claim}
\begin{proof}[Proof of claim]
     Suppose $x\in O$. Then there are $\sigma\in R(a)$ and $q\in P(a,\sigma)$ such that $x\in N_{\sigma(q)}$. That is to say, there is a projection $$\sigma\colon (\mathbb{B}_e/\rho_{e_0}``G)_{\downarrow p}\rightarrow{(\mathbb{B}_d)}_{\downarrow\sigma(p)}$$ in $V[\rho_{e_0}``G]$ and a condition $q\in P(a,\sigma)$ such that $\sigma(q)\in \ro(\mathcal{F}_x).$ 
     
     After inspecting the definition of the set of relevant conditions $P(a,\sigma)$ one realizes that $q\leq p$ (recall that $p$ is determined by $\sigma$). Thus, trivially,
     $$\sigma\colon (\mathbb{B}_e/\rho_{e_0}``G)_{\downarrow q}\rightarrow({\mathbb{B}_d})_{\downarrow\sigma(q)}$$ 
     is a projection.  This puts us in an ideal position to invoke the \emph{Capturing clause} of the \emph{Interpolation Lemma} (see p.~\pageref{lemma: interpolation}). Namely, we apply the lemma with respect to the following parameters: As projections,
       \begin{displaymath}
       \begin{tikzcd}
 \mathbb{B} \arrow{r}{\rho_e} \arrow[bend right]{rr}{\rho_d} & \mathbb{B}_e \arrow{r}{\rho_{e,d}}  & \mathbb{B}_d
\end{tikzcd}
\begin{tikzcd}
   (\mathbb{B}_e/\rho_{e_0}``G)_{\downarrow q}\arrow{r}{\sigma} &({\mathbb{B}_d})_{\downarrow\sigma(q)};
\end{tikzcd}
\end{displaymath}
as a condition in $\mathbb{B}$ we take $v$ the $\rho_e$-preimage of $q$ (which exists because $q\in P(a,\sigma)$ and  as a $\mathbb{B}_d$-generic $\mathcal{B}(\mathcal{F}_x)$ (note that $\rho_d(v)=\rho_{e,d}(q)\in \mathcal{B}(\mathcal{F}_x)$).

  The lemma yields a $(\mathbb{B}_e/\rho_{e_0}``G)_{\downarrow p}$-generic $h$ with $q\in h$  such that $$\sigma``h=\ro(\mathcal{F}_x).$$
   
     Unwrapping the definition of $q\in P(a,\sigma)$ we have that 
$$V[\rho_{e_0}``G]\models``q\forces_{(\mathbb{B}_e/ \rho_{e_0}``G)_{\downarrow p}}\Phi(\check{v},\check{\mathbb{B}}, \dot{g},\check{\sigma}``\dot{g},\check{a})\text{''}.$$ and since $q\in h$ and $h$ is $(\mathbb{B}_e/\rho_{e_0}``G)_{\downarrow p}$-generic,
     $$V[h]\models \Phi(v,\mathbb{B}, h,\sigma``h,a).$$
     But by our comments derived from the Interpolation Lemma, 
       $$V[h]\models \Phi(v,\mathbb{B}, h,\ro(\mathcal{F}_x),a),$$
       which literally means 
         $V[h]\models v\forces_{\mathbb{B}/h}\varphi(\check{x},\check{a}).\footnote{Here we used that $\tau_{\mathcal{B}(\mathcal{F}_x)}$ (i.e., the $\mathbb{P}_d$-generic sequence induced by $\mathcal{B}(\mathcal{F}_x)$) is $x$. }$

     Once again there is a $\mathbb{B}$-generic $G^*$ that constellates $h$ (i.e., $h\in \con_{\mathcal{B}}(G^*)$), captures all subsets of $\kappa$ in $V[G]$ (i.e., $\mathcal{P}(\kappa)^{V[G^*]}=\mathcal{P}(\kappa)^{V[G]}$) and $v\in G^*$. 

     Therefore, $V[G^*]\models \varphi(x,a)$
     and by absoluteness, $V[G]\models \varphi(x,a)$. This amounts to saying  $x\in A$, as needed.
    
\end{proof}
  The above claims ensure that $A\Delta O\subseteq C(\Sigma)\setminus C$. Therefore $A\Delta O$ is $\kappa$-meager since so is $C(\Sigma)\setminus C$. This proves that $A$ has the $\Vec{\mathcal{U}}$-\textsf{BP}.
\end{proof}

\begin{cor}

Assume there is a  a weak $(\Vec{\mathcal{U}},\lambda)$-nice system. 
Then there is a model of $\zf+\dc_\kappa+\text{``All subsets of $C(\Sigma)$ have the $\Vec{\mathcal{U}}$-$\bp$"}.$
\end{cor}

\section{Applications}\label{sec: applications}
In the present section we show that Merimovich's \emph{Supercompact Extender-Based Prikry forcing} \cite{Merimovich} and a new {Diagonal Extender-Based forcing} yield weak  $(\kappa,\lambda)$-nice $\Sigma$-systems. This will enable us to appeal to the abstract machinery developed in the previous sections and deduce the consistency of every set in $\mathcal{P}({}^\omega \kappa)\cap L(V_{\kappa+1})$ have the $\kappa$-$\psp$, as well as the consistency of the $\vec{\mathcal{U}}$-Baire Property of every set in $\mathcal{P}(\prod_{n<\omega}\kappa_n)\cap L(V_{\kappa+1})$. 
\subsection{The consistency of the $\kappa$-$\psp$}\label{sec: merimovich}
This section begins providing a succinct account of \emph{Merimovich Supercompact Extender-Based Prikry forcing} introduced in \cite{MerSuper}. Our presentation will be  different from \cite{MerSuper} and instead will be akin to the account on the Supercompact Extender Based Radin forcing in \cite{SupercompatRadinExtender}. This digression has to do with our need for the \emph{projected forcings} $\mathbb{P}_e$ to be \emph{small} (i.e., of cardinality ${<}\lambda$). 

\smallskip

\begin{setup}
    We assume that $\kappa$ is a ${<}\lambda$-supercompact cardinal with $\lambda>\kappa$ inaccessible. Namely, we assume the existence of an elementary embedding $j:V\rightarrow M$ such that $\crit(j)=\kappa$, $j(\kappa)>\lambda$ and $M^{<\lambda}\subseteq M$.
\end{setup}

\begin{definition}[Domains]
The set of \emph{domains} is $$\mathcal{D}(\kappa,\lambda):=\{d\in [\lambda\setminus \kappa]^{<\lambda}\mid \kappa=\min(d)\}.$$  
\end{definition}

\begin{definition}[Objects]
    For $d\in \mathcal{D}(\kappa,\lambda)$ a $\emph{$d$-object}$ is a map $\nu\colon \dom\nu\rightarrow \kappa$ fulfilling the following requirements:
    \begin{enumerate}
        \item $\kappa\in\dom\nu\s d$ and $|\dom\nu|<\kappa$.
        
        \item $\nu(\alpha)<\nu(\beta)$ for each $\alpha<\beta$ in $\dom\nu$.
    \end{enumerate}
    The set of $d$-objects will be denoted by $\ob(d)$. Given $\nu,\mu\in \ob(d)$ we write $\nu\prec \mu$ whenever $\dom\nu\s\dom\mu$, $|\nu|<\mu(\kappa)$ and $\nu(\alpha)<\mu(\kappa)$ for all $\alpha\in\dom(\nu).$ 
   
\end{definition}
The definition of a $d$-object is motivated by the properties of the \emph{maximal coordinate of $d$} (in symbols, $\mc(d)$) in the $M$-side of $j\colon V\rightarrow M$:
$$\mc(d):=\{\langle j(\alpha), \alpha\rangle\mid \alpha\in d\}.$$
For each  $d\in\mathcal{D}(\kappa,\lambda)$, $\mc(d)$ induces  a $\kappa$-complete ultrafilter over $\ob(d)$:
$$E(d):=\{X\s \ob(d)\mid \mc(d)\in j(X)\}.$$
This in turn yields a sequence 
$E:=\langle E(d)\mid d\in\mathcal{D}(\kappa,\lambda)\rangle$ called \emph{the $\mathcal{D}(\kappa,\lambda)$-extender induced by $j$}. Notice that while $E$ is not an extender in the traditional sense (see e.g., \cite{Kan}) it nevertheless shares with them most of their natural properties -- for instance, $E$ is a directed system of $\kappa$-complete ultrafilters and there is a canonical projection between $E(d)$ and $E(d')$ for $d\s d'$: $$\pi_{d',d}\colon \nu\in\ob(d')\mapsto\nu\restriction d\in\ob(d).$$ 
One  defines the \emph{Supercompact Extender-Based Prikry forcing} $\mathbb{P}_E$ as follows:
\begin{definition}
    A condition in $\mathbb{P}_E$ is a pair $p=\langle f^p, A^p\rangle$ such that:
    \begin{enumerate}
        \item $\dom(f^p)\in\mathcal{D}(\kappa,\lambda)$. 
        \item $f^p\colon \dom(f^p)\rightarrow{}^{{<}\omega}\kappa$ and
         $f^p(\alpha)$  
        is $<$-increasing for $\alpha\in\dom(f^p)$.
        \item $A^p\in E(\dom(f^p))$ and for each $\nu\in A^p$, $\nu(\kappa)>\max(f^p(\alpha))$.
    \end{enumerate}
    Given $p,q\in\mathbb{P}_E$ we write $p\leq^* q$ if $f^p\subseteq f^q$ and $\pi_{\dom(f^p),\dom(f^q)}``A^q\s A^p$.

    For $p\in\mathbb{P}_E$ and $\nu\in A^p$, $p\cat \nu$ is the condition $\langle f^p_{\nu}, A^p_{\nu}\rangle$ defined by:
    \begin{enumerate}
        \item $f^p_\nu$ is the function with $\dom(f^p_{\nu})=\dom(f^p)$ and $$f^p_{\nu}(\alpha):=\begin{cases}
            f^p(\alpha), & \text{if $\alpha\notin\dom\nu$;}\\
            f^p(\alpha)^\smallfrown\langle \nu(\alpha)\rangle, & \text{if $\alpha\in\dom\nu$.}
        \end{cases}$$
        \item $A^p_{\nu}:=\{\mu\in A^p\mid \nu\prec \mu\}.$ 
    \end{enumerate}
    If $p\in\mathbb{P}_E$ and $\vec\nu\in {}^{{<}\omega}(A^p)$ is $\prec$-increasing one defines $p\cat\vec\nu$ recursively as $$p\cat\vec\nu:=(p\cat\vec\nu\restriction|\vec\nu|-1)\cat\langle \nu_{|\vec\nu|-1}\rangle.$$

    The ordering in $\mathbb{P}_E$ is defined as follows: $q\leq p$ if and only if there is a $\prec$-increasing sequence $\vec\nu\in{}^{{<}\omega}(A^p)$ such that $q\leq^* p\cat\vec\nu.$
\end{definition}
\begin{notation}
    Let $e, d$ be domains and $A\in E(e)$. We shall denote $$A\restriction d:=\{\nu\restriction d\mid \nu\in A\}.$$
     Since the ``extender'' $E$ will be fixed hereafter we shall write $\mathbb{P}$ in place of the more informative notation $\mathbb{P}_E$. Also, following Merimovich, we denote by $\mathbb{P}^*$ the poset with universe $\{f^p\mid p\in\mathbb{P}\}$ endowed with the natural order.
\end{notation}

\begin{lemma}\label{lemma: PE SigmaPrikry}
 $\mathbb{P}$ is a $\Sigma$-Prikry poset taking $\Sigma:=\langle \kappa\mid n<\omega\rangle$.
\end{lemma}
\begin{proof}[Proof sketch]
This follows from Merimovich's work \cite{SupercompatRadinExtender}, so we just sketch the proof for the reader's benefit. Let us go over the clauses of Definition~\ref{SigmaPrikry}: (1) The \emph{length function} is $\ell(\langle f,A\rangle):=|f(\kappa)|$; (2) $(\mathbb{P}_n,\leq^*)$ is clearly $\kappa$-directed-closed; (3) The weakest extensions of a condition $\langle f,A\rangle$ are $\{\langle f,A\rangle\cat\vec\nu\mid \vec\nu\in\prod_{n\leq \ell}A_n\,\wedge\,\vec\nu\,\text{is $\prec$-increasing} \}$; (4) the previous set has size ${<}\lambda$ by inaccessibility of the latter; (5) this is routine; (6) this can be proved arguing as in \cite[\S.4.2]{SupercompatRadinExtender} replacing ``dense open'' by ``$0$-open''.
\end{proof}
\begin{lemma}[Cardinal structure, {\cite{SupercompatRadinExtender}}]\label{lemma: cardinal structure in PE}\hfill
\begin{enumerate}
    \item $\mathbb{P}$ is $\lambda^+$-cc and  preserves both $\kappa$ and $\lambda$.
    \item $\one\forces_{\mathbb{P}}``(\kappa^+)^{V[\dot{G}]}=\lambda\,\wedge\,\cf(\kappa)^{V[\dot{G}]}=\omega$''.\qed
\end{enumerate}
\end{lemma}

After having checked that $\mathbb{P}$ is $\Sigma$-Prikry we change gears and check that there is a natural $(\kappa,\lambda$)-nice $\Sigma$-system associated to $\mathbb{P}.$

\begin{definition}
   For  $d\in  \mathcal{D}(\kappa,\lambda)$, $\mathbb{P}_d$ is the subposet of $\mathbb{P}$ whose universe is $$\{\langle f, T\rangle\in \mathbb{P}\mid \dom(f)\s d\}.$$
\end{definition}
The next lemma is key to prove the existence of weak projections:
\begin{lemma}
    Let $e, d\in \mathcal{D}$ and  $A\in E(e)$. Then 
    $$A(d):=\{\eta\in A\mid \forall \nu\in A\restriction d\, (\nu\prec \eta\restriction d \rightarrow \exists\tau\in A\, (\tau\restriction d=\nu\,\wedge\, \tau\prec \eta)\}\in E(e).$$
\end{lemma} 
\begin{proof}
    The claim is equivalent to $\mc(e)\in j(A(d)).$ First, $\mc(e)\in j(A)$ because $A\in E(e)$. Next, let $\nu\in j(A)\restriction j(d)$ be such that $\nu\prec \mc(e)\restriction j(d)$.

    \begin{claim}
        $\nu=j(\bar{\nu})$ for some $\bar\nu\in A\restriction d$.
    \end{claim}
    \begin{proof}[Proof of claim]
        Since $\nu\prec \mc(e)\restriction j(d)$ that means that: 
        \begin{itemize}
            \item  $\dom(\nu)\s \dom(\mc(e)\restriction j(d))=j``(e\cap d)$.
            \item $\nu(\alpha)<\mc(e)(j(\kappa))=\kappa$ for all $\alpha\in\dom(\nu).$
            \item $|\nu|<\mc(e)(j(\kappa))=\kappa.$
        \end{itemize}
       Combining the above items, there is $\dom(\bar{\nu})\in [e\cap d]^{<\kappa}$ such that $$\dom(\nu)=j``\dom(\bar{\nu})=j(\dom(\bar{\nu})).$$
     Define $\bar\nu$ a map with domain $\dom(\bar\nu)$ such that $\bar\nu(\beta):=\nu(j(\beta))$. Note that $j(\bar\nu)(j(\beta))=j(\bar\nu(\beta))=j(\nu(j(\beta)))=\nu(j(\beta))$. This shows that $j(\bar\nu)=\nu\in j(A\restriction d)$, and therefore $\bar{\nu}\in A\restriction d$. 
    \end{proof} 
    Let $\bar{\tau}\in A$ be such that $\bar{\tau}\restriction d=\bar{\nu}$. Let $\tau=j(\bar{\tau})$. We have $\tau\restriction j(d)=\nu$ and we also claim that $\tau\prec \mc(e)$. Indeed, 
    \begin{itemize}
        \item $\dom(j(\bar{\tau}))=j``\dom(\bar{\tau})\s j``e=\dom(\mc(e))$.\footnote{This is  because $|\dom(\bar{\tau})|<\kappa$ in that $\bar{\tau}$ is an $e$-object.}
        \item $j(\bar{\tau})(j(\alpha))=\bar{\tau}(\alpha)<\kappa=\mc(e)(j(\kappa))$ for all $\alpha\in \dom(\bar{\tau})$.
        \item $|j(\bar{\tau})|=|\bar{\tau}|<\kappa$, once again because $\bar{\tau}\in \ob(e).$
    \end{itemize}
    The above completes the verification.
\end{proof}
 
Fix a domain $d\in \mathcal{D}$ and let $A\in E(e)$ for some other domain $e$. Define, $A^{(0)}(d):=A$ and $A^{(n+1)}(d):=(A^{(n)}(d))(d)$. By the above claim, $\langle A^{(n)}(d)\mid n<\omega\rangle\s E(e)$ so $A^{(\infty)}(d):=\bigcap_{n<\omega}A^{(n)}(d)\in E(e).$

\begin{lemma}\label{lemma: definable Polish spaces}
 There is a directed system of weak projections
    $$\mathcal{P}=\langle \pi_{e,d}\colon \mathbb{P}_e\rightarrow \mathbb{P}_d\mid e,d\in\mathcal{D}(\kappa,\lambda)\cup\{\infty\}\, \wedge\, d\s e\rangle$$
    given by $\langle f, A\rangle\mapsto \langle f\restriction d, A\restriction d\rangle$. Also,   $|\mathrm{tcl}(\mathbb{P}_d)|<\lambda$ for $d\neq \infty.$
\end{lemma}
\begin{proof}
    It is routine to check that 
    the map is order-preserving.  Let us now check Clause~(2) of being a weak projection. Let $p=\langle f,A\rangle$ and $\leq^*$-extend it to $p^*=\langle f,A^{\infty}(d)\rangle.$ Suppose $q\leq p^*\restriction d$. That means that there is a $\prec$-increasing sequence $\vec\nu\in (A^{(\infty)}(d)\restriction d)^{|n+1|}$ such that $q\leq^* (p^*\restriction d)\cat\vec\nu$. By definition there is $\langle \eta_0,\dots, \eta_n\rangle\in (A^{\infty}(d))^{|d+1|}$ (non-necessarily $\prec$-increasing) such that $\eta_i=\nu_i\restriction d$. The first goal is to find a $\prec$-increasing sequence $\langle \tau_0,\dots, \tau_n\rangle\in A^{|d|}$ with the same property.

    Since $\eta_{n-1},\eta_n\in A^{(\infty)}(d)$  they belong, respectively, to $A^{(n)}(d)$ and $A^{(n+1)}(d)$. By our choice,  $\eta_{n-1}\restriction d\in(A^{(n)}(d))\restriction d$ and also $$\eta_{n-1}\restriction d=\nu_{n-1}\prec \nu_n=\eta_{n}\restriction d.$$ The definition of $A^{(n+1)}(d)$ guarantees the existence of an object $\tau\in A^{(n)}(d)$ such that $\tau\prec \eta_n$ and $\tau\restriction d=\eta_{n-1}\restriction d$. Set $\tau_n:=\eta$ and $\tau_{n-1}:=\tau$.

    Repeat the same procedure as before but this time with $\eta_{n-2}$ and $\tau_{n-1}$ thus obtaining $\tau'\prec \tau_{n-1}$ such that $\tau'\restriction d=\eta_{n-2}\restriction d$. Set $\tau_{n-2}:=\tau'$. Proceeding this way we obtain a $\prec$-increasing sequence $\langle \tau_0,\dots, \tau_n\rangle\in A^{|d|}$.

    \smallskip

    Note that since $\langle \tau_0,\dots, \tau_{n}\rangle\in A^{|d|}$ is $\prec$-increasing it is ``addable'' to $p$; thus, $p\cat\langle \tau_0,\dots\tau_n\rangle$ is a well-defined condition. Said this define
    $$p':=\langle f^q\setminus \dom(f)\cup f\cat\langle \tau_0,\dots,\tau_n\rangle, C \rangle$$
    where $C$ is the intersection of the pullbacks of $A_{\langle \tau_0,\dots, \tau_n\rangle}$ and $A^q$ under the projections, $\pi_{\dom(f)\cup \dom(f^q), \dom(f)}$ and $\pi_{\dom(f)\cup \dom(f^q), \dom(f^q)}.$

    A moment of reflection makes clear that $p'\leq p$ and that $p'\restriction d\leq^* q$.
\end{proof}

\begin{lemma}[Capturing]\label{lemma: capturing subsets}
    Let $G$ a $\mathbb{P}$-generic filter. For each $a\in\mathcal{P}(\kappa)^{V[G]}$ there is a domain $d\in \mathcal{D}(\kappa,\lambda)$ such that $a\in \mathcal{P}(\kappa)^{V[g_d]}$ where $g_d:=\pi_d``G.$\footnote{This is sufficient to establish the \emph{$\kappa$-capturing} property displayed in Definition~\ref{def: nice system}: Let $d\in\mathcal{D}^*$ be as in the lemma. Given $d_0\in\mathcal{D}^*$, $d_0\sle e:=d\cup d_0$ and $x\in \mathcal{P}(\kappa)^{V[g_{e}]}$.}
\end{lemma}
\begin{proof}
  
  Following Merimovich \cite{MerSuper} a pair $\langle \mathcal{M}, f^*\rangle$ will be called \emph{good} if:
  \begin{enumerate}
        \item $\mathcal{M}$ is the union of an  $\s$-increasing sequence   $\langle \mathcal{M}_\xi\mid \xi<\kappa\rangle$ with:
         $$  \text{$\mathcal{M}_\xi\prec H_\chi$, $|\mathcal{M}_\xi|<\lambda$  and  $\mathcal{M}_{\xi}^{<\kappa}\s \mathcal{M}_{\xi}.$}$$
        \item $f^*=\bigcup_{\xi<\kappa} f_\xi$ for a $\leq_{\mathbb{P}^*}$-decreasing sequence $\langle f_\xi\mid \xi<\kappa\rangle$ such that
       $$ \text{$\textstyle f_\xi\in \bigcap\{D\in\mathcal{M}_\xi\mid \text{$D$ is dense open for $\mathbb{P}^*$}\}$}$$ 
    and
    $f_\xi\cup\{f_\xi\}\s \mathcal{M}_{\xi+1}$ for all $\xi<\kappa.$
  \end{enumerate}
Let $\dot{a}$ be a $\mathbb{P}$-name with $a=\dot{a}_G$.  The next can be proved as  \cite[Claim~3.29]{Prikryonextenders}:
  \begin{claim}The set of $q\in\mathbb{P}$ for which there is a good pair $\langle \mathcal{M}, f^*\rangle$ such that $\{ \dot{a},\mathbb{P},\mathbb{P}^*\}\cup\kappa\s \mathcal{M}$,  $f^q=f^*$ and $q$ is $\langle \mathcal{M}, \mathbb{P}\rangle$-generic, 
    
    is dense.
  \end{claim}
  
  Let $q\in G$ be as above. Recall that $\langle \mathcal{M},\mathbb{P}\rangle$-genericity means
  $$q\forces_{\mathbb{P}}``\forall D\in \check{\mathcal{M}}\,(\text{$D$ is dense open $\rightarrow$ $D\cap \check{\mathcal{M}}\cap \dot{G}\neq \emptyset$)''}.$$
 For each $\alpha<\kappa$ denote
  $$D_\alpha:=\{q\in\mathbb{P}\mid q\parallel_{\mathbb{P}}\check{\alpha}\in\dot{a}\}.$$
  Clearly,  $D_\alpha\in\mathcal{M}$ (because it is definable via parameters in $\mathcal{M}$) and is dense open. Thus, for each $\alpha<\kappa$ there is $q_\alpha\in D_\alpha \cap \mathcal{M}\cap G$ as itself $q$ is in $G$. 
  
  Let us note that $q_\alpha\in g_d$ being $d:=\dom(f^q)$: Indeed, since $q_\alpha\in\mathcal{M}$ there is $\xi<\kappa$ such that $q_\alpha\in \mathcal{M}_{\xi}$ and as a result $\dom(f^{q_\alpha})\in\mathcal{
  M}_{\xi}$ too. Since $$\{h\in\mathbb{P}^*_E\mid \dom(h)\supseteq \dom(f^{q_\alpha})\}\in\mathcal{M}_\xi$$ is dense open $f_{\xi+1}$ enters it and hence $\dom(f^{q_\alpha})\s \dom(f_{\xi+1})\s d$. Thereby we have $\pi_d(q_\alpha)=q_\alpha$. Note that this yields our claim because $\pi_d(q_\alpha)\in g_d$.

  \smallskip

  To complete the proof define $$A:=\{\alpha<\kappa\mid \exists r\in g_d\; (r\forces_{\mathbb{P}}\check{\alpha}\in \dot{a})\}.$$
  \begin{claim}
      $A=\dot{a}_G$.
  \end{claim}
  \begin{proof}[Proof of claim]
      Let $\alpha\in A$ and $r\in g_d$ witnessing it. Let $\hat{r}\in G$ be such that $\pi_d(\hat{r})=r$. By definition of $\pi_d$, $\hat{r}\leq \pi_d(\hat{r})$ so that $\hat{r}\forces_{\mathbb{P}}\check{\alpha}\in\dot{a}$, which ultimately yields $\alpha\in\dot{a}_G$. Conversely, let $r\in G$ with $r\forces_{\mathbb{P}}\check{\alpha}\in\dot{a}$. Since $q_\alpha$ already $\mathbb{P}$-decides $``\check{\alpha}\in\dot{a}$'' and both $r$ and $q_\alpha$ are compatible it must be that $q_\alpha\forces_{\mathbb{P}}\check{\alpha}\in\dot{a}$. Since $q_\alpha\in g_d$  we conclude that $\alpha\in A$.
  \end{proof}

The lemma has thus been proved.
\end{proof}
Putting all the results into the same canopy we conclude: 

\begin{lemma}\label{lemma: PE nice system}
There is a weak $(\kappa,\lambda)$-nice $\Sigma$-system
    $$\mathcal{P}=\langle \mathbb{P}_e, \pi_{e,d}\colon\mathbb{P}_e\rightarrow\mathbb{P}_d\mid e,d\in\mathcal{D}(\kappa,\lambda)\cup\{\infty\}\,\wedge\, d\s e\rangle.$$
\end{lemma}

The next corollary follow from  the previous lemmas and Theorem~\ref{the main construction}: 
\begin{cor}\label{cor: the kappapsp}
    Assume that $\kappa$ is a ${<}\lambda$-supercompact cardinal and $\lambda>\kappa$ is inaccessible. Then there is a model of $\zfc$ where: 
    \begin{enumerate}
        \item  $\kappa$ is a strong limit singular cardinal with $\cf(\kappa)=\omega$;
        \item Every $A\in L(V_{\kappa+1})\cap \mathcal{P}({}^\omega\kappa)$ has the $\kappa$-\psp. In particular, all the $\kappa$-projective subsets of ${}^\omega\kappa$ have the $\kappa$-$\psp$.
    \end{enumerate} 
\end{cor}
\begin{cor}
    Assume that $\kappa$ is a ${<}\lambda$-supercompact cardinal and $\lambda>\kappa$ is inaccessible. Then there is a model of $\zf+\dc_\kappa+\neg\ac$ where: 
    \begin{enumerate}
        \item $\kappa$ is a strong limit singular cardinal with $\cf(\kappa)=\omega$;
        \item for each $\kappa$-Polish $\mathcal{X}$ and each $A\in\mathcal{P}(\mathcal{X})$, $A$ has the $\kappa$-$\psp$.
    \end{enumerate}
\end{cor}
\subsection{The consistency of the $\kappa$-$\psp$ and  $\UBP$}\label{sec: alternative AIM}
In this section we construct a diagonal version  of Merimovich's forcing \S\ref{sec: merimovich}. 
The forcing was designed by Poveda and Thei during a visit of the latter to Harvard in the Spring of 2024.

\begin{setup}
    Suppose that $\Sigma=\langle \kappa_n\mid n<\omega\rangle$ is an increasing sequence  of ${<}\lambda$-supercompact cardinals with $\lambda>\sup_{n<\omega}\kappa_n$ inaccessible (by convention, $\kappa_{-1}:=\aleph_1$). Specifically, for each $n<\omega$, we assume the existence of $j_n\colon V\rightarrow M_n$ with $\crit(j)=\kappa_n$, $j_n(\kappa_n)>\lambda$ and ${}^{<\lambda} M_n\s M_n$. We also let $\lambda\leq \varepsilon\leq \min_{n<\omega}j_n(\kappa_n)$ which will determine the value of the power set of $\kappa:=\sup(\Sigma)$ in the generic extension. In our intended application, $\varepsilon$ will be just $\lambda$.
\end{setup}

\begin{definition}
    Fix $n<\omega$. An \emph{$n$-domain} is a set $d\in [\{\kappa_n\}\cup (\varepsilon\setminus \kappa)]^{<\lambda}$ such that $\kappa_n=\min(d).$ The set of $n$-domains will be denoted by $\mathcal{D}_n(\varepsilon,\lambda)$.
\end{definition}
    Note that $\{\kappa_n\}$ is always an $n$-domain -- \emph{the trivial $n$-domain.}

\smallskip

Given $d\in \mathcal{D}_n(\varepsilon,\lambda)$ one defines its \emph{maximal coordinate} by
$$\mc_n(d):=\{\langle j_n(\alpha),\alpha\rangle\mid \alpha\in d\}.$$
\begin{definition}
    Let $n<\omega$ and $d\in \mathcal{D}_n(\varepsilon,\lambda)$. A \emph{$d$-object} is a function $\nu\colon d\rightarrow \kappa_n$ with the following properties:
    \begin{enumerate}
        \item $\kappa_n\in \dom(\nu)$ and $|\dom(\nu)|<\kappa_n$.
        \item $\nu(\kappa_n)$ is an inaccessible cardinal above $\kappa_{n-1}$.
        \item $\nu(\alpha)<\nu(\beta)$ for all $\alpha<\beta$ in $\dom(\nu)$.
    \end{enumerate}
    The set of $d$-objects will be denoted by $\ob_n(d)$.\footnote{While formally speaking the notion of $d$-object depends on $n$ this latter can be read off from $d$; indeed, $n$ is the unique integer $m$ such that $\kappa_m=\min(d)$. Thus, the reader should not find ambiguities in our terminologies.}
\end{definition}
As usual, $d$-objets attempt to resemble the properties of $\mc_n(d)$ in the ultrapower $M_n$. More formally, $$\text{$M_n\models``\mc_n(d)\in j_n(\ob_n(d))$''}.$$
\begin{definition}
    Let $n<\omega$ and $d\in \mathcal{D}_n(\varepsilon,\lambda)$. Define 
    $$E_n(d):=\{X\s \ob_n(d)\mid \mc_n(d)\in j_n(X)\}.$$
    
\end{definition}
Clearly, $E_n(d)$ is a $\kappa_n$-complete (yet, not necessarily normal) measure on $\ob_n(d)$. Since $\kappa_n\in \dom(\nu)$ for any $\nu\in\ob_n(d)$, it is nevertheless the case that $E_n(d)$ projects onto a normal measure; specifically, it projects to
$$\mathcal{U}_n:=\{X\s \kappa_n\mid \kappa_n\in j_n(X)\}$$
via the map $\textsf{eval}_{\kappa_n}\colon \ob_n(d)\rightarrow\kappa_n$ defined as $\nu\mapsto \nu(\kappa_n).$

\smallskip

Let us next discuss a few more important aspects of the measures $E_n(d)$. Before doing so we have to define the ordering between objects. Please note that our definition yields a transitive ordering between objects.
\begin{conv}
   Let $n\leq m$ and $d_n,d_m$ be $n$ and $m$-domains, respectively. We shall write $d_n\s^\star d_m$ whenever $d_n\setminus\{\kappa_n\}\s d_m$.
\end{conv}
\begin{definition}
    Let $n<m$ and $d_n\in \mathcal{D}_n(\varepsilon,\lambda)$ and $d_m\in \mathcal{D}_m(\varepsilon,\lambda)$ such that $d_n\s^\star d_m$. Given $\nu\in \ob_n(d_n)$ and $\mu\in \ob_m(d_m)$ we write $\nu\prec \mu$ whenever $\dom(\nu)\s^\star \dom(\mu)$, $\nu(\alpha)<\mu(\kappa_m)$ for all $\alpha\in\dom(\nu)$ and $|\nu|<\mu(\kappa_m).$
\end{definition}

\begin{lemma}\label{lemma: prenormal}
     Let $n<m$, $d_n\in \mathcal{D}_n(\varepsilon,\lambda)$ and $d_m\in \mathcal{D}_m(\varepsilon,\lambda)$ be such that $d_n\s^\star d_m$. Suppose that $X\s \ob_n(d_n)$ and that $\nu\in j_m(X)$ is such that $\nu\prec \mc_m(d_m)$. Then there is $\bar{\nu}\in X$ such that $\nu=j_m(\bar\nu)$.
\end{lemma}
\begin{proof}
   Since $\nu\prec \mc_m(d_m)$ it follows that $\dom(\nu)\s^\star j_m``d_m$, $\nu(\alpha)<\kappa_m$ and  $|\nu|<\mc_m(d_m)(j_m(\kappa_m))=\kappa_m$. Hence there is $\dom(\bar\nu)\in[d_m]^{<\kappa_m}$ such that $\dom(\nu)=j_m``\dom(\bar\nu)$ and $\nu(\alpha)=j_m(\xi_\alpha)$ for certain $\xi_\alpha<\kappa_m$. Define $\bar\nu\colon \beta\mapsto \xi_{j_m(\beta)}$. It follows that $j_m(\bar\nu)=\nu$ and by elementarity $\bar\nu\in X$.
\end{proof}

\begin{lemma}[Normality]
    Suppose that $\ell<m<\omega$ are integers and that
    \begin{enumerate}
        \item $\langle d_n\mid \ell\leq n<m\rangle$ is a $\s^\star$-increasing sequence in $\prod_{\ell\leq n<m}\mathcal{D}_n(\varepsilon,\lambda)$,
        \item  $S\s \prod_{\ell<n<m}\ob_n(d_n)$ is a family  of $\prec$-increasing sequences of objects such that for each $\vec\nu\in S$ we are given a  set $X(\vec\nu)\in E_m(d_m)$.
    \end{enumerate}
    Then, $\triangle_{\vec{\nu}\in S}X(\vec\nu):=\{\eta\in\ob_m(d_m)\mid \forall\vec\nu \in S\, (\max_{\prec}\vec\nu\prec \eta\,\rightarrow\, \eta\in X(\vec\nu)\}$ is large with respect to $E_m(d_m).$
\end{lemma}
\begin{proof}
 Let us show that $\mc_m(d_m)\in j_m(\triangle_{\vec\nu\in S}X(\vec\nu)).$ Suppose that $\vec\mu\in j_m(S)$ is a sequence such that $\max_{\prec }\vec\mu\prec \mc_m(d_m)$. Applying the previous lemma to each of the individuals of $\vec\mu$ we conclude that $$\vec\mu=\langle j_m(\nu_\ell),\dots, j_m(\nu_{m-1})\rangle.$$
 Denote $\langle X^*(\vec\mu)\mid \vec\mu\in j_m(S)\rangle=j_m(\langle X(\vec\nu)\mid \vec\nu\in S\rangle)$. By our previous comment it follows that $X^*(\vec\mu)=X^*(j_m(\vec{{\nu}}))=j_m(X)(j_m(\vec{{\nu}}))=j_m(X(\vec\nu))$. Therefore $\mc_m(d_m)\in X^*(\vec\mu)$ in that $X(\vec\nu)\in E_m(d_m).$
\end{proof}
The next technical lemma will be key in establishing the existence of a family of weak projections between our main forcing and its subforcings:

\begin{lemma}\label{lemma: before weak projections of AIM}
    Let $e_n\s^\star e_{n+1}$ be members of $\mathcal{D}_n(\varepsilon,\lambda)$ and $\mathcal{D}_{n+1}(\varepsilon,\lambda)$, respectively. Suppose that $A\in E_n(e_n)$, $B\in E_{n+1}(e_{n+1})$ and  $d_{n}\s^\star d_{n+1}$ are also $n$ and $n+1$-domains, respectively. Then, the following set is $E_{n+1}(e_{n+1})$-large:
    $$B^*:=\{\eta\in B\mid \forall \nu\in A\restriction d_n\, ((\nu\prec \eta\restriction d_{n+1})\rightarrow\exists\tau\in A\, (\tau\restriction d_n=\nu\,\wedge\, \tau\prec \eta))\}.$$
\end{lemma}
\begin{proof}
  We show that $\mc_{n+1}(e_{n+1})\in j_{n+1}(B^*).$ First, $\mc_{n+1}(e_{n+1})\in j_{n+1}(B)$ because $B\in E_{n+1}(e_{n+1}).$ Second, let $\nu\in j_{n+1}(A\restriction d_n)$ such that $\nu\prec \mc_{n+1}(e_{n+1})\restriction j_{n+1}(d_{n+1})$. By the argument of Lemma~\ref{lemma: prenormal}, $\nu=j_{n+1}(\bar{\nu})$ for certain $\bar\nu\in A\restriction d_n$. Thus we deduce the existence of an object $\tau\in A$ such that $\tau\restriction d_{n}=\bar{\nu}$. Notice that $j_{n+1}(\tau\restriction d_{n})=j_{n+1}(\tau)\restriction j_{n+1}(d_{n})$. Since $\tau\in \ob_{n}(e_n)$ is routine to check that $j_{n+1}(\tau)\prec \mc_{n+1}(e_{n+1})$. 
\end{proof}
At this point everything is in place to introduce the main forcing notion:

\begin{definition}[Main forcing]
    A condition in $\mathbb{P}:=\mathbb{P}(\Sigma,\varepsilon,\lambda)$ is a vector
    $$p=\langle f^p_0,\dots, f^p_{\ell(p)-1},\langle f^p_{\ell(p)}, A^p_{\ell(p)}\rangle,\langle f^p_{\ell(p)+1},A^p_{\ell(p)+1}\rangle\dots\rangle$$
    meeting the following requirements:
    \begin{enumerate}
        \item For  each $n<\omega$, $f^p_n\colon \lambda\rightarrow\kappa_n$ is a partial function with $\dom(f^p_n)\in \mathcal{D}_n(\varepsilon,\lambda)$ and  $f^p_n(\kappa_n)$ is an inaccessible cardinal above $\kappa_{n-1}$.
        
        \item For each $n\geq \ell(p)$, $A^p_n\in E_n(\dom(f^p_n))$ and the sequence $$\textstyle \langle \dom(f^p_n)\mid n\geq\ell(p)\rangle\in \prod_{n\geq \ell(p)}\mathcal{D}_n(\varepsilon,\lambda)$$ is $\s^\star$-increasing.
    \end{enumerate}
    Given conditions $p,q\in \mathbb{P}$ the \emph{pure-extension ordering} $p\leq^* q$ is defined in the following fashion; namely, $p\leq^* q$ if and only if
    \begin{itemize}
        \item $\ell(p)=\ell(q)$;
        \item $f^q_n\s f^p_n$ for all $n<\omega$;

        \item $A^p_n\restriction \dom(f^q_n)\s A^q_n$ where 
        $A^p_n\restriction \dom(f^q_n):=\{\nu\restriction \dom(f^q_n)\mid \nu\in A^p_n\}.$
    \end{itemize}
\end{definition}
\begin{definition}[One-point extensions]
    Let $p\in \mathbb{P}$ and $\nu\in A^p_{\ell(p)}$. The \emph{one-point extension of $p$ by $\nu$}, denoted $p\cat\nu$, is declared to be the vector
    $$\langle f^p_0,\dots, f^p_{\ell(p)-1},f^p_{\ell(p)}\oplus \nu, \langle f^p_{\ell(p)+1}, (A^p_{\ell(p)+1})_{\langle\nu\rangle}\rangle,\langle f^p_{\ell(p)+2}, (A^p_{\ell(p)+2})_{\langle\nu\rangle}\rangle\dots\rangle$$
    where we have defined:
    \begin{enumerate}
        \item $f^p_{\ell(p)}\oplus \nu$ is the function with domain $\dom(f^p_{\ell(p)})$ and values
        $$(f^p_{\ell(p)}\oplus \nu)(\alpha):=\begin{cases}
            \nu(\alpha), & \text{if $\alpha\in\dom(\nu)$;}\\
            f^p_{\ell(p)}(\alpha), & \text{otherwise.}
        \end{cases}
        $$
        \item $(A^p_n)_{\langle \nu\rangle}:=\{\eta\in A^p_n\mid \nu\prec \eta\}.$
    \end{enumerate}
    More generally, given a vector of objects $\langle \nu_{\ell(p)},\dots,\nu_{m-1}\rangle\in \prod_{\ell(p)\leq n<m} A^p_n$ one defines $p\cat\vec\nu$ by recursion as follows: $p\cat\varnothing:=p$ and $$p\cat\vec\nu=(p\cat\langle \nu_{\ell(p)},\dots,\nu_{m-2}\rangle)\cat\nu_{m-1}.$$
\end{definition}
\begin{remark}
    Since every $\nu\in\ob_n(d)$ is required to have $\kappa_n$ in its domain the value of $f_n\oplus \nu$ at  $\kappa_n$ is always determined by $\nu(\kappa_n).$
\end{remark}
\begin{fact}
    $p\cat\vec\nu$ is a condition for all $\vec\nu\in \prod_{\ell(p)\leq i<n}A^p_i$ and $n\geq \ell(p).$
\end{fact}

\begin{definition}[The main forcing order]
Given  $p,q\in\mathbb{P}$ we write $p\leq q$ if there is $\vec\nu\in \{\varnothing\}\cup \prod_{\ell(q)\leq n<\ell(p)}A^q_{n}$ $\prec$-increasing such that $p\leq^* q\cat\vec\nu$.
\end{definition}

\begin{fact}
    The ordering $\leq$ is transitive. 
\end{fact}

\begin{lemma}
    $\mathbb{P}$ is $\Sigma$-Prikry taking $\Sigma:=\langle \kappa_n\mid n<\omega\rangle$.
\end{lemma}
\begin{proof}
The proof is very similar to the verification that the \emph{AIM forcing} is a $\Sigma$-Prikry forcing. We only sketch the main points going over the clauses of Definition~\ref{SigmaPrikry}: (1) the length function $\ell\colon \mathbb{P}\rightarrow \omega$ is clear; (2) is easy to show that $\langle \mathbb{P}_n,\leq^*\rangle$ is $\kappa_n$-directed-closed; (3) the weakest extensions of $p$ are of the form $p\cat\vec\nu$; (4) this follows from inaccessibility of $\lambda$; (5)  is a routine verification; (6) the argument is similar to the corresponding verification  for the AIM forcing (see \cite[\S3]{PovOmega} for a full detailed account). 
\end{proof}

\begin{lemma}[Cardinal structure]\label{lemma: cardinal structure in pseudoaim}\hfill
\begin{enumerate}
    \item $\mathbb{P}$ is $\lambda^+$-cc and  preserves both $\kappa$ and $\lambda$.
    \item $\one\forces_{\mathbb{P}}``(\kappa^+)^{V[\dot{G}]}=\lambda\,\wedge\,\cf(\kappa)^{V[\dot{G}]}=\omega$''.
    \item $\one\forces_{\mathbb{P}}``2^\kappa\geq |\varepsilon|$''. In particular, if $\varepsilon\geq (\lambda^+)^V$ the $\sch$ fails at $\kappa$.
\end{enumerate}
\end{lemma}
\begin{proof}
 (1) The argument for the $\lambda^+$-cc is similar to the one given by Gitik in \cite[Lemma~2.15]{Gitik-handbook} using the $\Delta$-system lemma.   
The preservation of $\kappa$ is an immediate consequence of $\langle \mathbb{P}_n,\leq^*\rangle$ being $\kappa_n$-closed and the Prikry Property \cite[Lemma~2.10(1)]{PartI}. The preservation of $\lambda$ is a consequence of the \emph{Complete Prikry Property} and  $|\{p\cat\vec\nu\mid \vec\nu\in\prod_{\ell(p)\leq n<m}A^p_n,\, m<\omega\}|<\lambda$ (see \cite[Lemma~2.10(3)]{PartI}).

\smallskip

 (2) Let $G\s\mathbb{P}$ generic and $\alpha\in (\kappa,\lambda)$ a $V$-regular cardinal. The set $$D_\alpha:=\{p\in\mathbb{P}\mid \forall n\geq \ell(p)\, \dom(f^p_n)\supseteq \alpha\setminus \kappa\}$$
 is dense. Thus, there is $p\in G\cap D_\alpha$. For each $n<\omega$ let $\vec\nu=\langle \nu_{\ell(p)},\dots,\nu_{\ell(p)+n}\rangle$ the unique $\prec$-increasing sequence in $\prod_{\ell(p)\leq i\leq \ell(p)+n}A^p_n$ for which $p\cat\vec\nu$ is in $G$. We claim that $\bigcup_{n<\omega}(\dom(\nu_n)\cap \alpha)$ covers $\alpha$. Indeed, for each $\gamma<\alpha$, $$E_\gamma:=\{q\leq p\mid \forall n\geq \ell(q)\,\forall \nu\in A^q_n\,(\gamma\in \dom(\nu)\,\wedge\,\nu(\gamma)\neq f^p_n(\gamma))\}$$ is dense below $p$. Hence there is $q\in G\cap E_\gamma$. Let $\mu\in A^q_{\ell(q)}$ be such that $q\cat\langle\mu\rangle\in G$. Since $q\leq p$ and $q\cat\langle\mu\rangle\in G$ it must be the case that $\gamma\in\dom(\nu_n)$ for some $n$. In particular $\gamma\in\dom(\nu_n)\cap \alpha$.

\smallskip

 (3) Fix $G\s \mathbb{P}$ generic over $V$. For each $\alpha\in (\kappa,\varepsilon)$ define $F_\alpha\in (\prod_{n<\omega}\kappa_n)^{V[G]}$ as  $F_\alpha(n):=f^p_n(\alpha)$ for some $p\in G$ with $\ell(p)>n$ and $\alpha\in \dom(f^p_n).$ This definition is well-posed thanks to genericity of $G$. We claim that $F_\alpha\notin V$. Indeed, for each $h\in (\prod_{n<\omega}\kappa_n)^V$ the set $$D_{\alpha,h}:=\{p\in \mathbb{P}\mid \forall n\geq \ell(p)\,\forall\nu\in A^p_n\,(\alpha\in\dom(\nu)\,\wedge\, \nu(\alpha)>h(n))\}$$
 is dense. Let $p\in G\cap D_{\alpha,h}$ and $\nu\in A^p_{\ell(p)}$ be the unique object $\nu$ such that $p\cat \langle\nu\rangle\in G$. By definition and our choice of $p$, $$F_\alpha(\ell(p)+1)=\nu(\alpha)>h(\ell(p)+1).$$

 Similarly, $F_\alpha\neq F_\beta$ for indices $\alpha\neq \beta$. because the set $E_{\alpha,\beta}:=\{p\in \mathbb{P}\mid \forall n\geq \ell(p)\,\forall\nu\in A^p_n\,(\alpha,\beta\in\dom(\nu)\,\wedge\, \nu(\alpha)<\nu(\beta))\}$ is dense.
\end{proof}

After having checked that $\mathbb{P}$ is $\Sigma$-Prikry we change gears and show that there is a natural $(\vec{\mathcal{U}},\lambda$)-nice weak $\Sigma$-system associated to $\mathbb{P}.$ 
\begin{definition}
    Let $\mathcal{D}$ denote the collection of all $\s^\star$-increasing sequences $\mathbf{d}=\langle d_n\mid n<\omega\rangle\in \prod_{n<\omega}\mathcal{D}_n(\varepsilon,\lambda)$. Given two sequences $\mathbf{d},\mathbf{e}\in \mathcal{D}$ we write $\mathbf{d}\sle \mathbf{e}$ whenever $d_n\s^\star e_n$ for each $n<\omega.$
\end{definition}
\begin{remark}
    Note that $(\mathcal{D},\sle)$ has a $\sle$-minimal element; namely, $$\mathbf{min}:=\langle \{\kappa_n\}\mid n<\omega\rangle.$$
    This minimal element will correspond to the \emph{Diagonal Prikry forcing} induced by $\langle \mathcal{U}_n\mid n<\omega\rangle$ (see Definition~\ref{def: diagonal prikry}). Recall that these are the normal measures induced by the maps $\textsf{eval}_{\kappa_n}\colon \nu\mapsto \nu(\kappa_n)$ for $n<\omega.$
\end{remark}

\begin{definition}
   For each sequence of domains $\mathbf{d}\in \mathcal{D}$  denote by $\mathbb{P}_\mathbf{d}$  the subposet of $\mathbb{P}$ whose universe is $\{p\in \mathbb{P}\mid \forall n<\omega\; \dom(f^p_n)\s d_n\}.$ Our top forcing $\mathbb{P}_\infty$  will  be  $\mathbb{P}$.
\end{definition}

\begin{lemma}[Weak Projections]\label{lemma: homogeneity of Extender Based} 
\hfill
\begin{enumerate}\label{lemma: definable Polish spaces}
    \item There is a directed system of weak projections
    $$\mathcal{P}=\langle \pi_{\mathbf{e},\mathbf{d}}\colon \mathbb{P}_{\mathbf{e}}\rightarrow \mathbb{P}_{\mathbf{d}}\mid \mathbf{e},\mathbf{d}\in\mathcal{D}\cup\{\infty\}\, \wedge\, \mathbf{d}\sle \mathbf{e}\rangle$$
where $\mathbb{P}_{\mathbf{min}}$ is the diagonal Prikry forcing relative to $\langle \mathcal{U}_n\mid n<\omega\rangle.$
    \item In the case where $\varepsilon=\lambda$ we also have   
   $|\mathrm{tcl}(\mathbb{P}_{\mathbf{d}})|<\lambda$ for all $\mathbf{d}\in \mathcal{D}$. 
   
\end{enumerate}
\end{lemma}
\begin{proof}
(1) Suppose first that $\mathbf{d}\neq\mathbf{min}$. Let the map $\pi_{\mathbf{e},\mathbf{d}}\colon \mathbb{P}_{\mathbf{e}}\rightarrow \mathbb{P}_\mathbf{d}$
 given by $p\mapsto p\restriction\mathbf{d}$ where
    $$p\restriction \mathbf{d}:=\langle f^p_0\restriction d_0,\dots, f^p_{\ell(p)-1}\restriction d_{\ell(p)-1},\langle f^p_{\ell(p)}\restriction d_{\ell(p)}, A^p_{\ell(p)}\restriction d_{\ell(p)}\rangle,\dots\rangle.$$
It  routine to check that $\pi_{\mathbf{e},\mathbf{d}}$ is a well-defined order-preserving map.  
    \begin{claim}
        $\pi_{\mathbf{e},\mathbf{d}}$ witnesses Clause~(3) of Definition~\ref{def: projections}.
    \end{claim}
    \begin{proof}[Proof of claim]
        Let $p\in \mathbb{P}_{\mathbf{e}}$ and say it takes the form
        $$\langle f_0,\dots, f_{\ell-1}, \langle f_\ell, A_\ell\rangle, \langle f_{\ell+1},A_{\ell+1}\rangle, \dots\rangle. $$
We are going to use Lemma~\ref{lemma: before weak projections of AIM} to $\leq^*$-extend $p$ to a condition $p^*\in \mathbb{P}_{\mathbf{e}}$ which witnesses the condition of being a weak projection. Define $\langle A^*_{n}\mid n\geq \ell\rangle$ inductively as follows. First, set $A^*_{\ell}:=A_\ell$. Second, appeal to Lemma~\ref{lemma: before weak projections of AIM} with $\langle A,B\rangle:=\langle A^*_\ell, A_{\ell+1}\rangle$ and the pair of domains $\langle d_\ell,d_{\ell+1}\rangle$ thus obtaining $A^*_{\ell+1}\in E_{\ell+1}(\dom(f_{\ell+1}))$. Inductively, apply the lemma to the pair $\langle A^*_{n}, A_{n+1}\rangle$ thus obtaining $A^*_{n+1}$. In the end we define
$$p^*:=\langle f_0,\dots, f_{\ell-1}, \langle f_\ell, A^*_\ell\rangle, \langle f_{\ell+1},A^*_{\ell+1}\rangle, \dots\rangle. $$
Now let  $q\in \mathbb{P}_{\mathbf{d}}$ be such that $q\leq p^*\restriction\mathbf{d}$. By definition  there is a $\prec$-increasing sequence $\vec\nu=\langle \nu_{\ell},\dots,\nu_{\ell(q)-1}\rangle$ such that $q\leq^* (p^*\restriction \mathbf{d})\cat\vec\nu$. Moreover, by definition of the measure one sets of $p^*\restriction \mathbf{d}$ there is a sequence $\langle \eta_{\ell}, \dots, \eta_{\ell(q)-1}\rangle \in \prod_{\ell\leq n<\ell(q)}A^*_i$ (not necessarily $\prec$-increasing) such that $\nu_i=\eta_i\restriction d_i$. Let us now use the defining property of the sets $A^*_i$'s to get an alternative sequence $\langle \tau_{\ell}, \dots, \tau_{\ell(q)-1}\rangle$ that is $\prec$-increasing.

First, set $\tau_{\ell(q)-1}:=\eta_{\ell(q)-1}$. To construct the rest of objects we argue by recursion as follows. Since $\eta_{\ell(q)-2}\in A^*_{\ell(q)-2}$, $\tau_{\ell(q)-1}\in A^*_{\ell(q)-1}$ and their respective projections to $d_{\ell(q)-2}$ and $d_{\ell(q)-1}$ do form a  $\prec$-increasing pair, the definition of $A^*_{\ell(q)-1}$ enable us to find $\tau_{\ell(q)-2}\in A^*_{\ell(q)-2}$ such that $$\text{$\tau_{\ell(q)-2}\restriction d_{\ell(q)-2}=\nu_{\ell(q)-2}$ and $\tau_{\ell(q)-2}\prec \tau_{\ell(q)-1}$.}$$
To construct $\tau_{\ell(q)-3}$ argue in the same way yet this time with respect to $\eta_{\ell(q)-3}$ and $\tau_{\ell(q)-2}$. Note that in the end we get $\tau_{\ell(q)-3}\prec \tau_{\ell(q)-2}\prec \tau_{\ell(q)-1}$.

Having constructed $\langle \tau_{\ell},\dots,\tau_{\ell(q)-1}\rangle$ we are in conditions to define $p'\leq p$ such that $\pi_{\mathbf{e},\mathbf{d}}(p')\leq q$. Specifically, define
$$p':=\langle f^q_0\cup f_0,\dots, f^q_{\ell}\cup (f_\ell\oplus \tau_{\ell}), \dots,  \langle f^q_{\ell(q)}\cup (f_{\ell(q)}\setminus \{\kappa_{\ell(q)}\}), B_{\ell(q)}\rangle,\dots \rangle$$
where $\langle B_{n}\mid n\geq \ell(q)\rangle$ is defined as
$$\left\{(\pi_{\dom(f_n)\cup\dom(f^q_n),\dom(f_n)}^{-1} A^*_n)\cap (\pi_{\dom(f_n)\cup\dom(f^q_n),\dom(f^q_n)}^{-1} A^q_n)\right\}_{\langle \tau_{\ell(q)-1\rangle}}.$$
One can check that $p'\leq p$ and that $\pi_{\mathbf{e},\mathbf{d}}(p')\leq^* q.$
    \end{proof}
    In the case were $\mathbf{d}=\mathbf{min}$ the projection is slightly different; namely,
    $$p\restriction \mathbf{min}:=\langle f^p_0(\kappa_0),\dots, f^p_{\ell(p)-1}(\kappa_{\ell(p)-1}),\textsf{eval}_{\kappa_n}``A^p_{\ell(p)},\dots\rangle.$$
    We leave for the interested reader to check  that this is a weak projection. 
    
    It is also routine to check that $\langle \pi_{\mathbf{e},\mathbf{d}}\colon \mathbb{P}_{\mathbf{e}}\rightarrow \mathbb{P}_{\mathbf{d}}\mid \mathbf{e},\mathbf{d}\in\mathcal{D}\cup\{\infty\}\, \wedge\, \mathbf{d}\sle \mathbf{e}\rangle$ is a directed system.
    
\smallskip

    (2) Obvious because $\lambda$ is inaccessible in $V$.
\end{proof}

Let us finalize our verifications showing that the above weak $\Sigma$-system has the capturing property. For the reader convenience, we recall that for a condition $p\in \mathbb{P}$, $W(p)$ denotes the collection of \emph{weak extensions of $p$}; namely, conditions of the form $p\cat\vec\nu$ for $\vec\nu\in \prod_{\ell(p)\leq n<m}A^p_n$ for $m<\omega$.  Given $m<\omega$ we denote by $W_m(p)$ the collection of $q\in W(p)$ such that $\ell(q)=\ell(p)+m$.
\begin{lemma}[Capturing]\label{lemma: capturing subsets}
    Let $G$ a $\mathbb{P}$-generic filter. For each  $x\in\mathcal{P}(\kappa)^{V[G]}$ there is  $\mathbf{d}\in \mathcal{D}$ such that $x\in \mathcal{P}(\kappa)^{V[g_{\mathbf{d}}]}$ where $g_\mathbf{d}:=\pi_{\mathbf{d}}``G.$
\end{lemma}
\begin{proof}
Let $x\s \kappa$ be in $V[G]$. For simplicity of notations assume that the trivial condition $\one_\mathbb{P}$ forces this. For each $n<\omega$, let $\sigma_n$ be a $\mathbb{P}$-name such that $(\sigma_n)_G=x\cap \kappa_n$. Denote $D_n\coloneqq\{p\in \mathbb{P}\mid \exists y\in\mathcal{P}(\kappa)\;(p\forces_{\mathbb{P}} \check{y}=\sigma_n)\}$. Since $\mathbb{P}$ is $\Sigma$-Prikry it does not introduce bounded subsets to $\kappa$ -- 
hence,  $D_n$ is dense and open.  Appealing iteratively to the \emph{Strong Prikry property} we construct a $\leq^*$-decreasing sequence of conditions $\langle p_n\mid n<\omega\rangle$ such that for each $n<\omega$ there is some $m_n<\omega$ such that $\mathbb{P}^{p_n}_{\geq m_n}\s D_n.$ Letting $p$ be a $\leq^*$-lower bound for this sequence we get that, for each $n<\omega$, $\mathbb{P}^p_{m_n}\s D_n$. 

We have actually proved the existence of densely-many conditions $p$ as above, so we may assume that our condition $p$ is in $G$.  For each $n<\omega$, there is a unique  $q_n\in W_{m_n}(p)\cap G$  such that $q_{n}$ decides the value of  $\sigma_n$.  Let $g_{\mathbf{d}}$ be the generic filter induced by $\pi_{\mathbf{d}}$ and $G$ where  $\mathbf{d}:=\langle \dom(f^p_n)\mid n<\omega\rangle$. 

\begin{claim}\label{ChoosingThebranch}
$W(p)\cap G=\{q\in W(p)\mid \pi_{\mathbf{d}}(q)\in g_{\mathbf{d}}\}.$
\end{claim}
\begin{proof}[Proof of claim]

The left-to-right inclusion is obvious. For the other let $q$ in the right-hand-side. 
There is $q'\in G$ such that $\pi_{\mathbf{d}}(q')\leq_{\mathbb{P}_{\mathbf{d}}} \pi_{\mathbf{d}}(q)$. Since $q\in W(p)$ (i.e., $q=p\cat\vec\nu$ for some $\vec\nu$) observe that  $\pi_{\mathbf{d}}(q)=q$.  From this we infer that $q'\leq \pi_{\mathbf{d}}(q')\leq \pi_{\mathbf{d}}(q)=q$, and thus $q\in G$.
\end{proof}

By  virtue of the above claim   $W(p)\cap G\in V[g_{\mathbf{d}}]$. It is routine to check that $x=\{\alpha<\kappa\mid \exists n<\omega\, \exists q\in W(p)\cap G\;(q\forces_{\mathbb{P}} \check{\alpha}\in\sigma_n)\}\in V[g_{\mathbf{d}}]$. 
\end{proof}

\begin{cor}
    There is a weak $(\kappa,\lambda)$-nice $\Sigma$-system
    $$\mathcal{P}=\langle \pi_{\mathbf{e},\mathbf{d}}\colon \mathbb{P}_{\mathbf{e}}\rightarrow \mathbb{P}_{\mathbf{d}}\mid \mathbf{e},\mathbf{d}\in\mathcal{D}\cup\{\infty\}\, \wedge\, \mathbf{d}\sle \mathbf{e}\rangle$$
    with $\mathbb{P}_{\mathbf{min}}$ being the diagonal Prikry forcing $\mathbb{P}(\langle \mathcal{U}_n\mid n<\omega\rangle)$.
\end{cor}

\begin{cor}\label{cor: the kappapsp}
    Assume that $\Sigma=\langle \kappa_n\mid n<\omega\rangle$ is an increasing sequence of  ${<}\lambda$-supercompact cardinals and $\lambda>\kappa:=\sup(\Sigma)$ is inaccessible. Then there is a model of $\zfc$ where: 
    \begin{enumerate}
        \item  $\kappa$ is a strong limit singular cardinal with $\cf(\kappa)=\omega$;
        \item Every $A\in L(V_{\kappa+1})\cap \mathcal{P}({}^\omega\kappa)$ has the $\kappa$-\psp.
          \item Every $A\in L(V_{\kappa+1})\cap \mathcal{P}(C(\Sigma))$ has the $\UBP$.
    \end{enumerate} 
\end{cor}

\subsection*{Acknowledgements}
 Dimonte was supported by the Italian PRIN 2022 Grant Models, Sets and Classifications and the Italian PRIN 2017 Grant Mathematical Logic: models, sets, computability. Poveda acknowledges support from the Harvard CMSA and the Harvard Department of Mathematics. Thei  was supported by the Italian PRIN 2022 Grant ``Models, sets and classifications'' and by the European Union - Next Generation EU.

\bibliographystyle{alpha} 
\bibliography{citations}
\end{document}